\documentclass[11pt]{amsart}

\usepackage[margin=2.5cm]{geometry}
\usepackage{color}
\usepackage{mathrsfs}
\usepackage{mathtools}
\usepackage{amsmath}
\usepackage{amssymb}
\usepackage{enumitem}
\usepackage{bbm}
\usepackage{esint}
\usepackage{nicefrac}
\numberwithin{equation}{section}
\usepackage[colorlinks,citecolor=green,linkcolor=red]{hyperref}

\usepackage[latin1]{inputenc}
\usepackage{tcolorbox}

\newtheorem{theorem}{Theorem}[section]

\newtheorem{lemma}[theorem]{Lemma}
\newtheorem{proposition}[theorem]{Proposition}

\theoremstyle{definition}
\newtheorem{definition}[theorem]{Definition}

\newtheorem{remark}[theorem]{Remark}

\newcommand{\N}{\mathbb{N}}

\newcommand{\R}{\mathbb{R}}

\newcommand{\sfd}{{\sf d}}

\newcommand{\rr}{\mathbb R}
\newcommand{\restr}[1]{\lower3pt\hbox{$|_{#1}$}}

\newcommand{\eps}{\varepsilon}  
\newcommand{\nchi}{{\raise.3ex\hbox{$\chi$}}}

\newcommand{\fr}{\hfill$\blacksquare$}  

\newcommand{\Leb}[1]{{\mathcal L}^{#1}}  
\newcommand{\HH}{\mathcal{H}}

\newcommand{\Lip}{\mathrm{Lip}}
\newcommand{\lip}{\mathrm{lip}}
\newcommand{\esssup}{{\rm ess}\sup}
\newcommand{\essinf}{{\rm ess}\inf}
\newcommand{\diam}{\mathrm{diam}}

\newcommand{\RCD}{{\sf RCD}}
\newcommand{\CD}{{\sf CD}}
\newcommand{\mm}{\mathfrak m}

\renewcommand{\limsup}{\varlimsup}
\renewcommand{\liminf}{\varliminf}
\renewcommand{\d}{{\rm d}}
\newcommand{\X}{{\rm X}}

\newcommand{\Xdm}{(\X,\sfd,\mm)}
\newcommand{\rmCh}{{\rm Ch}}

\newcommand{\supp}{{\rm supp}}

\newcommand{\cD}{\mathcal{D}}

\newcommand{\Per}{{\rm Per}}

\newcommand{\dD}{{\mbox{\boldmath$D$}}}

\newcommand{\cN}{\mathcal{N}}

\renewcommand{\phi}{\varphi}

\newcommand{\Omegastar}{{\Omega^*}}
\renewcommand{\cD}{\mathcal D}

\newcommand{\Ig}{\mathcal I^{\flat}_g}

\newcommand{\mres}{\mathbin{\vrule height 1.6ex depth 0pt width 0.13ex\vrule height 0.13ex depth 0pt width 1.3ex}}

\title[]{Fine P\'olya-Szeg\H{o} rearrangement inequalities \\in metric spaces and applications}

\author[]{Francesco Nobili} 
\address{Universit\'a di Pisa, Dipartimento di Matematica, Largo Bruno Pontecorvo 5,
56127 Pisa, Italy}
\email{\url{francesco.nobili@dm.unipi.it}}

\author[]{Ivan Yuri Violo}
\address{Universit\'a di Pisa, Dipartimento di Matematica, Largo Bruno Pontecorvo 5,
56127 Pisa, Italy}
\email{\url{ivanyuri.violo@dm.unipi.it}}
\begin{document}
\begin{abstract}
We study fine P\'olya-Szeg\H{o} rearrangement inequalities into weighted intervals for Sobolev functions and functions of bounded variation defined on metric measure spaces supporting an isoperimetric inequality. We then specialize this theory to spaces with synthetic Ricci lower bounds and characterize equality cases under minimal assumptions. 

As applications of our theory, we show new results around geometric and functional inequalities under Ricci lower bounds answering also questions raised in the literature. Finally, we study further settings and deduce a Faber-Krahn theorem on Euclidean spaces with radial log-convex densities, a boosted P\'olya-Szeg\H{o} inequality with asymmetry reminder on weighted convex cones, the rigidity of Sobolev inequalities on Euclidean spaces outside a convex set and a general lower bound for Neumann eigenvalues on open sets in metric spaces.
\end{abstract} 
\maketitle
\thispagestyle{empty}
\allowdisplaybreaks
\setcounter{tocdepth}{2}
\tableofcontents
\section{Introduction}
Monotone rearrangement methods for functions are powerful tools in modern analysis when dealing with variational problems. The first effective application of the concept of decreasing rearrangement is due to Faber \cite{Faber23} and Krahn \cite{Krahn25,Krahn26} in the proof of the celebrated Faber-Krahn inequality conjectured by Lord Rayleigh in the $19^{\rm th}$ century. Starting from these works, P\'olya and Szeg\H{o} in the book \cite{PoliaSzego51} extended this machinery to several other problems in mathematical physics that found exceptional applications among which we mention \cite{Lieb83,AshbaughBenguria1992,HamelNadirashviliRuss2011,BrascoDePhilippisVelichkov2015}. Nowadays, it is well understood that an isoperimetric inequality guarantees that the decreasing rearrangement decreases the overall Dirichlet energy carried by a function and this principle is commonly referred to as the P\'olya-Szeg\H{o} inequality. We refer to the books \cite{LiebLoss1997,Baernstein19,ciao}, the surveys \cite{Talenti2016,Frank2022} and references therein for an account and for  historical notes on the subject.

A big step in the progress of this theory has been made by Brothers and Ziemer  \cite{BrothersZiemer88} who employed sophisticated arguments involving the theory of sets of finite perimeter to understand the case of equality in the P\'olya-Szeg\H{o} inequality for weakly differentiable functions (see \cite{FeroneVolpicelli03} for an alternative proof). This fine investigation culminated in \cite{CianchiFusco02} where this analysis was extended to the case of functions of bounded variation.

In the meantime, the advances around different types of isoperimetric problems made it possible to investigate analogous P\'olya-Szeg\H{o} rearrangement principles in settings possibly different from the standard Euclidean space. For example B\'erard and Meyer \cite{BerardMeier1982} studied the notion of \emph{spherical} rearrangement on manifolds with Ricci {curvature lower bound}, taking advantage of the celebrated L\'evy-Gromov isoperimetric inequality \cite[Appendix C]{Gromov07}. Additionally, rearrangements were considered {on Cartan-Hadamard manifolds \cite{hebey99} and on}  the Euclidean space but with Gaussian weight \cite{Ehrhard1983} with full understanding of equality cases in \cite{CarlenKerce01}, see also previous works \cite{Borell1975,SudakovCirel1974} and \cite{BakryLedoux96,Bobkov1997} around Gaussian isoperimetric inequalities. Finally, we mention the recent works about rearrangements in metric spaces with synthetic Ricci curvature lower bound  by Mondino and Semola \cite{MondinoSemola20} and the authors \cite{NobiliViolo21}.

\smallskip

All these examples (and many more) contributed to affirming the principle that the underlying isoperimetric principle \emph{identifies} the correct notion of rearrangement for which a  P\'olya-Szeg\H{o} inequality can be proved. However, each setting comes with its own features and one is typically requested to run again similar arguments in a new framework, ideally with minimal assumptions and characterization of equality cases. This said, our main goals are:
\begin{itemize}
    \item building a general machinery of fine rearrangement principles in metric spaces, covering as many settings as possible;
    \item specializing this investigation to several settings, most notably (but not limited to) spaces with synthetic Ricci curvature lower bounds, together with new characterizations of the equality cases under minimal assumptions;
\end{itemize}
As applications of our theory, we obtain the following results:
\begin{itemize}
    \item[--] quantitative improvements in terms of the diameter on the $p$-Lichnerowicz inequality, the Sobolev inequality and the log Sobolev inequality in the compact spaces with \textit{positive Ricci curvature lower bound}, see Theorem \ref{thm:main quant compact spaces}.
    \item[--] geometric and functional rigidity of the $p$-Faber Krahn inequality, the $p$-Sobolev inequality  and the $p$-log Sobolev inequality in non-compact spaces with \textit{non-negative Ricci curvature and Euclidean volume growth}, see Theorem \ref{thm:main noncompact AVR}.
    \item[--] $p$-Faber-Krahn inequality in the \textit{Euclidean space weighted with a radial log-convex density}, together with its geometric and functional rigidity, see  Theorem \ref{thm:faber Krahn log convex}.
     \item[--] $p$-Sobolev inequalities in the \textit{Euclidean space outside a convex set} together with geometric and functional rigidities, see  Theorem \ref{thm:main outsideconvex}.
    \item[--] sharp lower bound on the first $p$-Neumann Laplacian eigenvalue on \textit{open subsets of metric spaces supporting a relative isoperimetric inequality}, see Theorem \ref{thm:main neumann}.
\end{itemize}

\subsection{A general P\'olya-Szeg\H{o} rearrangement theory}
As discussed, we aim at covering as many settings as possible, among which -but not limited to- weighted manifolds. We shall therefore consider working in the realm of metric measure spaces $\Xdm$ and with the Sobolev and BV calculus respectively, see Section \ref{sec:calculus}.

Given $\varnothing \neq \Omega\subset \X$ open and a function $u \colon \Omega \to \R$, satisfying  $\mm(\{u>t\})<\infty$ for $t > \essinf u$, we will consider the decreasing rearrangement $u^*$ on the real line with a weighted measure $\omega$. The weights are of the type $\omega=g(t)\d t$, where $g:I\to(0,\infty)$ is a continuous function and $I\subset \R$ is an open interval. We also assume the following natural basic properties:
\[ 
\mm(\Omega)\le \omega(\R),\qquad {
\begin{array}{l}
     \omega((-\infty,x])<\infty,\quad \forall x \in I,  \\
     \lim_{x\to -\infty}\omega((-\infty,x])=0
\end{array}
}
\]
Under these assumptions, it is possible to define the \textit{decreasing  rearrangement} with respect to $\omega$ 
\[
u^* \colon \Omega^* \to \R,
\]
that is a monotone function and equidistributed with  $u$ and where $\Omega^*\subset I$ is a canonical open sub-interval satisfying $\omega(\Omega^*)=\mm(\Omega)$. See Section \ref{sec:monotone rearrangements} for the details.

We also consider the \emph{isoperimetric profile function} $\Ig \colon [0,\omega(I)] \to [0,\infty)$ given by $\Ig(v)=g(r_v),$ where $r_v\in I$  satisfies $\omega((-\infty,r_v))=v$ (see Section \ref{sec:key}).  Note that there is a slight abuse of terminology, since under these general assumptions, half lines are not necessarily isoperimetric  (as it happens for log-concave weights \cite{Bobkov96}, see also \cite{Ros05} for an axiomatization where this is enforced).

Our first main result is that assuming an isoperimetric inequality with respect to the profile $\Ig(\cdot)$ yields a P\'olya-Szeg\H{o} rearrangement inequality. The main novelty is that this holds in arbitrary metric measure spaces.
\begin{theorem}\label{thm:main Sobolev PZ metric}
Let $\Xdm$ be a metric measure space, $p>1$, let $\varnothing \neq \Omega\subset \X$ be open and let $u \in W^{1,p}_{loc}(\Omega)$ be such that $\mm(\{u>t\})<\infty$ for all $t > \essinf u$. Let $\omega=g\d t$ be a weight as above and let $u^*$ be the decreasing rearrangement of $u$ with respect to $\omega$. Suppose that it holds
\begin{equation}\label{eq:isop general}
    \Per(E,\Omega)\ge {\sf C} \cdot \Ig(\mm(E)), \qquad \forall \, E \subset (\supp (u))^\eps\cap \Omega, \quad \mm(E)<\infty,
\end{equation}
for some constants $ {\sf C},\eps>0$. Then, it holds
\begin{equation}\label{eq:PZ abstract}
    \int_\Omega |Du|_p^p\,\d\mm \ge {\sf C}^p\int_\Omegastar |(u^*)'|^p\, \d \omega,
\end{equation}
and moreover, if the left-hand side is finite then $u^* \in {\sf AC}_{loc}(\Omega^*)$.
\end{theorem}
We also study the rigidity case.  Roughly, we show that under mild regularity assumptions on $u$, equality in \eqref{eq:PZ abstract} implies equality in \eqref{eq:isop general} when $E$ is a superlevel set of $u$. Moreover, assuming also some suitable geometric regularity on the space $\X$ and on the isoperimetric inequality \eqref{eq:isop general}, we show that $u$ is radial. See Theorem \ref{thm:Polya W1p} and Theorem \ref{thm:radiality astratta} for the precise statements that we do not report here for brevity. Several applications of these rigidity results in more specific settings will be stated in the upcoming parts of the Introduction. More generally, we believe that this machinery is potentially applicable in other frameworks we do not cover in this work.

Lastly, as in \cite{CianchiFusco02}, we obtain a version of Theorem \ref{thm:main Sobolev PZ metric} for  BV functions, namely that the rearrangement decreases the total variation, the singular part and the jump part, see Theorem \ref{thm:polya BV}. This result will also appear the upcoming parts in more specific settings.
\subsection{The theory under Ricci lower bounds}
We now specialize the discussion to the theory of  spaces with synthetic Ricci curvature lower bounds. Here we assume the reader to be familiar with the theory of ${\sf CD}$ and ${\sf RCD}$-spaces, referring to the surveys \cite{Villani2016,Sturm24_Survey} as well as \cite{AmbICM,Gigli23_working}.

We start with the  setting of (compact) spaces having a positive Ricci curvature lower bound. We will consider $u^*$ to be the decreasing rearrangement with respect to the weight $\mm_{N-1,N}$, for $N>1$, given by
\[
\mm_{N-1,N} \coloneqq  c_N \sin^{N-1}(t)\d t\mres{(0,\pi)},\qquad c_N \coloneqq \Big( \int_0^\pi\sin^{N-1}(t)\, \d t \Big)^{-1}.
\]
We will also denote by ${\sf BBG}_N(D)$, $D \in [0,\pi],$ the diameter-improvement factor on the L\'evy Gromov isoperimetric inequality by Berard-Besson-Gallot \cite{BerardBessonGallot85}, whose value can be found in \eqref{eq:BBG constant}.
\begin{theorem}\label{thm:main polya compact}
Let $\Xdm$ be an  essentially non-branching ${\sf CD}(N-1,N)$ space  with $ N \in (1,\infty)$ and $\mm(\X)=1$. Fix also $p \in (1,\infty)$. We have:
\begin{itemize}
    \item[{\rm i)}] for every $u \in W^{1,p}(\X)$ it holds
    \begin{equation}
    \int |Du|^p\,\d \mm \ge  {\sf BBG}_N(\diam(\X))^p \int_0^\pi|(u^*)'|^p\, \d \mm_{N-1,N};
    \label{eq:main polya compact}
    \end{equation}
    \item[{\rm ii)}] for every $u \in BV(\X)$ it holds
    \begin{align*}
    |\dD u|(\X) &\ge   {\sf BBG}_N(\diam(\X))\int_0^\pi c_N \sin^{N-1}(t) \,\d TV( u^*),\\
     |\dD^s u|(\X) &\ge   {\sf BBG}_N(\diam(\X))\int_0^\pi c_N \sin^{N-1}(t)\,\d TV^s( u^*),\\
     |\dD u|(J_u) &\ge   {\sf BBG}_N(\diam(\X))\int_0^\pi c_N \sin^{N-1}(t) \,\d TV( u^*)\mres{J_{u^*}}.
    \end{align*}
\end{itemize}
Moreover, if  $u \in W^{1,p}(\X)$ satisfies (with both sides non-zero and finite)
\begin{equation}
    \int |Du|^p\,\d \mm =  \int_0^\pi |(u^*)'|^p\, \d \mm_{N-1,N},
\label{eq: equality PZ compact}
\end{equation}
then $\diam(\X) = \pi$. Furthermore, if also $(u^*)'\neq 0$ a.e.\ on $\{\essinf u< u^* < \esssup u\}$ then $u$ is radial, i.e.\ for some $x_0$ it holds
\[
    u = u^* \circ \sfd_{x_0},\qquad \mm\text{-a.e.}.
\]
Finally, if $\X$ is an ${\sf RCD}(N-1,N)$ space then it is a spherical suspension with tip $x_0$.
\end{theorem}
We recall that $|\dD u|$, $|\dD^s u|$ and $J_u$ denote respectively the total variation, its singular part and the jump set of $u,$ while $TV(u^*),TV^s(u^*)$ and $J_{u^*}$ denote the same objects, in the classical sense, for $u^*$ (see Section \ref{sec:calculus}).
The above rearrangement result extends the classical one by \cite{BerardMeier1982} and the recent ones in the non-smooth setting in \cite{MondinoSemola20,NobiliViolo21} in numerous ways: we consider possibly sign-changing functions; include the celebrated Berard-Besson-Gallot sharp lower bound of the isoperimetric profile (as done in \cite{BerardBessonGallot85} in the smooth setting) and include also BV functions. 

Furthermore,  quantitative information on the diameter can be deduced thanks to the bound
\[
{\sf BBG}_N(D)^2-1 \gtrsim  (\pi -D)^N,\qquad \forall \, D \in (0,\pi]
\]
(see \eqref{eq:BBG quantitativo}).
Indeed, combining the above with \eqref{eq:main polya compact} is particularly effective in quantifying almost rigidity statements for sharp functional inequalities, see Theorem \ref{thm:main quant compact spaces} and references below.

The very last conclusion (see Section \ref{sec:Polya Ricci} for the notion of a spherical suspension) was also obtained in \cite{MondinoSemola20} using the almost rigidity in the L\'evy Gromov isoperimetric inequality, while here we use only the rigidity of the maximal diameter \cite{Ketterer13}. Moreover, the radiality of $u$ in ${\sf RCD}$ setting was also obtained in \cite{MondinoSemola20} but assuming that $u\ge 0$, that $u$ is Lipschitz regular and that $|Du|\neq 0$. Instead, here we also cover the ${\sf CD}$ case and we only need to assume $(u^*)'\neq 0$. The latter non-trivial improvement is highly desirable in view of applications (see Proposition \ref{prop:non vanishing} and Remark \ref{rem:non vanishing} for more details) but requires a fine analysis to deal with critical points of general Sobolev and BV functions. In particular, we automatically improve  the rigidity results in Talenti-type comparison principles recently obtained in \cite{mondinovedovato21,Wu24}, by relaxing the assumptions.

\smallskip

We pass to the  setting of (non-compact) spaces with non-negative Ricci lower bound and Euclidean volume growth. In the following $N\in(1,\infty)$ and  $u^*$ is the decreasing rearrangement with respect to the weight
\[
\mm_{0,N} \coloneqq   N\omega_N t^{N-1}\d t\mres{(0,\infty)},
\]
with $\omega_N\coloneqq\frac{\pi^{N/2}}{\Gamma\left(N/2+1\right)}$. If $N\in \N$ then $\omega_N$ is the volume of the unit ball in $\rr^N.$ We refer to Section \ref{sec:Polya Ricci} for the underlying isoperimetric principle and the notion of Euclidean metric cone.
\begin{theorem}\label{thm:main polya noncompact}
 Let $\Xdm$ be a ${\sf CD}(0,N)$ space, with $ N \in (1,\infty)$ and such that
 \[
 {\sf AVR }(\X)\coloneqq \lim_{r\uparrow\infty} \frac{\mm(B_r(x))}{\omega_Nr^N} \in (0,\infty).
 \]
 Fix also $p \in (1,\infty)$. We have:
\begin{itemize}
    \item[{\rm i)}] for every $u \in W^{1,p}_{loc}(\X)$ such that $\mm(\{ u>t \})<\infty$ for all $t >\essinf u$ it holds
    \begin{equation}
        \int |Du|^p\,\d \mm \ge  {\sf AVR}(\X)^{\frac pN} \int_0^{+\infty} |(u^*)'|^p\, \d \mm_{0,N}; \label{eq:polya Sobolev AVR}
    \end{equation}
    \item[{\rm ii)}] for every $u \in BV_{loc}(\X)$ with $\mm(\{ u>t \})<\infty$ for all $t >\essinf u$ it holds
    \begin{align*}
    |\dD u|(\X) &\ge  {\sf AVR}(\X)^{\frac 1N} \int_0^\infty N\omega_N t^{N-1} \,\d TV(u^*),\\
    |\dD^s u|(\X) &\ge  {\sf AVR}(\X)^{\frac 1N} \int_0^\infty N\omega_N t^{N-1} \,\d TV^s(u^*),\\
    |\dD u|(J_u) &\ge  {\sf AVR}(\X)^{\frac 1N}\int_0^\infty N\omega_N t^{N-1} \,\d TV(u^*)\mres{J_{u^*}}.
    \end{align*}
\end{itemize}
Moreover, if $\X$ is {an essentially non-branching ${\sf CD}(0,N)$ space} and equality occurs  in \eqref{eq:polya Sobolev AVR}  with both sides non-zero and finite, then $\X$ is {an $N$-volume cone}. If also $(u^*)'\neq 0$ a.e.\ on $\{\essinf u< u^* < \esssup u\}$ then $u$ is radial, i.e.\ for some tip $x_0\in \X$ it holds
\[
    u = u^* \circ\big( {\sf AVR}(\X)^{\frac 1N} \sfd_{x_0}\big),\qquad  \mm\text{-a.e.}.
\]
{Finally, if $\X$ is ${\sf RCD}(0,N)$ then it is a $N$-Euclidean  cone with tip $x_0$.}
\end{theorem}
For non-negative functions, inequality \eqref{eq:polya Sobolev AVR}  already appeared in the non-smooth setting in \cite{NobiliViolo21,NobiliViolo24} and previously in Riemannian manifolds in \cite{BaloghKristaly21}. The focus here is on the characterization of the equality case. The main novelty, as for Theorem \ref{thm:main polya compact}, is that we obtain the radiality of $u$ assuming only that $(u^*)'\neq 0$, while in \cite{NobiliViolo24} we had to assume $|Du|\neq 0$ and that $u$ is Lipschitz.

A particular class where the above applies (see, e.g., \cite[Section 1.2]{CavallettiManini22}), is that of weighted convex cones possibly with non-branching anisotropic norms, referring to \cite{LionsPacella90,CabreRosOtonSerra16} and references therein (see also\cite{LionsPacellaTricarico1988,FigalliIndrei13,CiraoloFigallIRoncoroni20,CintiGlaudoPratelliRosOtonSerra20} for further developments and applications). In particular, by combining the quantitative isoperimetric principle deduced in \cite{CintiGlaudoPratelliRosOtonSerra20} with asymmetry arguments (see \cite{HansenNadirashvili94,BrascoDephilippis17}) we derive a boosted version of \eqref{eq:polya Sobolev AVR} containing an extra asymmetry reminder, see Theorem \ref{thm:boosted covnex cones}.

\subsection{Applications to geometric and functional inequalities under Ricci lower bounds}
The first result deals with a quantitative geometric stability in the Lichnerowicz $p$-spectral gap, in the optimal Sobolev inequalities and in the log-Sobolev inequality. Here, we shall not discuss the validity and the equality cases in the functional inequalities as these are by now well understood (see \cite{Ketterer15,CavallettiMondino17} and references therein). We also set $p^* \coloneqq \frac{pN}{N-p}$ for $p\in(1,N)$ and denote by $\lambda_{p,N-1,N}$ the first non-trivial (Neumann) $p$-eigenvalue, also called $p$-spectral gap, of the metric measure space $I_{N-1,N}=([0,\pi],|\cdot |, \mm_{N-1,N}).$ 
\begin{theorem}\label{thm:main quant compact spaces}
For all $N\in (1,\infty)$, $p\in(1,\infty)$ and $q\in (2,2^*]$ (if $N>2$) there are constants $C_N>0$, $C_{N,p}>0$ and $C_{N,q}>0$ such that the following holds.
Let $\Xdm$ be an essentially non-branching $\CD(N-1,N)$. Then:
\begin{itemize}
    \item[{\rm i)}] for all $u \in W^{1,p}(\X)$ non-zero with $\int u|u|^{p-2}\,\d\mm =0$ it holds
        \begin{equation}
        \big(\pi- \diam(\X)\big)^{N} \le C_{N,p} \left( \|u\|_{L^p(\X)}^{-p} \|Du\|_{L^p(\X)}^p- \lambda_{p,N-1,N}\right),
        \label{intro:p-gap quantitative}
        \end{equation}
        \item[{\rm ii)}] for all $u \in W^{1,2}(\X)$ non-constant it holds
        \begin{equation}
        \big(\pi- \diam(\X)\big)^N \le C_{N,q} \left(  \frac{q-2}{N} - \frac{\| u\|^2_{L^{q}(\X)} - \mm(\X)^{2/q-1}\|u\|_{L^2(\X)}^2 }{\mm(\X)^{2/q-1} \| Du\|_{L^2(\X)}^2 }\right),
        \label{intro:q-sob quantitative}
        \end{equation}
        
        \item[{\rm iii)}] for  all $u \in W^{1,2}(\X)$ non-constant with    $\int |u|\,\d\mm=1$ it holds
        \begin{equation}
        \big(\pi- \diam(\X)\big)^N \le C_N \left(  \Big(\int |u|\log|u|\,\d\mm\Big)^{-1}\int_{\{|u| >0\}}\frac{|Du|^2}{|u|}\,\d\mm- 2N \right).
        \label{intro:log Sob quantitative}
        \end{equation}
      
\end{itemize}
\end{theorem}
The result i) for $p=2$ was already known from \cite{CavallettiMondinoSemola23}, where a deeper analysis is carried out to understand also the shape of almost eigenfunctions. The result ii) instead quantifies the almost rigidity results we obtained in \cite{NobiliViolo21}. As for the result iii), it addresses quantitatively a rigidity question raised in \cite[(A) in Section 4]{OhtaTakatsu20}  (see also \cite[b) in Remark 6.4]{OhtaMai21}).  We mention also the works \cite{MondinoSemola20,mondinovedovato21,Kristaly22-GAFA,GunesMondino23,DFV24,Wu24,KristalyMondino24} for further applications of rearrangements methods.

\smallskip

Next, we face rigidity properties of sharp functional inequalities under non-negative Ricci curvature and Euclidean volume growth. To do this, we first introduce the relevant notation: we denote by $I_N$ the model metric measure space $([0,\infty),|. |, \mm_{0,N})$, by $\lambda_{p,N,\rho}^\cD$ the Dirichlet $p$-eigenvalue of $[0,\rho)$ in $I_N$ and consequently set $F_{N,p} \coloneqq \lambda_{p,N,\rho_1}^\cD$ where $\rho_1>0$ so that $\mm_{0,N}((0,\rho_1))=1$, namely the eigenvalue relative to the ball of unit volume. We denote by $S_{p,N}, L_{p,N}$ the values of the optimal constant in the $p$-Sobolev inequality and in the $p$-log Sobolev inequality in $I_N$ (see \eqref{eq:Sobolev constant}, \eqref{eq:LogSobolev constant}).
\begin{theorem}\label{thm:main noncompact AVR}
Let $\Xdm$ be {an essentially non-branching ${\sf CD}(0,N)$ space} for some $N\in (1,\infty)$ and with ${\sf AVR}(\X) >0$. {Let $p,p' \in (1,\infty)$ 
 be H\"older conjugate exponents}. Then:
\begin{itemize}
    \item[{\rm i)}] suppose  $\Omega\subset \X$ is open so that $ \mm(\Omega) <\infty$ and that
    \begin{equation}
    F_{N,p}\big(\mm(\Omega)^{-1}{\sf AVR}(\X)\big)^{p/N} \|u\|^p_{L^{p}(\Omega)} =\|Du\|^p_{L^p(\Omega)},\label{intro: p-Eig AVR}
    \end{equation}
    holds for some non-zero $u \in W^{1,p}_0(\Omega)$. 
    
    Then $\Omega$ coincides up to $\mm$-negligible sets with a ball $B_\rho(x_0)$ for some $x_0 \in \X$, $\rho=\left(\frac{{\sf AVR}(\X)\mm(\Omega)}{\omega_N}\right)^{1/N} $  and, up to a sign, $u= \varphi \circ \big({\sf AVR}(\X)^{1/N}  \sfd_{x_0}\big)$ $\mm$-a.e.\ on $B_\rho(x_0)$, where $\varphi\in {\sf AC}_{loc}(0,\rho)$, $\phi\ge 0,$  satisfies
    \begin{equation}
    \left(|\varphi '|^{p-2}\varphi'\right)'+(\varphi ')^{p-1} \left(\log (N\omega_Nt^{N-1})\right)'=-\frac{F_{N,p}}{\mm(\Omega)^{p}}  \varphi ^{p-1}, \quad \text{a.e.\ in }(0,\rho). \label{intro:p-ODE AVR}
    \end{equation}
    \item[{\rm ii)}] suppose $p <N$ and
    \begin{equation}
    \|u\|_{L^{p^*}(\X)} = S_{p,N}{\sf AVR}(\X)^{-\frac 1N}\|Du\|_{L^p(\X)},\label{intro:p-Sob AVR}
    \end{equation}
    holds for some non-zero $u \in W^{1,p}_{loc}(\X)\cap L^{p^*}(\X)$, then there are $a \in \R,b >0,x_0 \in \X$ so that
    \[
    u = a\big(1+b \sfd_{x_0}(\cdot)^{\frac p{p-1}}\big)^{\frac{N-p}{p}},\qquad\mm\text{-a.e.};
    \]
    \item[{\rm iii)}] suppose
    \begin{equation}
    \int |u|^p\log |u|^p\,\d \mm = \frac{N}{p}\log\left( L_{p,N}{\sf AVR}(\X)^{-\frac pN} \|Du\|_{L^p(\X)}^p\right),
    \label{intro:p-LogSov AVR}
    \end{equation}
    holds for some $u \in W^{1,p}(\X)$ with $\int|u|^p\,\d\mm=1$, then there are  $ \lambda\in \R\setminus\{0\},x_0 \in \X$ so that
    \[
    u = \lambda^{\frac N{pp'}} \big(\Gamma( \tfrac {N}{p'} +1) \omega_N\big)^{-\frac 1p}e^{-|\lambda|\frac{ \sfd_{x_0}(\cdot)^{p'}}{p}},\qquad\mm\text{-a.e.}.
    \]
\end{itemize}
Finally, in any of the above, $\X$ is also {$N$-volume cone with tip $x_0$ and if $\X$ is also ${\sf RCD}(0,N)$ then it is in fact an $N$-Euclidean  cone. }
\end{theorem}
We observe that iii) answers affirmatively question \cite[(I) in Section 6]{BaloghKristalyTripaldi23}. 
The geometric rigidity for i) was known from \cite{AntonelliPasqualettoPozzettaSemola22} in the non-collapsed Riemannian sub-class and previously in \cite{BaloghKristaly21} for $p=2$ under additional regularity assumptions on manifolds. Item ii) instead appeared in \cite{NobiliViolo24} for $p=2$. We also mention similar investigations on critical points of \eqref{intro:p-Sob AVR} on manifolds in \cite{CatinoMonticelli22,FogagnoloMalchiodiMazzieri23,FogagnoloMalchiodiMazzieri23_Correction,CatinoMonticelliRoncoroni23}.

In the above statement, we only focused on the characterization of equality cases of the related sharp functional inequality, namely i) concerns the sharp Faber-Krahn inequality previously addressed in \cite{FogagnoloMazzieri22,BaloghKristaly21,AntonelliPasqualettoPozzettaSemola22,ChenLi23}), ii) concerns the sharp Euclidean Sobolev inequality \cite{Aubin76-2,Talenti76} extended in \cite{BaloghKristaly21,NobiliViolo21} (see also \cite{Kristaly23} for a proof via Optimal Transport in the spirit of \cite{C-ENV04}) and iii) concerns the sharp log-Sobolev inequality \cite{DelPinoDolbeault03} (see also \cite{Gentil03,AguehGhoussoubKang04}) extended in \cite{BaloghDonKristaly24,BaloghKristalyTripaldi23}.
\subsection{Applications to geometric and functional inequalities in the Euclidean space and beyond}
The first application is the Faber-Krahn inequality with log-convex radial densities in the Euclidean space. We refer to \cite{RosalesCaneteBayleMorgan08} and \cite{Chambers19} for the related isoperimetric problem and to Section \ref{sec:log-convex conj} for a more detailed presentation. Let us consider $f \colon [0,\infty)\to\R$ smooth and convex and let us associate the radially log-convex density
\[
  g(x) = e^{f(|x|)},\qquad x \in \R^d.
\]
Define the non-negative numbers
\[
R(g)\coloneqq \sup\{ |x|\ \colon \ f(x)=f(0),\, x \in \R^d\}, \qquad V_g\coloneqq \int_{B_{R(g)}(0)} g\,\d\Leb d.
\]
For $\varnothing\neq \Omega\subset \R^d$ open, we denote by $\mm_g(\Omega)\coloneqq \int_\Omega g\,\d \Leb d$, by $H^{1,p}_0(\Omega;g)$ the closure of $C^\infty_c(\Omega)$ with respect to the weighted Sobolev norm $\| u\|_{H^{1,p}_0(\Omega;g)}^p \coloneqq \int_\Omega |\nabla u|^p g\,\d\Leb d + \int_\Omega |u|^p g\,\d \Leb d$ and, correspondingly, by $\lambda_{p}^\cD(\Omega;g)$ the first non-trivial $p$-Dirichlet eigenvalue.
\begin{theorem}\label{thm:faber Krahn log convex}
  Let $d\ge 2$ and let $\varnothing \neq \Omega\subset \R^d$ be open and fix $p \in (1,\infty)$. Consider a radial log-convex density $g$, as above.  Then, it holds
  \begin{equation}
    \lambda_{p}^\cD(\Omega;g)\ge  \lambda_{p}^\cD(B;g),
  \label{eq:faber density conjecture}
  \end{equation}
  where $B$ is the ball centered at the origin so that $\mm_g(\Omega)=\mm_g(B) $. Moreover, suppose that there exist $u \in H^{1,p}_0(\Omega;g)$ so that 
  \begin{equation}
      \int_\Omega|\nabla u|^pg\,\d \Leb d =   \lambda_{p}^\cD(B;g)\int_\Omega |u|^pg\,\d \Leb d.
  \label{eq: equality faber}
  \end{equation}
  Then,  $\Omega$ is up to negligible sets a ball $B_\rho(x_0)$ for some $\rho>0$ and some $x_0 \in B_{R(g)}(0)$ and, up to a sign, $u= \varphi \circ |x-x_0|$ a.e.\ on $B_\rho(x_0)$, where $\varphi\in {\sf AC}_{loc}(0,\rho)$, $\phi\ge 0,$  satisfies
  \begin{equation}
    \left((|\varphi'|^{p-2}\varphi'\right)'+(\varphi')^{p-1} \left(\log\big(e^{f(t)} d\omega_d t^{d-1}\big)\right)'=-\lambda_{p}^\cD(B;g) \varphi ^{p-1}, \quad \text{a.e.\ in $(0,\rho)$};
  \label{eq:ODE log convex}
  \end{equation}
  Finally, if $\mm_g(\Omega)\ge V_g$ then {necessarily} $x_0=0.$
\end{theorem}
For $p=2$, the fact that equality in \eqref{eq:faber density conjecture} implies that $\Omega$ is a ball was recently shown also in \cite{ChenMao2024}. See also \cite{BrandoliniChiacchio23} for related results via decreasing rearrangements.

\smallskip

Next, we present further applications around functional inequalities outside convex sets of the Euclidean space. Isoperimetric inequalities in this setting were first deduced (starting from \cite{ChoeGulliver92,Kim00}) in \cite{ChoeGhomiRitore07} and then extended in \cite{FuscoMorini23} (see also \cite{Choe03} and reference therein and \cite{Krummel17,LiuWangWeng23} for the higher codimension case). Building upon this theory, which we briefly report in Section \ref{sec:outside convex}, we are able to deduce the following rigid Sobolev inequality. Here $\dot H^{1,p}(\Omega)$ denotes the closure of functions in $C^\infty(\Omega)$ with bounded support with respect to the  $L^p$-norm of the gradient. Recall also that $S_{p,d}$ is the best constant in the Euclidean $L^p$-Soboelv inequality (see \eqref{eq:Sobolev constant} for its value).
\begin{theorem}\label{thm:main outsideconvex}
    Let $d\ge 2,p\in(1,d)$ and $C\subset \R^d$ be closed, convex with non-empty interior. Then:
    \begin{equation}
       \|u\|_{L^{p^*}(\R^d\setminus C)} \le 2^{-1/d}S_{p,d}\|\nabla u\|_{L^p(\R^d\setminus C)},\qquad \forall \, u \in \dot H^{1,p}(\R^d\setminus C).\label{eq:Sobolev outside convex}
    \end{equation}
    Moreover, if equality holds for some non-zero function $u \in \dot H^{1,p}(\R^d\setminus C)$, then (up rigid motions) $C=\R^{d-1}\times [-L,0]$ for some $L \in (0,\infty]$ and there are $a\in\R, b>0$ so that 
    \[
    u(x) =  a(1+b |x|^{\frac{p}{p-1}})^{\frac{d-p}{p}},\qquad \text{a.e.\ in $\R^{d-1}\times (0,\infty)$} 
    \]
    and $u=0$ a.e.\ in $\R^{d-1}\times (-\infty,-L).$
\end{theorem}
As for different functional inequalities in this setting, we mention that a related Faber-Krahn inequality was studied in \cite{Keomkyo07} with geometric characterization of equality cases. Besides, the very same arguments to prove the above would establish rigid Euclidean log-Sobolev inequalities in this setting. For brevity reasons, we do not include this result.

\smallskip 

Finally, we present a last application dealing with lower bounds for Neumann eigenvalues. This was studied in \cite{BCT15} in the Euclidean space. Here we propose an alternative argument based on rearrangements with a relative isoperimetric inequality as well as extend the result to metric measure spaces. In what follows we denote by $\lambda_p^\cD(\Omega)$ and $\lambda_p^\cN(\Omega)$ the Dirichlet and Neumann $p$-eigenvalue of an open set $\Omega$ in a metric measure space (see Section \ref{sec:calculus} for the precise definition). Recall also that with $\lambda_{p,N,R}^\cD$ we denote  the Dirichlet $p$-eigenvalue of $[0,R)$ in $I_N=([0,\infty),|. |, \mm_{0,N})$. 
\begin{theorem}\label{thm:main neumann}
    Let $\Xdm$ be a metric measure space and let $\varnothing \neq \Omega\subset\X$ be open with $\mm(\Omega)<\infty$. Suppose there exist  constants $C>0$, $Q>1$ such  that
    \begin{equation}\label{eq:isop relative}
          \Per (E,\Omega)\ge C  (\omega_Q)^\frac1QQ\min (\mm(E),\mm(\Omega \setminus E))^\frac{Q-1}{Q}, \quad \text{for all $E\subset \Omega$ Borel.}
    \end{equation}
    Then, for all $p\in(1,\infty)$ it holds
    \begin{equation}\label{eq:neumann lower bound}
        \lambda_{p}^\cN(\Omega)\ge C^p 2^\frac{p}{Q}\lambda_{p,Q,R}^\cD,
    \end{equation}
    where $R>0$ is such that $\mm_Q([0,R))=\mm(\Omega).$
\end{theorem}
In the case $Q \in \N$  the number $\lambda_{p,Q,R}^\cD$ coincides with $\lambda_p^\cD(B_R(0))$, the Dirichlet $p$-eigenvalue of the ball $B_R(0)$ in $\R^Q$ satisfying $|B_R(0)|=\mm(\Omega).$ In particular the result above in the Euclidean setting coincides with the main result in \cite{BCT15}, which was proved with a different argument. In our proof, the key point is to recast the relative isoperimetric inequality \eqref{eq:isop relative} as an isoperimetric comparison with respect to the interval $I=[0,2R]$ weighted with the log-concave function  $g(t)= \omega_Q Q \min(t^{Q-1},(2R-t)^{Q-1})$ on $I$. Then, Theorem \ref{thm:main Sobolev PZ metric} essentially gives the conclusion.

We finish by discussing the domain of applicability of the above result. If $\Xdm$ is a length PI-space, it is well known that \eqref{eq:isop relative} is satisfied by every ball $B\subset \X$ (see \cite[Remark 4.4]{Ambrosio2001}). Additional examples of sets that satisfy \eqref{eq:isop relative} are uniform domains in any PI-space (\cite[Theorem 4.4]{BjornSh07}). See \cite{Raj20} for a proof that on doubling quasiconvex spaces there are many uniform domains. 

Finally, a condition as \eqref{eq:isop relative} holds on Riemannian manifolds with a smallness condition on the $L^p$-integral defect of the Ricci curvature, see \cite{Gallot88} and \cite{Aubry07} for related diameter estimates (we also refer to \cite{Besson04}). Actually, a version of the L\'evy-Gromov inequality is also known \cite[Proposition 4.2]{SetoWei2017}. Hence, our theory imply P\'olya-Szeg\H{o} inequalities similar to Theorem \ref{thm:main polya compact} with a different constant.

\medskip

\noindent \textbf{Addendum}. {In a first version of this manuscript, both Theorem \ref{thm:main polya noncompact} and Theorem \ref{thm:main noncompact AVR} were deduced under stricter assumptions in the essentially non-branching ${\sf CD}(0,N)$ setting. During the revision process of this manuscript, we became aware of the subsequent work \cite{PasqualettoRajala25}, pointed out to us by the authors and the anonymous referee. This work directly improves on the hypotheses in the aforementioned theorems as we comment in Remark \ref{topological reg avr}.} 

\section{Preliminaries: calculus on metric measure spaces}\label{sec:calculus}
A metric measure space is a triple $\Xdm$ where
\[\begin{split}
(\X,\sfd)\qquad & 	\text{is a complete and separable metric space},\\
\mm\neq 0\qquad & \text{is non negative and boundedly finite Borel measure}.
\end{split}
\]
Unless otherwise stated, we shall always assume $\supp ( \mm) =\X$.  Two metric measure spaces are said \emph{isomorphic}, provided there exists a measure preserving isometry between them. If $\varnothing \neq \Omega \subset \X$ is open, we denote by $L^p(\Omega),L^p_{loc}(\Omega)$ the equivalence class up to $\mm\mres\Omega$-a.e.\ equality of Borel functions in $\Omega$ which are respectively $p$-integrable in $\Omega$ and  $p$-integrable on a neighborhood of each point in $\Omega$. We shall sometimes also write $L^p(\mm)$ to stress the measure we are considering. Analogously, we set $C(\Omega),\Lip(\Omega),\Lip_{bs}(\Omega),\Lip_{loc}(\Omega)$ respectively the set of continuous functions on $\Omega$, the set of Lipschitz functions in $\Omega$, the set of Lipschitz functions having support bounded and contained in $\Omega$ and, lastly, the set of functions which are Lipschitz on a neighborhood of each point in $\Omega$. We denote by $\lip \, u(x)$ the local Lipschitz constant of $u\colon \Omega \to \R$.

We say that $\mm$ is \emph{uniformly locally doubling} provided for every $R>0$ there is $C(R)>0$ so that
\[
\mm(B_{2r}(x))\le C(R)\mm(B_r(x)),\qquad \forall \, r \in (0,R), x \in \X.
\]
To be more concise, we shall only say that $\mm$ is doubling or $\Xdm$ is a doubling metric measure space when the above holds. Also, we call $\X$ a PI-space, if it is doubling and supports a local Poincar\'e inequality, referring to the books \cite{HeinonenKoskelaShanmugalingam55,Bjorn-Bjorn11} and references therein.

Given $E\subset \X$ Borel and $x\in\X$, we define 
\[
\overline{D}(E,x) \coloneqq \limsup_{r\downarrow 0}\frac{\mm(B_r(x)\cap E)}{\mm(B_r(x))},\qquad \underline{D}(E,x) \coloneqq \liminf_{r\downarrow 0}\frac{\mm(B_r(x)\cap E)}{\mm(B_r(x))}.
\]
When the limit exists, the common value is denoted by $D(E,x)$ and is called density of $E$ at $x$. The \emph{essential boundary} of $E$ is given by
\[
\partial^* E \coloneqq \big\{x \in \X \colon \overline{D}(E,x) >0, \overline{D}(E^c,x) >0\big\}\subset \partial E.
\]
When  $\mm$ is doubling we have that $\mm$-a.e.\ $x \in E$ is a density point thanks to the validity of the Lebesgue differentiation theorem (see, e.g., \cite{HeinonenKoskelaShanmugalingam55}). Given $u \colon \Omega\to \bar \R \coloneqq \R \cup \{ \infty\}$ Borel, we say that $u$ is approximately continuous at $x \in \Omega$, provided $u^\wedge (x) = u^\vee(x)\in \R$ where
\begin{align*}
    & u^\wedge(x) \coloneqq {\rm ap }\liminf_{y\to x}u(y) \coloneqq \sup\big\{ t \in \R \colon D(\{u<t\},x)=0 \big\},\\
    & u^\vee(x) \coloneqq {\rm ap }\limsup_{y\to x}u(y) \coloneqq \inf\, \, \big\{ t \in  \R \colon D(\{u>t\},x)=0 \big\},
\end{align*}
with convention that $\sup \varnothing = -\infty$ and $\inf \varnothing = +\infty$. The \emph{precise representative} of $u$ is defined by
\[
\bar u (x) \coloneqq  \frac{u^\wedge(x)  + u^\vee(x)}{2},\qquad x \in\Omega,
\]
with convention that $\infty-\infty=0$. If $u$ is approximately continuous at $x$, then $\bar u(x)$ can be characterized by the number $a \in \R$ so that
\begin{equation}
D( \{ |u(\cdot) -a| >\epsilon\},x) =0,\qquad\forall \, \eps>0.\label{eq:approximate value density}
\end{equation}
From this characterization, it follows that when $\mm$ is doubling and $u \in L^1_{loc}(\Omega)$ then  $u$ is approximately continuous at $\mm$-a.e.\ $x \in \Omega$ and $\bar u(x)$ equals the Lebesgue value. We shall denote 
\[
C_u \coloneqq \{ x \in \Omega \colon u \text{ is approximately continuous at }x \}.
\]
More generally, we always have $ u^\wedge \le u^\vee$ so, it is convenient to define the jump set
\[
J_u \coloneqq \{x \in \Omega \colon u^\wedge(x) < u^\vee(x) \}.
\]
Since $u^\wedge,u^\vee,\bar u$ are Borel functions, then $C_u,J_u$ are Borel sets. By definition, these will always be considered as subsets of $\Omega$ and they do not depend on $u$ up to $\mm\mres\Omega$-negligible sets.

We start recalling basic facts around Sobolev functions in metric spaces introduced in \cite{Cheeger00,Shan00}. Here we follow the axiomatization of \cite{AmbrosioGigliSavare11-3} and refer to \cite{AmbrosioIkonenLucicPasqualetto24,GP20} for the relevant background on the topic. We start with the definition of Sobolev functions.
\begin{definition}
   Let $\Xdm$ be a metric measure space and $u \in L^p(\X)$ for $p\in(1,\infty)$. We define the $p$-Cheeger energy 
   \[
   \rmCh_p(u) \coloneqq \inf \Big\{ \liminf_{n\to\infty} \int (\lip \, u_n) ^p \, \d \mm \colon (u_n) \subset \Lip(\X), u_n\to u \text{ in }L^p(\X)\Big\},
   \]
   set to $+\infty$ if the set is empty. Then, we write $u \in W^{1,p}(\X)$ provided $u \in L^p(\X)$ and $\rmCh_p(u) <\infty$. 
\end{definition}
It can be proven \cite{AmbrosioGigliSavare11-3} that there exists a unique $\mm$-a.e.\ minimal element $|D u|_p \in L^p(\X)$, called $p$-minimal weak upper gradient, satisfying
\[
\rmCh_p(u) = \int| D u|^p_p\,\d\mm.
\]
With a slight abuse of notation, we will always write $\| Du\|_{L^p(\X)}$ for $\rmCh^{1/p}_p(u)$. We shall also need the local Sobolev space. Let $\varnothing \neq \Omega\subset\X$ be open, and $u \in L^p_{loc}(\Omega)$. We say that  $u \in W^{1,p}_{loc}(\Omega)$, provided for every $\eta \in \Lip_{bs}(\Omega)$ it holds $\eta u \in W^{1,p}(\X)$ (extended by zero outside of $\Omega$) and set
\[
|D u|_p \coloneqq |D(\eta u)|_p,\qquad \mm\text{-a.e.\ on }{\{\eta =1\}},
\]
which by the locality property (see \cite{GP20}) is a well-defined object playing the role of the minimal $p$-weak upper gradient. We also set 
\[
W^{1,p}(\Omega)\coloneqq \{u \in W^{1,p}_{loc}(\Omega)  \colon u,|D u|_p \in L^p(\Omega)\}.
\]
We shall need the following fact
\begin{equation}
\begin{array}{l}
     u \in W^{1,p}_{loc}(\Omega),\\
     \int_\Omega |Du|_p^p\,\d\mm<\infty,
\end{array}\implies \exists (u_n)\subset C(\Omega)\cap W^{1,p}_{loc}(\Omega)\,\text{ so that }\, \begin{array}{l}
        u-u_n\to 0\text{ in }L^p(\Omega),  \\
        \int_\Omega |D u_n|_p^p \, \d \mm \to \int_\Omega |Du|_p^p\,\d\mm. 
   \end{array} \label{eq:MeyersSerrin} 
\end{equation}
This can be shown  using the density of continuous functions in $W^{1,p}(\X)$, recently proved in \cite{eriksson2023density} for arbitrary metric measure spaces, and arguing as in the proof of the standard Meyers-Serrin theorem in the Euclidean space.

Next, we introduce the space $W^{1,p}_0(\Omega)\subset W^{1,p}(\X)$ as the $W^{1,p}(\X)$-closure of $\Lip_{bs}(\Omega)$. Given $\Omega\subset \X$ open with $\mm(\Omega)<\infty$, we define its Neumann and Dirichlet $p$-eigenvalue respectively as
\begin{align*}
    &\lambda_{p}^\cN(\Omega) \coloneqq \inf \left\{ \int_\Omega |D u|_p^p\d \mm  \ \colon \  u \in  W^{1,p}(\Omega), \,\, \int_{\Omega} |u|^p\d \mm=1,\, \int_{\Omega} |u|u^{p-2}\d \mm=0  \right\},\\
    &\lambda_{p}^\cD(\Omega) \coloneqq \inf \left\{ \int_\Omega |D u|_p^p\d \mm \   \colon \  u \in  W^{1,p}_0(\Omega), \,\, \int_{\Omega} |u|^p\d \mm=1\right\},
\end{align*}
where by convention the infimum of the empty set is set to $+\infty.$
Moreover if $\Omega=\X$ we simply write $\lambda_p(\X)\coloneqq \lambda_{p}^\cN(\X)$ (note that  $\lambda_{p}^\cD(\X)=0$). 

We discuss now the notion of functions of locally bounded variation following \cite{Miranda03,AmbrosioDiMarino14} (see also \cite{DiMarinoPhD,Martio16-2,NobiliPasqualettoSchultz21} for other equivalent notions).
\begin{definition}
Let $\Xdm$ be a metric measure space, let $\varnothing \neq \Omega\subset\X$ be open and let $u\in L^1_{loc}(\Omega)$. For any $U\subset \Omega$ open, define
\begin{equation}\label{eq:def du}
    |\dD u|(U) \coloneqq \inf \Big\{\liminf_{n\to\infty} \int \lip\, u_n\, \d \mm \colon (u_n)\subset \Lip_{loc}(U),\, u_n \to u \text{ in }L^1_{loc}(U)\Big\}.
\end{equation}
  We then say that $u$ is of \emph{locally bounded variation}, writing $u \in BV_{loc}(\Omega)$, provided $|\dD u|$ is finite on some open neighborhood of each point in $\Omega$. If $u \in L^1(\Omega)$ and $|\dD u|(\Omega)<\infty$, we say that $u$ is of \emph{bounded variation} and write $u \in BV(\Omega)$.
\end{definition}

Whenever $u \in BV_{loc}(\Omega)$, by a standard Carath\'eodory  construction, it is possible to show (\cite{Miranda03,AmbrosioDiMarino14}) that $|\dD u|$ is the restriction to open sets of a locally finite non-negative Borel measure on $\Omega$, called \emph{total variation measure}, denoted with the same symbol, and satisfying
\begin{equation}\label{eq:inf tot var}
    |\dD u|(E)=\inf_{U}   |\dD u|(U), \quad \text{ for all Borel sets $E\subset \Omega$,}
\end{equation}
where the infimum runs over all the open sets $U\subset \Omega$ containing $E.$ For all $E\subset \Omega$ Borel we say that $E$ is of locally finite perimeter (resp.\  finite perimeter) if $\nchi_E \in BV_{loc}(\Omega)$ (resp.\  $ \nchi_E \in BV(\Omega)$). In any case, we denote the \textit{perimeter measure of $E$} as
\[
\Per(E,\cdot) \coloneqq |\dD \nchi_E|.
\]
For any $E\subset \Omega$ Borel, we recall that
\begin{equation}
    \Per(E,\Omega) = \Per(E,\partial E \cap \Omega)\label{eq:topological supp Per}.
\end{equation}

Recall  the \emph{coarea formula} (see \cite{Miranda03,AmbrosioDiMarino14}): if $u \in BV_{loc}(\Omega)$, then   the sets $\{u>t\}$ are  of locally finite perimeter in $\Omega$ for a.e.\ $t$ and   for any $\varphi \colon \Omega\to [0,\infty]$ Borel it holds
\[
\int \varphi \,\d |\dD u| = \int_\R\int\varphi \, \d \Per(\{u>t\},\cdot)\d t.
\]
It is part of the statement that all the above integrals make sense.

We shall also consider several times Sobolev and BV calculus on weighted intervals, i.e.\ for $\omega$ a boundedly finite non-negative measure in $\R$ and $J\subset \R$ open interval (possibly unbounded). In this case, we will denote by $BV_{loc}(J;\omega)$ the function space $BV_{loc}(J)$ with respect to the metric measure space $(\R,|\cdot |,\omega)$ (the same for other spaces e.g.\ $BV$, $W^{1,p}$, $W^{1,p}_{loc}$). This is to stress that we are considering a different weight instead of the Lebesgue measure. Notice that, differently from before, we do not require here that $\supp(\omega)=\R$. Lastly, for a function $v$ of locally bounded variation in $J$ in the classical sense, we will denote by $TV(v)$ the standard total variation measure of $v$. In this case, the Radon-Nikodym decomposition reads as $TV(v) = |v'|\Leb 1\mres J + TV^s(v)$ with $TV^s(v)$ singular with respect to the Lebesgue measure. In particular, if $v$ is locally absolutely continuous in an open interval $J$, we shall write $v \in {\sf AC}_{loc}(J)$ and in this case $TV(v) = |v'|\d t$ holds as measures on $J$. We will discuss equivalences of metric and classical approaches in Appendix \ref{Appendix:Calculus1D}.
\begin{remark}
    \rm
    Notice that if $E\subset \Omega$ is Borel with $\Per(E,\Omega)<\infty$, there is a slight ambiguity when writing $\Per(E,.)$. Indeed, as $\nchi_E \in L^1_{loc}(\X)$, the value $\Per(E,U)$ is  also defined, possibly $+\infty$, for any $U\subset \X$ open. In general, it is not true that $\Per(E,.)$ extends to a Borel measure on $\X$ as this heavily depends on $\Omega$ (see also \cite{CaputoKoivuRajala24} for a related theory of perimeter extension domains). However, the condition $\Per(E,\Omega)<\infty$ guarantees that this is true as a finite measure on $\Omega$. In this note, we shall typically consider  $\Per(\{u>t\},\cdot)$ for $u \in BV_{loc}(\Omega)$ and this will always be considered as a measure on $\Omega$. We therefore shall not insist on this fact.
    \fr
\end{remark}

For $u \in BV_{loc}(\Omega)$, by the Radon-Nykodym theorem, we have the decomposition
\[
|\dD u| = |\dD^a u| + |\dD^s u|,
\]
where $|\dD^a u|\ll \mm\mres\Omega$ and $|\dD^s u|\perp \mm\mres\Omega$. We denote $|Du|_1$ the density of the absolutely continuous part, i.e.\ $|\dD^a u| = |Du|_1 \mm\mres{\Omega}$.
 Even though in some parts of this note, a finer decomposition would be in place under a doubling and Poincar\'e assumption (see \cite{Ambrosio2001,Ambrosio2002,AmbrosioMirandaPallara04,KinnunenKorteShanmugalingamTuominen12}), we mainly want to work in more general metric measure spaces and thus we need to consider the above more abstract result.

We conclude this part by discussing the dependence of the non-smooth calculus on the integrable exponent $p$, especially in the limit case $p=1$. On an arbitrary metric measure space $\Xdm$ and for $\varnothing \neq \Omega\subset\X$ open and $p>1$, it is always true that $u \in W^{1,p}_{loc}(\Omega)\subset BV_{loc}(\Omega)$ and 
\begin{equation}
   |\dD u|\ll \mm\mres\Omega\qquad\text{and}\qquad |Du|_1 \le |Du|_p, \quad \mm \text{-a.e.\ on }\Omega. \label{eq:dependence gradient}
\end{equation}
By locality of the total variation on open sets and using the definition of $W^{1,p}_{loc}(\Omega)$, it suffices to check the above for $u \in W^{1,p}(\X)$ with $u,|D u|_p \in L^1(\X)$. To see this, consider an optimal sequence $u_n \subset \Lip(\X)$ for $\rmCh_p(u)$ with the property that $u_n \to u$ and $\lip (u_n) \to |D u|_{p}$ in $L^p(\X)$ as $n\uparrow\infty$. For any $U\subset \X$ open and bounded, the lower semicontinuity of $|\dD u|$ on open sets with respect to the $L^1_{loc}$-convergence guarantees that 
\[
|\dD u|(U) \le \liminf_{n\uparrow\infty} |\dD u_n|(U_n) \le \liminf_{n\uparrow\infty} \int_U \lip \, u_n\,\d \mm = \int_U |Du|_p\,\d \mm,
\]
having also used \cite[Remark 5.1]{AmbrosioDiMarino14}. Since we assume here $u,|Du|_p \in L^1(\X)$ and since $U$ is arbitrary, the above gives at the same time that $u \in BV(\X)$ and $|Du|_1 \le |Du|_p$ $\mm$-a.e.\ as desired.

In the general situation, we should not expect equality in \eqref{eq:dependence gradient}. Even worse, there might exist $u \in BV(\X) \cap L^{p}(\X)$ with $|\dD u|\ll \mm$ and $\d |\dD u|/\d \mm \in L^{p}(\X)$ but $u \notin W^{1,p}(\X)$ (\cite[Example 7.4]{AmbrosioDiMarino14}) and a similar phenomenon might occur also considering only Sobolev functions for different exponent $p$, see\cite[Remark 5.2]{KinnunenKorteShanmugalingamTuominen12} and \cite{DiMarinoSpeight13}. However, in several parts of this note, we shall enforce regularity of some sort to the underlying metric measure space preventing some of these situations from happening, typically PI-paces see \cite{Cheeger00,Ambrosio-Pinamonti-Speight15} or curvature dimension conditions \cite{GigliHan14,GigliNobili21}. In this case, we shall automatically write $|Du|$ in place of $|Du|_p$ without notice except when $p=1$.
\section{Decreasing rearrangement in metric spaces}\label{sec:monotone rearrangements}
In this part, we study the decreasing rearrangement into a weighted real line for real valued functions defined in metric spaces and investigate their fine and regularizing properties inspired by  \cite{CianchiFusco02}. The goal is also to single out key assumptions making it possible to develop a general theory of P\'olya-Szeg\H{o} inequalities in a metric setting.
\subsection{Definitions and notation}\label{sec:Def and Not}
Let  $\Xdm$ be a metric measure space and $
\varnothing \neq \Omega\subset \X$ be open (possibly unbounded) and $u:\Omega\to \rr$ be a Borel function such that $\mm( \{u>t\})<\infty$ for any $t >\essinf u$. We define  $\mu:(\essinf u,+\infty)\to[0,\infty)$, the distribution function of $u$ as  $\mu(t)\coloneqq \mm(\{u> t\})$. For $u$ and $\mu$ as above, we consider the generalized inverse $u^{\#}:[0,\infty]\to [\essinf u,\esssup u]$ of $\mu$ defined by
\begin{equation*}
	u^{\#}(s)\coloneqq 
	\begin{cases*}
		{\rm ess}\sup u & \text{if $s=0$},\\
		\inf\left\lbrace t > \essinf u \ :\ \mu(t)<s \right\rbrace &\text{if $s>0$},\\
  \essinf u & \text{if $s=\infty$}.
	\end{cases*}
\end{equation*}
Note that $u^{\#}$ is non-increasing, left-continuous. Moreover, it holds
$$u^{\#}(s)={\rm ess}\inf u, \qquad \text{for all $s \ge  \mm(\Omega)$,}
$$
meaning that $\lim_{s\to +\infty} u^{\#}(s)={\rm ess}\inf u$ in the case $\mm(\Omega)=\infty.$

To define the decreasing rearrangement of $u$ on the real line, we  consider a non-zero positive measure $\omega \ll \Leb 1$  satisfying the basic assumptions:
\begin{align}
    &\mm(\Omega)\le \omega(\R),   \label{eq:mass compatibility}  \\
&{\omega((-\infty,x])<\infty,\quad\forall x \in I,\qquad \lim_{x\to -\infty}\omega((-\infty,x])=0}, \label{eq:zero a sx}
\end{align}
where in \eqref{eq:mass compatibility} we use the convention that $\infty\le\infty$ is true. Note that if $\mm(\Omega)<\infty$ then \eqref{eq:mass compatibility} always holds up to multiplying $\omega$ by a constant. 
We also define the usual \textit{cumulative distribution function} $F_\omega: \R \to [0,\infty] $ as 
\[
F_\omega(x) \coloneqq \omega((-\infty,x)),\qquad \forall \, x\in\R.
\]
We also set 
\begin{align*}
    &l^*\coloneqq \inf \{ l \in\R \ : \ F_\omega(l)>0  \} \in \R \cup \{-\infty\},\\
    &r^* \coloneqq \sup \{r \in \R \ : \ F_\omega(r)<\mm(\Omega)\} \in \rr \cup \{+\infty\}.
\end{align*}
Notice that $r^*$ is the smallest such that $\omega((-\infty,r^*))=\mm(\Omega).$ We can then define 
\[
\Omega^*\coloneqq (l^*,r^*),
\]
note that $\Omega^*$ depends only on $\Omega$ and not on $u.$
Given $u$ as above we then define the \emph{decreasing rearrangement} $u^*:\Omega^* \to [\essinf u,\esssup u]$ of $u$ with respect to  $\omega$ as
\[
\Omega^* \ni x\mapsto u^*(x) \coloneqq u^\#(F_\omega(x)).
\]
Note that $u^*$ is finitely valued, non-increasing and left-continuous as it is $u^\#$. In particular, $u^*$ is of locally bounded variation in the open interval $\Omega^*$ in the classical sense.
By construction, $u^*$ is independent of the chosen representative of $u$ up to $\mm\mres\Omega$-a.e.\ equality. Therefore, we can also consider rearranging $u \in L^1_{loc}(\Omega)$   as $u^*$ will be well defined.  We collect in the following lemma some elementary properties of $u^*$ and,  since they are usually presented only for non-negative functions with bounded support, we include a short proof.
\begin{lemma}\label{lem:basic prop u*}
Let $\Xdm$ be a metric measure space, let $\varnothing \neq \Omega\subset\X$ be open, and let $u \colon \Omega \to \R$ be Borel with $\mm(\{u>t\})<\infty$ for $t>\essinf u$ and $\omega \ll\Leb 1$ be satisfying \eqref{eq:mass compatibility},\eqref{eq:zero a sx}. It holds:
\begin{itemize}
    \item[{\rm i)}] $u$ and $u^*$ are equimeasurable, meaning that $\mm( \{u>t\}) = \omega( \{u^*>t\})$ for all $t \in \R$, in particular $\essinf u=\essinf u^*$ and $\esssup u=\esssup u^*;$
    \item[{\rm ii)}] $\mm(\{u=t\})=\omega( \{u^*=t\})$ for all $t>\essinf u$;
    \item[{\rm iii)}] $\mm( \{u<t\})=\omega(\{u^*<t\})$ for all $t\in \R$;
    \item[{\rm iv)}]for all continuous functions $G:\R \to \R$ it holds
    \begin{equation}
    \int_\Omega G(u) \,\d \mm= \int_{\Omega^*} G(u^*)\, \d \omega,
    \label{eq:composition rearr G}
    \end{equation}
    meaning that  one integral makes sense if and only if the other does, in which case the above holds;
    \item[{\rm v)}] $\| u\|_{L^p(\Omega)} = \| u^*\|_{L^p(\Omega^*;\omega) }$, for all $p \in [1,\infty]$.
    \end{itemize}
\end{lemma}
\begin{proof}
    Item i) can be checked using the definitions as in the case $u\ge 0$ (see, e.g., \cite{ciao}).  Item ii) follows from i) observing that
    \[
    \mm(\{u\ge t\})=\lim_{s \to t^-}\mm(\{u>s\})=\lim_{s \to t^-}\omega(u^*>s)= \omega(\{u^*\ge t\}), \quad \forall \, t>\essinf u.
    \]
   Item iii) for $t\le \essinf u$ turns into $0=0$, so we can assume  $t>\essinf u$. Then
   \begin{align*}
       \mm( \{u<t\})=\mm(\Omega)-\mm( \{u\ge t)\overset{{\rm i),ii)}}{=}\omega(\Omega^*)-\omega( \{u\ge t\})=\omega(\{u^*<t\}),
   \end{align*}
   note that the above  makes sense even if $\mm(\Omega)=\infty$, since by assumption $\mm( \{u\ge t)<\infty$ for all $t>\essinf u$.
  To show iv) we can assume that $G\ge 0$, otherwise we can consider first $G^+$ and then $G^-.$ By the layer cake formula it is sufficient to show that $\mm(\{G(u)>t\})=\omega(\{G(u^*)>t\}\cap \Omega^*)$ for all $t\ge 0$. Since $G$ is  continuous, the set $G^{-1}((t,\infty))$ is countable union of disjoint open intervals $(a_i,b_i)$. However by i) and iii) we have $\mm(\{a_i<u<b_i\})=\omega(\{a_i<u^*<b_i\})$ for all $i$. Hence we conclude by countable additivity. Item v) follows from iv) taking $G(t)=(t^+)^p$, $G(t)=(t^-)^p.$
\end{proof}
\subsection{Key assumptions}\label{sec:key}
In this section, we single out key assumptions for our rearrangement theory to work and discuss their consequences. Throughout this part, $\Xdm$ will be a metric measure space and $\varnothing \neq \Omega\subset \X$ is an open set and $u \colon \Omega\to\R$ is a Borel function satisfying $\mm(\{u>t\})<\infty$ for $t>\essinf u$.

The first is the requirement of some extra regularity on the measure $\omega$ used in the rearrangement.
\begin{equation}\label{eq:ass1}\tag{{\color{red}{\rm I}}}
      \parbox{12cm}{\textit{Assumption on the weighted interval} $(I,\omega)$: $I\subset \R$ is a possibly unbounded open interval and $\omega=g \Leb 1\mres I$ where $g: I\to (0,\infty)$ is continuous. 
      Moreover, if $(I,\omega)$ is used as a target for the rearrangement  $u^*$ of some $u:\Omega\to [0,\infty),$ we assume the compatibility conditions \eqref{eq:zero a sx} and \eqref{eq:mass compatibility} to be satisfied.
      }
\end{equation}
If  $(I,\omega)$ satisfies Assumption \eqref{eq:ass1}, then the cumulative distribution function $F_g\coloneqq F_\omega$ is strictly increasing in $I$ and admits an inverse $F_g^{-1}: (0,\omega(\R))\to I$. We also define the \textit{isoperimetric perimeter profile} function $\mathcal I^\flat_g(v): [0,\omega(\R)]\to [0,\infty)$ as
\begin{equation}\label{eq:profile}
    \mathcal I^{\flat}_g(v)\coloneqq g(F_g^{-1}(v)), \text{ if $v>0$,} \quad \mathcal I_g(0)\coloneqq 0.
\end{equation}
For each $t\in (\essinf u, \esssup u)$, we also set
\begin{equation}\label{eq:rt}
    r_t\coloneqq F_g^{-1}(\mu(t)),
\end{equation}
which is well defined because $0<\mu(t) <\mm(\Omega)\le \omega(\R)$. Note that  $t \mapsto r_t \in I$ is monotone non-increasing, hence Borel and can be naturally extended to the endpoints $\{\essinf u, \esssup u\}$ (when they are finite) setting $r_{t_0}\coloneqq \lim_{r\to t_0^+} r_t$, $r_{t_1}\coloneqq \lim_{r\to t_1^-} r_t$ for $t_0\coloneqq \essinf u$ and $t_1\coloneqq \esssup u.$   It might happen that $r_{t_0}=+\infty$ and/or $r_{t_1}=-\infty.$ In fact $r_{t_1}=\inf I.$
In particular, because $\mu(t)=\omega(\{u^*>t\})$, we must have 
\begin{equation}\label{eq:boundary superlevel}
    \partial \{u^*>t\}=\{r_t\}, \quad \text{for every $t \in (\essinf u, \esssup u)$.}
\end{equation}
Recall that  $u^*$ is classically of locally bounded variation on $\Omega^*$. Hence, from Lemma \ref{lem:LEMMONE} we know that it belongs as well to $BV_{loc}(\Omega^*;\omega)$. In particular,  $\Per(\{ u^*>t\},\cdot)$ (with respect to the metric measure space $(\R,|\cdot |,\omega)$) is a well-defined locally finite measure on $\Omega^*$. Therefore, by combining \eqref{eq:perimeter of halfline} with \eqref{eq:profile},\eqref{eq:rt},  we have 
 \begin{equation}\label{eq:Per u* g rt}
     \Per( \{u^*>t\},\cdot)=g(r_t)\delta_{r_t}= \mathcal I^{\flat}_g(\mu(t))\delta_{r_t},\quad  \text{for every $t \in (\essinf u, \esssup u)$,}
 \end{equation}
We shall always write $\Per(\{ u^*>t\})\coloneqq\Per(\{ u^*>t\},\Omega^*)$ for brevity.

A key assumption to deduce a P\'olya-Szeg\H{o} inequality that we will often make is the validity of 
\begin{equation}\label{eq:isop polya pre}
   \Per(\{u>t\},\Omega)\ge {\sf C}\cdot  \Ig(\mu(t)),
\end{equation}
for a suitable set of values of $t$ and ${\sf C}>0$. The above can also be stated for an arbitrary Borel set $E\subset\Omega$ by requiring
\begin{equation}\label{eq:general isop pre}
    \Per(E,\Omega)\ge {\sf C} \cdot \Ig(\mm(E)).
\end{equation}

We now discuss a regularity property of the function $u$ we are going to rearrange (recall that $\bar u$ is the precise representative of $u$ defined in Section \ref{sec:calculus}).
\begin{equation}\label{eq:ass2}\tag{{\color{red}{\rm II}}}
 \parbox{12cm}{\textit{Assumption on} $u \in BV_{loc}(\Omega)$: it holds $\mm(\Omega\setminus C_u)=0$ and for all $r\in \R$ 
 \[
     \Per(\{u>t\}, \{ \bar u>r\}\cap C_u)=0, \quad \text{for a.e.\ $t\in (-\infty,r)$,}
 \]
 and 
 \[
     \Per(\{u>t\}, \{ \bar u>r\}\cap C_u)= \Per(\{u>t\},C_u), \quad  \text{for a.e.\ $t\in (r,\infty)$}.
 \]
}
\end{equation}
Let us list a few important and relevant cases when Assumption \eqref{eq:ass2} is true.
\begin{itemize}
    \item[{\rm i)}] \textit{$u \in C(\Omega)\cap BV_{loc}(\Omega)$}:  we clearly have $\Omega = C_u$ and $u=\bar u$. Also, we have that $\{u>t\}$ has locally finite perimeter for a.e.\ $t$ by coarea. On any such $t$ and for every $r>t$, it holds
 \[
 \Per(\{u>t\}, \{\bar u>r\}\cap C_u) = \Per(\{u>t\}, \{ u>r\}\cap \partial \{u>t\}) =0,
 \]
 where in the first identity we used that $C_u = \Omega$ and that the perimeter measure is concentrated on the topological boundary (recall \eqref{eq:topological supp Per}). The second part of the assumption is also clear by continuity.
\item[{\rm ii)}] $u \in BV_{loc}(\Omega)$ \textit{on PI-space}: the fact that  $\mm(\Omega\setminus C_u)=0$ is guaranteed by the local doubling property. Moreover by \cite{Ambrosio2001,Ambrosio2002}, we know that for a.e.\ $t>r$, the perimeter measure $\Per(\{u>t\},\cdot)$ is concentrated on the essential boundary $\partial^*\{u>t\}\cap \Omega$ (see \cite[Proposition 4.1]{CaputoCavallucci24} for the localized result on open sets). In this case, we have automatically for all $r>t$ that
\[
    \Per(\{u>t\}, \{ \bar u>r\}\cap C_u) =\Per(\{u>t\}, \{ \bar u>r\}\cap C_u\cap \partial^*\{u>t\}) =0,
\]
since at any $x \in  C_u\cap \partial^*\{u>t\}$ we have $\bar u(x) =t$  and thus $x \notin \{ \bar u>r\}$ as $r>t$. The last fact also shows the second property when $r<t$.
\item[{\rm iii)}] $u$ \textit{is monotone in $J\subset I$ open subinterval when $(I,\omega)$ satisfies Assumption \eqref{eq:ass1}}. Indeed  $\omega(J\setminus C_u)=0$ is clear, while by \eqref{eq:Per u* g rt} we know that the perimeter measure $\Per(\{u>t\},\cdot)$ is concentrated on  $\partial^*\{u>t\}\cap J$ and we can argue as in the previous point. Actually, any $u \in BV_{loc}(J;\omega)$ satisfies Assumption \eqref{eq:ass2}, but we shall neither prove nor use this fact.
\end{itemize}
\begin{lemma}[Coarea under Assumption \eqref{eq:ass2}]\label{lem:fine coarea}
     Let $\Xdm$ be a metric measure space, let $\varnothing \neq \Omega\subset\X$ be open and $u \in BV_{loc}(\Omega)$ be as in Assumption \eqref{eq:ass2}. For every $ r<s$ and for every $\varphi \colon \X \to {[0,\infty]}$ Borel with $\phi=0$ in $\X\setminus C_u$ it holds
     \begin{equation}
     \int_{\{r<\bar u< s\}} \varphi  \,\d |\dD u| = \int_r^s\int \varphi \, \d \Per( \{u>t\},\cdot)\d t,
     \label{eq:coarea fine}
     \end{equation}
     where $\bar u $ is the precise representative of $u$. In particular, $|\dD u|(\{\bar u=t\}\cap C_u)=0$ holds for all $t \in \R$.
\end{lemma}
\begin{proof}
   The standard coarea formula and the fact that $\phi$ vanishes outside of $C_u$ give
    \[
   \int_{\{r<\bar u< s\}} \varphi  \,\d |\dD u| = \int_{-\infty}^{+\infty}\int_{\{r<\bar u< s\}\cap C_u}  \phi \,\d\Per(\{u>t\},\cdot )\d t.
    \]
    Thanks to Assumption \eqref{eq:ass2} we have
    \[
    \Per(\{u>t\},\{r<\bar u< s\}\cap C_u )= \Per(\{u>t\}, C_u ) \nchi_{(r,s)}(t), \quad \text{a.e.\ $t\in\R$.}
    \]
   Combining this with the above, we  reach
    \[
   \int_{\{r<\bar u<s\}} \varphi  \,\d |\dD u| = \int_r^s\int_{C_u}  \phi \,\d\Per(\{u>t\},\cdot )\d t,
    \]
    from which using again that $\phi$  vanishes outside of $C_u$ gives \eqref{eq:coarea fine}.
\end{proof}
Under Assumption \eqref{eq:ass2} we introduce some further notation. 
Recall first, by mutual singularity in the Radon-Nikodym decomposition, that we can find $S \subset \Omega$ Borel such that
\[
|\dD^s u|(\Omega) = |\dD u|(S) \quad \text{and} \quad \mm(S)=0.
\]
Next we fix a finitely valued Borel representative of $|Du|_1 \in L^1(\Omega)$ denoted by $|\underline D u|_1$. Set
\begin{equation}\label{def:Du+}
    D^0_u \coloneqq \{ x \in  C_u \setminus S \colon | \underline D u|_1(x) = 0\}\subset \Omega, \quad D^+_u \coloneqq \{ x \in  C_u \setminus S \colon |\underline Du|_1(x) \in (0,\infty)\}\subset \Omega.
\end{equation}
In particular $D^0_u\cup D^+_u=C_u \setminus S.$
Since $\mm(\Omega\setminus C_u)=0$, by construction it holds
\begin{align}
&\mm( \Omega\setminus (D^0_u \cup D^+_u))=0,\label{eq:sym diff} \\
&|\dD u|\mres{D^+_u} = \nchi_{D^+_u}|Du|_1\mm \mres \Omega\label{eq:AC part}.
\end{align}
Clearly, the sets $D^0_u$ and $D^+_u$ depend on the chosen representative of $|Du|_1$, however this is not a big issue since they will be used only on some technical intermediate steps of the arguments. We observe that, with the above notation, it holds
\begin{equation}\label{eq:togli Du+}
   |\dD^s u|=0\quad \implies\quad \Per(\{u>t\},\Omega\setminus D^+_u)=0, \qquad \text{for a.e.\ $t\in \R$}.
\end{equation}
Indeed, by coarea formula we have
\[
0=\int_{S\cup D_u^0} |D u|_1 \d \mm=|\dD u|(\Omega\setminus D^+_u )  = \int_\R \Per(\{u>t\},\Omega\setminus D^+_u)\,\d t.
\]
We also point out the following locality property. From  Lemma \ref{lem:fine coarea} we notice for all $t \in \R$
\[
    \int_{\{ \bar u = t\}\cap C_u } |Du|_1\,\d \mm \overset{\eqref{eq:AC part}}{=}  \int_{\{ \bar u = t\} \cap C_u} \nchi_{D^+_u}\,\d |\dD u|\le  |\dD u|(\{ \bar u = t\} \cap C_u )  =0,
\]
Hence, again by $\mm(\Omega\setminus C_u)=0$, we have 
\begin{equation}
    |Du|_1=0,\qquad \mm\text{-a.e.\ on } \{ u =t \}. \label{eq:Du1 local on Cu}
\end{equation}
\begin{remark}[The sets $D^+_u,D^0_u$ in 1-dimension]\label{rmk:nobady cares}\rm
In the case $u \in BV_{loc}(J;\omega)$ with $(I,\omega)$ a weighted interval satisfying Assumption \eqref{eq:ass1} and $J\subset I$ open sub-interval, we will \textit{always} do the above construction in the following more explicit way.
    We choose the representative $|\underline D u|_1$   to be $|u'|$ on $J$ extended by zero outside the set $D_u$ of differentiability points of $u$ on $J$. Moreover, we  take
\[
D_u^+\coloneqq D_u \cap \{r \in J \colon |u'|(r)>0\}, \qquad  D_u^0\coloneqq D_u \cap \{r \in J \colon |u'|(r)=0\},
\]
as the sets that were introduced in \eqref{def:Du+}. This is an admissible choice since  $D_u\subset C_u$ and the singular part of $|\dD u|$ is concentrated on $S\coloneqq J\setminus D_u$, indeed so is $TV(u)$ and by Lemma \ref{lem:LEMMONE} we know that $|\dD u|=g TV(u)$ as measures on $J$.  \fr
\end{remark}
\subsection{Properties and regularization effects} 
Here we study fine regularizing effects of decreasing rearrangements.

We begin by splitting the distribution  function   into an absolutely continuous part and an a priori less regular part. We follow \cite[Lemma 3.1]{CianchiFusco02} (see also \cite[Section 4]{Burchard97}). 
\begin{lemma}[Formula for the distribution function]\label{lem:distfunct}
    Let $\Xdm$ be a metric measure space and let $\varnothing \neq \Omega \subset \X$ be open. For every $u \in BV_{loc}(\Omega)$ non-negative satisfying Assumption \eqref{eq:ass2} it holds
    \begin{equation}\label{eq:identity mu}
         \mu(t) = \mm(\{u>t\}\cap D^0_u) + \int_t^\infty \int \frac{\nchi_{D^+_u}}{|\underline D u|_1} \, \d \Per(\{u>r\},\cdot)\d r.
    \end{equation}
    Moreover, if  $\mm(\{u>r\})<\infty$ for some $r\in \R$ we have
    \begin{equation}\label{eq:mu' inequality}
       \mu'(t)  \le  -\int \frac{\nchi_{ D^+_u}}{|\underline D u|_1} \, \d \Per(\{u>t\},\cdot), \quad \text{for a.e.\ $t>r$.}
    \end{equation}
\end{lemma}
\begin{proof}
For any $F \subset \Omega $ Borel by \eqref{eq:sym diff}  we have that
\[ \mm(F ) = \mm( F \cap D^0_u) + \mm( F \cap  D^+_u).\]
Thus, choosing $F=\{\bar u>t\}$ and since $u=\bar u$ $\mm$-a.e.\ on $\Omega$, we can write
 \[ 
\begin{split}
    \mu(t)  &= \mm(\{ u>t\}\cap D^0_u)  +  \int \frac{\nchi_{ \{\bar u>t\}\cap D^+_u}}{|\underline D u|_1}|D u|_1\,\d \mm\\
    &\overset{\eqref{eq:AC part}}{=}\mm(\{ u>t\}\cap D^0_u) +  \int \frac{\nchi_{ \{\bar u>t\} \cap D^+_u}}{|\underline D u|_1}\, \d |\dD u|  \\  
    &= \mm(\{u>t\}\cap D^0_u) +  \int_t^\infty \int  \frac{\nchi_{ D^+_u}}{|\underline D u|_1}\,\d\Per(\{u>r\},\cdot)\d r,
\end{split}
\]
having used, lastly, the coarea formula \eqref{eq:coarea fine} relying on Assumption \eqref{eq:ass2}. This proves \eqref{eq:identity mu}.     Finally, differentiating \eqref{eq:identity mu} we get \eqref{eq:mu' inequality} as $t\mapsto \mm(\{u>t\}\cap D^0_u)$ is monotone non-increasing.
\end{proof}
We report below \cite[Lemma 2.5]{CianchiFusco02}.
\begin{lemma}\label{lem:monotone functions}
    Let $J\subset \R$ be a (possibly unbounded) open interval and let $v\colon J \to\R $ be monotone. Then:
    \begin{itemize}
        \item[{\rm i)}] $v(F)$ is Borel for all $F\subset J$ Borel;
        \item[{\rm ii)}]  $v$ is  differentiable almost everywhere. Moreover, if $s_0 \in J$ is a point of differentiability with $v'(s_0) \neq 0$, then $v(s) \neq v(s_0)$ for all $s\neq s_0$.
    \end{itemize}
\end{lemma}
We now revisit \cite[Lemma 2.4]{CianchiFusco02} according to our weighted model spaces.
\begin{lemma}\label{lem:BVdistribution}
 Let $J \subset \R$ be a (possibly unbounded) open interval and let $ v  \colon J \to \R$ be a monotone non-increasing function. Then, denoted by $D_v^0\subset J$ the set where $v'$ exists and $v'=0$,
   the set $v (D_v^0)$ is Borel and 
    \[ \Leb 1 (v (D_v^0)) =0.\]
    Moreover, if $(I,\omega)$ is a weighted interval as in Assumption \eqref{eq:ass1} and $J\subset I$  {and  $\omega(\{ v > r\})<\infty$ for some $r\in \R $}, then the function $ (r,\infty)\ni t\mapsto  h(t) \coloneqq \omega (D^0_v \cap \{v>t\})$ is non-increasing, right-continuous and satisfies
    \[
    |h'|(t) =0\qquad \text{a.e.\ }t>r.
    \]
\end{lemma}
\begin{proof}
 Being $v$ monotone on $J$, it is clearly of locally bounded variation in the classical sense. Recall also from i) in Lemma \ref{lem:monotone functions} that $v(F)$ is Borel if $F\subset \R$ is Borel, since $v$ is monotone. Let $(a,b)\subset \R$ be any open interval and notice that, denoting by ${\rm Var}_v(\cdot)$ the pointwise variation of $v$ by means of supremum over partitions, we have
    \[ \Leb 1(v((a,b))) \le  {v(a^+)-v(b^-)}  = {\rm Var}_v((a,b)).\]
    Thus, since any open set $U \supseteq
 D^0_v$ is a countable union of disjoint open intervals, we can write $\Leb 1 (v(D^0_v))  \le TV(v)(U)$. By outer regularity of $TV(v)$, the first claim follows since 
    \[\Leb 1( v(D^0_v) ) \le TV(v)(D^0_v) =\int_{D^0_v} |v'| \,\d\Leb 1  =0, \]
    where we used that $TV(v)\mres{D_v}=|v'|\mathcal L^1\mres{D_v},$ with $D_v$ set of differentiability points in $J$.

    Let us now consider that $(I,\omega)$ as in Assumption \eqref{eq:ass1}, hence $\omega=g\Leb 1$ for some continuous function $g:I\to (0,\infty)$. In particular, $\omega$ and $\Leb 1$ are mutually absolutely continuous in $I$.
    Assume thus $\omega(\{v>r\})<\infty$ for some $r \in \R$. The function  $h(t) \coloneqq \omega (D^0_v \cap \{v>t\})$ is, by definition, finite, monotone non-increasing and right continuous in $(r,\infty).$ Hence it is a.e.\ differentiable and it is of locally bounded variation in $(r,\infty)$. Also by right continuity
    \[ \omega( D^0_v \cap \{s < v \le t\})=h(s)-h(t) = TV(h)((s,t]),\qquad \forall \, s<t.\]
    In particular, the above implies that for all intervals $(t_1,t_2)\subset (r,\infty)$
    \[ \omega( D^0_v \cap \{t_1 < v <t_2\}) =\lim_{t\to t_2^-} \omega( D^0_v \cap \{t_1 < v \le t\})=\lim_{t\to t_2^-}  TV(h)((t_1,t])=TV(h)((t_1,t_2)).\]
    From the arbitrariness of the interval $(t_1,t_2)$, it follows by outer regularity that
    \[\omega( D^0_v \cap v^{-1}(F)) = TV(h)(F),\qquad \forall \, F \subset  (r,\infty) \text{ Borel}. \]
    By what we proved in the first part we have that $F \coloneqq
 v(D^0_v)$ is Borel with $\Leb 1(F)=0$, hence
    \[  TV(h)((r,\infty)\setminus F) = \omega( D^0_v \cap ({ I} \setminus v^{-1}(v(D^0_v)))) =0, \]
    which means that $TV(h)$ is concentrated on $v(D^0_v)$. Since $v(D^0_v)$ is $\Leb 1$-negligible, we  conclude.
\end{proof}
Since a function and its rearrangement are equidistributed, we can compute the derivative $\mu'(t)$ (which  exists a.e.\ as $\mu(t)$ is monotone) in terms of explicit quantities related to the rearrangement $u^*$. This will be possible for a general $u$ due to the regularizing effects of the decreasing rearrangement.

We follow and adapt here \cite[Lemma 3.2]{CianchiFusco02} (compare with the version \cite[Lemma 3.10]{MondinoSemola20} which required approximation procedures \cite[Lemma 3.6]{MondinoSemola20}).

\begin{lemma}[Derivative of distribution function] \label{lem:derivative mu}
Let $\Xdm$ be a metric measure space, let $\varnothing \neq \Omega\subset\X$ be open and let $u \colon \Omega\to \R$ be Borel with $\mm(\{u>t\})<\infty$ for all $t>\essinf u$. Let  $(I,\omega)$ be a weighted interval as in Assumption \eqref{eq:ass1} and $u^*$ be the rearrangement of $u$ with respect to  $\omega$. Then for every $t \in u^*(D^+_{u^*})$ the preimage $(u^*)^{-1}(\{t\})=\{r_t\}$ is a singleton (with $r_t$ as in \eqref{eq:rt}) and it holds
\[
\begin{aligned}
-\mu'(t) &= \frac{\Per(\{u^*>t\})}{|(u^*)'|((u^*)^{-1}(\{t\}))},&\quad  & \text{a.e.\ }t \in u^*(D^+_{u^*}),\\
\mu'(t) &= 0,&\quad  & \text{a.e.\ }t \in (\essinf u,\esssup\, u) \setminus u^*(D^+_{u^*}).
\end{aligned}
\]
\end{lemma}
\begin{proof}
Set $m\coloneqq \essinf u$ and $M\coloneqq \esssup u.$ Recall that $u^*\in BV_{loc}(\Omega^*;\omega)$ and that $u^*$ satisfies Assumption \eqref{eq:ass2}. Recall that $D_{u^*}^+=\{|(u^*)'|>0\}\subset\Omegastar$ and  $D_{u^*}^0=\{|(u^*)'|=0\}\subset\Omegastar$ (see Remark \ref{rmk:nobady cares}). Therefore applying \eqref{eq:identity mu} to $u^*$ we get
\[
 \mu(t) = \omega(\{u^*>t\}\cap D^0_{u^*}) + \int_t^\infty \int \frac{\nchi_{D^+_{u^*}}}{|(u^*)'|} \, \d \Per(\{u^*>r\},\cdot)\d r.
\] 
Moreover, by Lemma \ref{lem:BVdistribution} we have
    \[ 
    \frac{\d}{\d t} \omega(\{u^*>t\}\cap D^0_{u^*})=0, \qquad \text{a.e.\ }t\in(m,M).
    \]
    Therefore, to conclude it is sufficient to show that 
    \begin{equation}\label{eq:kiave integrale}
         \int \frac{\nchi_{D^+_{u^*}}}{|(u^*)'|}\,\d \Per( \{u^*>t\},\cdot) =\nchi_{u^*(D^+_{u^*})}(t) \frac{ \Per(\{u^*>t\})}{|(u^*)'((u^*)^{-1}(\{t\}))|},\qquad \forall \, t\in (m,M),
    \end{equation}
    where we are implicitly saying that  $(u^*)^{-1}(\{t\})$ is a singleton for $t \in u^*(D^+_{u^*})$ and taking the right hand side to be zero whenever $t\notin u^*(D^+_{u^*}).$

    To prove \eqref{eq:kiave integrale} we start recalling that by \eqref{eq:boundary superlevel} and \eqref{eq:Per u* g rt} we have that for all $t \in(m,M)$ the measure $\Per( \{u^*>t\},\cdot)$ is concentrated in a point $r_t \in I$ satisfying $\partial \{u^*>t\}=\{r_t\}$. We claim  
   \begin{equation}\label{eq:kiave}
       t  \in u^*(D_{u^*}^+) \iff  r_t\in D_{u^*}^+,
   \end{equation}
   and if one of (and thus both)  the above holds, then  $(u^*)^{-1}(\{t\})=\{r_t\}.$ This would clearly imply \eqref{eq:kiave integrale}.  Suppose $r_t \in D^+_{u^*}$. Then  $u^*$ is continuous at $r_t$ and since $\partial \{u^*>t\}=\{r_t\}$  we get $u^*(r_t)=t$, which proves one direction of \eqref{eq:kiave}. Moreover by ii) in Lemma \ref{lem:monotone functions} it holds that $u^*(r) \neq u^*(r_t)$ for all $r\neq r_t$, which implies $ (u^*)^{-1}(t)=\{r_t\}.$ For the other implication assume that $ t  \in u^*(D_{u^*}^+)$, then again by  ii) in Lemma \ref{lem:monotone functions} it holds that $(u^*)^{-1}(t)=\{s\}$ and that $u^*$ is strictly decreasing and continuous at   $s$. Therefore $\{u^*>t\}=(-\infty,s)$ holds and $s=r_t.$    
\end{proof} 
The following will be used to deal with the jump energy when rearranging in the BV space.
\begin{lemma}
 \label{lem:regularization}
Let $\Xdm$ be a metric measure space, let $\varnothing \neq \Omega\subset\X$ be open and let $u \colon \Omega\to \R$ be Borel with $\mm(\{u>t\})<\infty$ for all $t>\essinf u$.   Let $\omega$ be a measure in $\R$ satisfying $\omega \ll \Leb 1 $ and \eqref{eq:mass compatibility},\eqref{eq:zero a sx} and let $u^*$ be the rearrangement of $u$ with respect to  $\omega.$ Then
\begin{equation}
\partial^* \{ u>t\} \cap C_u = \emptyset, \qquad \text{a.e.\ }t \in (\essinf u,\esssup \, u)\setminus u^*(\Omega^*).\label{eq:regularization ii}
\end{equation}
\end{lemma}
\begin{proof}
By monotonicity of $u^*$ on $\Omega^*$, we have a countably (possibly empty) family of disjoint intervals $[\alpha_i,\beta_i)$ so that (cf.\ \cite[Eq. (2.6)]{CianchiFusco02})
\[
 u^*(\Omega^*)=  (\essinf u,\esssup \, u)  \setminus \bigcup_i [\alpha_i,\beta_i).
\]
Thus, by equimeasurability (recall Lemma \ref{lem:basic prop u*}), we have
\begin{equation}\label{eq:zero measure outisde intervals}
    \mm(\{ \alpha_i < u < \beta_i\}) = \omega( \{\alpha_i < u^* <\beta_i \})=0, \qquad\forall \, i.
\end{equation}
To conclude it suffices to show that  $\partial^* \{ u>t\} \cap C_u =\varnothing $ for all $t \in \cup_i  (\alpha_i,\beta_i)$. Suppose by contradiction that for some $i$ and some $t \in (\alpha_i,\beta_i)$ there exists $x \in \partial^* \{ u>t\} \cap C_u$. Then it follows from the definition of essential boundary that $t \in [u^\wedge(x),u^\vee(x)]$ and,  since $u$ is approximately continuous at $x$,  actually $\bar u(x) = t$. 
However by \eqref{eq:zero measure outisde intervals}
\[
|\bar u(x)-u(y)| \ge \min \{ \beta_i - t, t - \alpha_i\}>0,\qquad  \mm\text{-a.e. }y \in \Omega,
\]
 which contradicts the fact that $u$ is approximately continuous at $x$ (recall \eqref{eq:approximate value density}).
\end{proof}
Next, we shall see how the differentiability properties of $\mu$ depend on non-vanishing properties of gradients. This will be key when analyzing rigidity results in this note.
\begin{proposition}\label{prop:non vanishing}
Let $\Xdm$ be a metric measure space, let $\varnothing \neq \Omega\subset\X$ be open and let $u \in BV_{loc}(\Omega)$ be as in Assumption \eqref{eq:ass2} with $\mm(\{u>t\})<\infty$ for all $t>\essinf u$. Let  $(I,\omega)$ be a weighted interval as in Assumption \eqref{eq:ass1} and $u^*$ be the rearrangement of $u$ with respect to  $\omega$. Then, for every $t_2>t_1$, it holds:
\begin{itemize}
    \item[{\rm i)}] $(t_1,t_2) \ni t\mapsto \mu(t)$ is absolute continuous if and only if
    \[
        \omega(D^0_{u^*} \cap \{ t_1<u^* <t_2\}) =0.
    \]
    \item[{\rm ii)}] It holds
    \begin{equation}
        \omega(D^0_{u^*} \cap \{ t_1<u^* <t_2\}) \le \mm(D^0_{u} \cap\{ t_1<u <t_2\}). \label{eq: Duo implies Du*0}
    \end{equation}
\end{itemize}
Furthermore
\begin{equation}
    \int \frac{\nchi_{D^+_u}}{|\underline D u|_1} \, \d \Per(\{u>t\},\cdot) \le -\mu'(t) = \int \frac{\nchi_{D^+_{u^*}}}{|(u^*)'|} \, \d \Per(\{u^*>t\},\cdot), \qquad \text{a.e.\ $t\in(t_1,t_2).$} \label{eq: mu' mu'} 
\end{equation}
Finally, if equality in \eqref{eq: mu' mu'} holds for a.e.\ $t\in(t_1,t_2)$ then equality holds in \eqref{eq: Duo implies Du*0} and, in particular,  if additionally any in \emph{i)} holds then also $ \mm(D^0_{u} \cap \{ t_1<u <t_2\}) =0$.
\end{proposition}
\begin{proof}
We first prove i). Suppose first that $ \omega(D^0_{u^*} \cap \{ t_1<u^* <t_2\})=0$. Using this and \eqref{eq:identity mu} we can write
\[
\mu(t) - \mu(t_2) =\omega(D^0_{u^*} \cap \{ u^*=t_2 \})+ \int_{t}^{t_2} \int \frac{\nchi_{D^+_{u^*}}}{|(u^*)'|} \, \d\Per(\{u^*>s\},\cdot)\,\d s,\qquad \forall \, t \in (t_1,t_2).
\]
In particular, $t\mapsto \mu(t)$ is locally absolutely continuous in $(t_1,t_2)$ and in fact absolutely continuous, being bounded and monotone. Conversely, suppose $(t_1,t_2)\ni t\mapsto \mu(t)$ is absolutely continuous and define $h(t) \coloneqq \omega(D^0_{u^*} \cap \{  t<u^*<t_2 \})$. We can thus write for all $t\in(t_1,t_2)$
\[
 h(t)  + \omega(D^0_{u^*} \cap \{ u^*=t_2 \}) = \mu(t) - \mu(t_2) - \int_{t}^{t_2} \int \frac{\nchi_{D^+_{u^*}}}{|(u^*)'|} \, \d\Per(\{u^*>s\},\cdot)\,\d s,
\]
In particular, $h$ is absolutely continuous in $(t_1,t_2)$ (as it is a difference of two absolutely continuous functions). Moreover applying Lemma \ref{lem:BVdistribution} with $v = {\min\{u^*,t_2\}}$ we get
\begin{equation}\label{eq:h modified zero derivative}
    \frac{\d}{\d t}\omega(D^0_{u^*} \cap \{  t<u^*\le t_2 \})=0, \qquad {\text{a.e.\ $t> t_1$}},
\end{equation}
and, noting that $h(t)$ and $\omega(D^0_{u^*} \cap \{  t<u^*\le t_2 \}$ differ only by a constant independent of $t$, also that $|h'|(t) =0$ for a.e.\ $t \ge 0$. Therefore, $ \lim_{t \downarrow t_1} h(t)= \lim_{t \uparrow t_2} h(t)$ and thus
\[  \omega(D^0_{u^*} \cap \{ t_1<u^* <t_2 \}) = \lim_{t \downarrow t_1} h(t)= \lim_{t \uparrow t_2} h(t) =0,
\]
proving i).

We now prove ii) and \eqref{eq: mu' mu'}. By Lemma \ref{lem:distfunct}, we know
\begin{equation}
\begin{aligned}
    \mm(\{ t<u \le t_2 \}&\cap D^0_u) + \int_t^{t_2} \int \frac{\nchi_{ D^+_u}}{|\underline D u|_1} \, \d \Per(\{u>r\},\cdot)\d r    =\mu(t) -\mu(t_2)\\
    &=\omega(\{ t<u^* \le t_2 \}\cap D^0_{u^*}) + \int_t^{t_2} \int \frac{\nchi_{D^+_{u^*}}}{|(u^*)'|} \, \d \Per(\{u^*>r\},\cdot)\d r,
\end{aligned}\label{eq:mu mu}
\end{equation}
for all $t<t_2$. Being $t \mapsto  \mm(\{ t<u <t_2 \}\cap D^0_u) $ monotone  non-increasing, it is a.e.\ differentiable with $\frac{\d}{\d t} \mm(\{ t<u <t_2 \}\cap D^0_u)  \le 0$ for a.e.\ $t$. Thus, differentiating in \eqref{eq:mu mu} and recalling \eqref{eq:h modified zero derivative},  we get \eqref{eq: mu' mu'}. Integrating \eqref{eq: mu' mu'} between $t$ and $t_2$ and plugging in \eqref{eq:mu mu}  we have
\[
    \omega(D^0_{u^*} \cap \{ t<u^* <t_2 \}) \le \mm(D^0_{u} \cap \{ t<u <t_2 \}), \quad \forall\,  t <t_2,
\]
having used that $\mm(D^0_u \cap \{  u=t_2\}) \overset{\eqref{eq:Du1 local on Cu}}{=} \mm( \{ \bar u=t_2\}) = \omega( \{ u^*=t_2\})\overset{\eqref{eq:Du1 local on Cu}}{=}\omega(D^0_{u^*} \cap \{ u^*=t_2\}) $, the intermediate equality due to equi-measurability (recall ii) in Lemma \ref{lem:basic prop u*}). Then sending $t \to t_1$ we conclude the proof of ii).

We pass now to the last statement. If equality holds in \eqref{eq: mu' mu'} for a.e.\ $t \in (t_1,t_2)$, integrating \eqref{eq: mu' mu'} and plugging in  \eqref{eq:mu mu} gives equality in \eqref{eq: Duo implies Du*0}. This concludes the proof.
\end{proof}
\begin{remark}\label{rem:non vanishing}\rm
    In the previous works \cite{MondinoSemola20,NobiliViolo24}, functional rigidity results were derived under  $|Du|\neq 0$, having relied on a basic version of Lemma \ref{lem:derivative mu} in combination with an approximation argument (cf.\ \cite[Lemma 3.6]{MondinoSemola20}). Notice that, in light of \eqref{eq: Duo implies Du*0}, this is a stronger assumption with respect to the requirement $(u^*)'\neq 0$. For instance, in \cite{NobiliViolo24} an elliptic regularity step was needed to guarantee that $|Du|\neq 0$ for a non-negative distributional solution $u$ of the Euler Lagrange equation arising from Sobolev inequalities. This is one of the main reasons why the arguments of \cite{NobiliViolo24} did not extend to the case of $p$-Sobolev inequality for $p\neq 2$.

    However, we will see in our main rigidity results that equality will hold in \eqref{eq: Duo implies Du*0}, making sufficient to assume $(u^*)'\neq 0$. The latter a priori weaker condition is much easier to be verified in applications, see the proof of Theorem \ref{thm:main noncompact AVR}.    \fr
\end{remark}
\section{P\'olya-Szeg\H{o} inequalities: metric spaces}\label{sec:Polya}
The aim of this section is to produce general rearrangement inequalities \'a la P\'olya-Szeg\H{o} for functions defined on metric measure spaces \emph{assuming} the validity of an isoperimetric principle. Then, building upon this general theory, we will prove our main results in the upcoming parts.
\subsection{Theory for Sobolev functions}
We report an argument from \cite{FuscoCIME08} showing a first qualitative rearrangement principle stating that $u^*$ is locally absolutely continuous whenever there is an underlying isoperimetric inequality and the total variation of $u$ has no singular part.
\begin{lemma}[{{\sf AC}-to-{\sf AC}} property of the rearrangement]\label{lem:u* is AC}
Let $\Xdm$ be a metric measure space, let $\varnothing \neq \Omega\subset\X$ be open and let $u \in BV_{loc}(\Omega)$ be as in Assumption \eqref{eq:ass2} with $\mm(\{u>t\})<\infty$ for all $t>\essinf u$ and satisfying $|\dD^s u|(\Omega)=0$  (the last assumption is always true if  $u \in W^{1,p}_{loc}(\Omega)$). Let  $(I,\omega)$ be a weighted interval as in Assumption \eqref{eq:ass1} and $u^*$ be the rearrangement of $u$ with respect to  $\omega = g\d t\mres I$.
Suppose that $|\dD u|(\{t_1<\bar u<t_2\}) <+\infty$ for  $t_1,t_2\in [\essinf u,\esssup u]$ with $t_1<t_2$  and that it holds for some constant ${\sf C}>0$
\begin{equation}\label{eq:isop in u}
    \Per(\{u>t\},\Omega)\ge {\sf C}\cdot \Ig(\mu(t)), \qquad \text{a.e.\ $t\in(t_1,t_2)$.}
\end{equation}
Then, $u^* \in {\sf AC}_{loc}(r_{t_2},r_{t_1})\subset I$ (for $r_t$ as in \eqref{eq:rt}).  
\end{lemma}
\begin{proof} 
Recall that $u^*\in BV_{loc}(\Omega^*;\omega)$. In particular, \eqref{eq:Per u* g rt} gives that for all $t \in (\essinf u,\esssup u)$ it holds $g(r_t) = \Ig(\mu(t))$ where we recall by \eqref{eq:rt} that $r_t\in \Omega^*$ satisfies $\omega((-\infty,r_t))=\mm(\{u>t\}).$ Since $ \int_{\{t_1<\bar u<t_2\}}|Du|_1\,\d\mm=|\dD u|(\{t_1<\bar u<t_2\})<\infty$, there exists a modulus of continuity  $f \colon [0,\infty) \to [0,\infty)$ with $f(0)=0$ so that
    \[
    \int_{E}|Du|_1\,\d \mm \le f(\mm(E)),
    \]
    for all $E\subset \{t_1<\bar u<t_2\} $  Borel.  Fix an arbitrary  sub interval $[a,b]\subset (r_{t_2},r_{t_1}).$ It is sufficient to show that $u^*$ is absolutely continuous in $[a,b].$ Choose arbitrary pairwise disjoint intervals $(a_i,b_i)\subset[a,b]$ $i=1,\dots,N.$ 
    Since $g$ is positive and continuous in $I$ we have $c\coloneqq \inf_{[a,b]} g>0.$ Note also that,  by monotonicity and by definition of $u^*$,  for all $t \in [u^*(b_i),u^*(a_i)]$ we have $r_t \in [a,b].$
    Using the fine coarea formula of Lemma \ref{lem:fine coarea} with $\varphi = \nchi_{D^+_u}$ (since $D_u^+\subset C_u)$, recalling that $|\dD u|=\nchi_{D^+_u}|Du|_1\mm$ and $\mm(\Omega\setminus C_u)=0$ under the present hypotheses,   we have
    \begin{align*}
        \int_{\{ u^*(b_i)<u< u^*(a_i)\}}|Du|_1\,\d \mm &=  \int_{\{ u^*(b_i)<\bar u< u^*(a_i)\}} \nchi_{ D^+_u}\,\d |\dD u| \\
        &\overset{\eqref{eq:coarea fine}}{=} \int_{ u^*(b_i)}^{u^*(a_i)} \Per(\{u>t\}, D^+_u)\,\d t\\
        &\overset{\eqref{eq:togli Du+}}{=}\int_{ u^*(b_i)}^{u^*(a_i)}\Per(\{u>t\},\Omega)\,\d t \\
        &\overset{\eqref{eq:isop in u}}{\ge}  {\sf C} \int_{ u^*(b_i)}^{u^*(a_i)} g(r_t)\,\d t \ge {\sf C} \, c \big(u^*(a_i)-u^*(b_i)\big).
    \end{align*}
   Moreover, by equimeasurability and since $u^*$ is monotone we have for all $i=1,\dots,N$
    \[
    \mm(\{ u^*(b_i)<u<u^*(a_i)\}) \le \omega ((a_i,b_i))\le D \Leb 1(a_i,b_i),
    \]
    where where $D\coloneqq \sup_{[a,b]} g<\infty.$  
    In particular $\mm(u^{-1}\left(\cup_i (u^*(b_i),u^*(a_i))\right))\le D\Leb 1(\cup_i (a_i,b_i)).$
    Summing up we obtain
    \begin{align*}
          \sum_i \big(u^*(a_i)-u^*(b_i)\big) &\le ({\sf C}\, c)^{-1}  \int_{\cup_i \{ u^*(b_i)<u<u^*(a_i)\}}|Du|_1\,\d \mm \\
          &\le   ({\sf C}\, c)^{-1} f\left(\mm(u^{-1}\left(\cup_i (u^*(b_i),u^*(a_i))\right)\right))) \le  ({\sf C}\, c)^{-1} f\left(D\Leb 1(\cup_i (a_i,b_i))\right).
    \end{align*}
    This concludes the proof, by arbitrariness of $(a_i,b_i)$.
\end{proof}
We now face our key technical result around P\'olya-Szeg\H{o} inequalities.
\begin{theorem}\label{thm:Polya W1p}
Let $\Xdm$ be a metric measure space, let $\varnothing \neq \Omega\subset\X$ be open and let $u \in BV_{loc}(\Omega)$ be as in Assumption \eqref{eq:ass2} with $\mm(\{u>t\})<\infty$ for $t>\essinf u$ and with $|\dD^s u|=0$. Let  $(I,\omega)$ be a weighted interval as in Assumption \eqref{eq:ass1} and $u^*$ be the rearrangement of $u$ with respect to  $\omega= g\d t\mres I$. Then, for every $t_1,t_2 \in [\essinf u, \esssup u]$ with $t_1<t_2$ it holds 
\begin{equation}\label{eq:improved polya rs}
    \int_{\{t_1<u< t_2\}} |D u|_1^p \, \d \mm  \ge \int_{t_1}^{t_2} \left(\frac{\Per( \{u>t\},\Omega)}{\Per(\{u^*>t\})}\right)^p \int |(u^*)'|^{p-1}\,\d\Per(\{ u^*>t\},\cdot) \d t.
\end{equation}
In particular, if $u \in W^{1,p}_{loc}(\Omega)$ and there is ${\sf C}>0$ so that it holds
\begin{equation}
\Per(\{u>t\},\Omega)\ge  {\sf C}\cdot  \Ig(\mu(t)), \qquad \text{a.e.\ $t\in(t_1,t_2)$},
\label{eq:isop improved}
\end{equation}
then
\begin{equation}\label{eq:pz t1t2}
    \int_{\{t_1<u< t_2\}} |D u|_p^p \, \d \mm \ge {\sf C}^p\int_{r_{t_2}}^{r_{t_1}} |(u^*)'|^p \, \d \omega,
\end{equation}
and if the left hand side  is finite then $u^* \in {\sf AC}_{loc}(r_{t_2},r_{r_1})$ (for $r_t$ as in \eqref{eq:rt}). 

Moreover, if \eqref{eq:isop improved} holds and equality occurs in \eqref{eq:pz t1t2}
(with both sides non-zero and finite), then
\begin{equation}
\Per(\{u>t\},\Omega) = {\sf C}\cdot  \Ig(\mu(t)),\qquad \text{a.e.\ } t \in (t_1,t_2),\label{eq:isop polya}
\end{equation}
Finally, if also $(u^*)' \neq 0$ a.e.\ on $(r_{t_2},r_{t_1})$, then the inverse $H\coloneqq (u^*)^{-1} \colon (t_1,t_2) \to (r_{t_2},r_{t_1})$ is locally absolutely continuous and it holds
\begin{equation}\label{eq:new mess fixed}
       |H'(t)||Du|_p\equiv {\sf C},  \qquad \text{a.e.\ }t\in(t_1,t_2)\text{ and }\Per( \{u>t\},\cdot)\text{-a.e.},
   \end{equation}
\begin{equation}\label{eq:du positivo}
    \mm(\{|D u|_1=0\}\cap \{t_1<   u< t_2\})=0.
\end{equation}
\end{theorem}
\begin{proof}
Set $m\coloneqq \esssup u$ and $M\coloneqq \esssup u.$ We subdivide the proof into two parts.

\smallskip

\noindent {\sc Rearrangement inequalities}. Here we prove \eqref{eq:improved polya rs} and \eqref{eq:pz t1t2}. We start with the former. Notice first that the right-hand side of \eqref{eq:improved polya rs} makes sense since $u^* \in BV_{loc}(\Omega^*;\omega)$ and thanks to coarea formula \eqref{eq:coarea fine}, recalling also \eqref{eq:Per u* g rt}. Set now $E \coloneqq \{t \in (m,M) \colon \Per(\{u>t\},\Omega) =0\}$ and notice that it is a Borel set, because $t\mapsto \Per(\{u>t\},\Omega)$ is Borel by the coarea formula. Hence for a.e.\  $t\in (m,M)\setminus E$   we can estimate
\begin{align*}
    \int \nchi_{D_u^+}|D u|_1^{p-1}\,\d\Per( \{u>t\},\cdot) &=  \int \nchi_{D_u^+} \frac{|D u|_1^p}{|\underline{D}u|_1}\,\d\Per( \{u>t\},\cdot)\\
   \text{(by Jensen's inequality)} \qquad   &\ge  \left( \frac{\Per( \{u>t\},D_u^+)}{\int \nchi_{D_u^+}|\underline Du|_1^{-1}\, \d\Per( \{u>t\},\cdot)}\right)^p \int \frac{\nchi_{D_u^+}}{|\underline Du|_1} \, \d\Per( \{u>t\},\cdot) \\
   & \overset{\eqref{eq:togli Du+}}{=}  \frac{\mathcal \Per( \{u>t\},\Omega)^p}{\Big(\int \nchi_{D_u^+}|\underline Du|_1^{-1} \, \d\Per( \{u>t\},\cdot)\Big)^{p-1}}.
\end{align*}
 Let us rewrite \eqref{eq:mu' inequality} here
 \begin{equation}\label{eq:mu' ineq riscritta}
     -\mu'(t) \ge  \int \nchi_{D_u^+}|\underline Du|_1^{-1} \, \d\Per( \{u>t\},\cdot), \qquad \text{a.e.\ }t \in(m,M).
 \end{equation}
Combining the above, we get
\begin{equation}
  \int\nchi_{D_u^+} |D u|^{p-1}_1\,\d\Per( \{u>t\},\cdot) \ge  \frac{\Per( \{u>t\},\Omega)^p}{\big(-\mu'(t)\big)^{p-1}}, \qquad \text{a.e.\ }t \in (m,M)\setminus E.
    \label{eq:key ineq}
\end{equation} 
Note that \eqref{eq:mu' ineq riscritta} with \eqref{eq:togli Du+} also implies
\begin{equation}\label{eq:CHIAVE DI TUTTO}
    \mu'(t)<0,\quad \text{for a.e.\ $t \in (m,M)\setminus E$.}
\end{equation}
We can thus obtain
\begin{equation}\label{eq:pz chain}
    \begin{split}
    \int_{\{t_1<u<t_2\}}|Du|_1^p \,\d\mm  & = \int_{\{t_1<\bar u<t_2\}} \nchi_{D_u^+} |Du|_1^p \,\d\mm \\
    &=  \int_{\{t_1<\bar u< t_2\}}  \nchi_{D_u^+} |\underline Du|_1^{p-1} \,\d |\dD u|\\
    & \overset{\eqref{eq:coarea fine}}{=} \int_{t_1}^{t_2}\int \nchi_{D_u^+} |\underline Du|_1^{p-1}\,\d \Per(\{u>t\},\cdot)\d t\\
    &\ge \int_{(t_1,t_2)\setminus E}\int  \nchi_{D_u^+}  |\underline Du|_1^{p-1}\, \d\Per( \{u>t\},\cdot) \d t  \\
    &\overset{\eqref{eq:key ineq}}{\ge}   \int_{(t_1,t_2)\setminus E }\frac{\Per( \{u>t\},\Omega)^p}{\big(-\mu'(t)\big)^{p-1}}\,\d t,
\end{split}
\end{equation}
(having used  Assumption \eqref{eq:ass2} in the third equality). By Lemma \ref{lem:derivative mu}, we know that
\begin{equation}\label{eq:mu' swag}
    -\mu'(t) = \frac{\Per(\{ u^*>t\})}{|(u^*)'|((u^*)^{-1}(\{t\}))},\qquad\text{a.e.\ }t \in  u^*(D^+_{u^*}).
\end{equation}
Since $\mu'(t) =0$ for a.e.\ $t \in (m,M)\setminus u^*(D^+_{u^*})$ by Lemma \ref{lem:derivative mu}, from \eqref{eq:CHIAVE DI TUTTO} we deduce that up to measure zero sets $   (m,M) \setminus E \subset u^*(D^+_{u^*})$. Hence we can use \eqref{eq:mu' swag} in almost all these points  and by definition of $E$, since $t\mapsto r_t$ is Borel and $\Per(\{u^*>t\},\cdot)$ is concentrated on $r_t$ (cf. \eqref{eq:Per u* g rt}), we get
\begin{align*}
      \int_{(t_1,t_2)\setminus E}\frac{\Per( \{u>t\},\Omega)^p}{\big(-\mu'(t)\big)^{p-1}}\,\d t &= 
      \int_{(t_1,t_2)\setminus E} \left(\frac{\Per( \{u>t\},\Omega)}{\Per( \{u^*>t\})}\right)^p |(u^*)'|^{p-1}(r_t)\Per(\{ u^*>t\}) \, \d t \\
      &=\int_{t_1}^{t_2} \int \left(\frac{\Per( \{u>t\},\Omega)}{\Per( \{u^*>t\})}\right)^p |(u^*)'|^{p-1}\,\d \Per(\{ u^*>t\},\cdot) \d t.
\end{align*}
We thus proved \eqref{eq:improved polya rs}. 

Next, we prove \eqref{eq:pz t1t2}. Suppose  that $u \in W^{1,p}_{loc}(\Omega)$ and \eqref{eq:isop improved} hold. Recall that in this case  $u \in BV_{loc}(\Omega)$ with $|\dD^s u|=0$ and by  \eqref{eq:dependence gradient} also $|Du|_1\le |Du|_p$ $\mm$-a.e.\ on $\Omega$. Thanks to \eqref{eq:isop improved}, we know that $u^* \in {\sf AC}_{loc}(r_{t_2},r_{t_1})$, by Lemma \ref{lem:u* is AC}. Again by \eqref{eq:isop improved}, we have
\begin{equation}\label{eq:E null}
    \text{$E \cap (t_1,t_2)$ has measure zero,}
\end{equation}
(recall $E \coloneqq \{t \in (m,M) \colon \Per(\{u>t\},\Omega) =0\} $) since $\Per( \{u^* >t\})=g(r_t)>0$ for all $t \in (t_1,t_2)$ by \eqref{eq:Per u* g rt} and the assumptions on the weight $g$.
Combining \eqref{eq:improved polya rs} with \eqref{eq:isop improved} and using coarea \eqref{eq:coarea fine} for the right-hand side, we get
\[
\int_{\{t_1<u< t_2\}} |D u|_p^p \, \d \mm \overset{(\ast)}{\ge} \int_{\{t_1<u< t_2\}} |D u|_1^p \, \d \mm \ge {\sf C}^p \int_{r_{t_2}}^{r_{t_1}} | (u^*)'|^{p-1}|Du^*|_1 \, \d \omega.
\]
Recalling the identifications in Lemma \ref{lem:LEMMONE}, we get \eqref{eq:pz t1t2}.

\smallskip

\noindent {\sc Equality case}. Suppose now that equality occurs in \eqref{eq:pz t1t2}. Since we must have  equality in inequality $(*)$, we deduce immediately that
\begin{equation}\label{eq:for future use}
    |Du|_p =|Du|_1,\quad  \mm\text{-a.e.\ on }\{t_1<\bar u<t_2\}.
\end{equation}

Next, since inequality \eqref{eq:pz t1t2} is deduced by combining inequalities \eqref{eq:improved polya rs} and \eqref{eq:isop improved} we deduce that both are actually equalities. In particular \eqref{eq:isop polya} is proved. Note that we also used that
\[
\int |(u^*)'|^{p-1}\,\d \Per(\{ u^*>t\},\cdot) \d t>0, \qquad \text{a.e. } t\in (m,M),  
\]
since $t\in u^*(D_{u^*}^+)$ for a.e.\ $t \in (m,M),$ as $E$ is a negligible set by \eqref{eq:E null}.
We shall now trace all the equalities in the proof of \eqref{eq:improved polya rs} to get the rest of the desired properties. In particular, the last inequality in \eqref{eq:pz chain} is in fact an equality and so  \eqref{eq:key ineq} is an equality for a.e.\ $t \in (t_1,t_2).$ In turn equality in \eqref{eq:pz chain} implies equality in \eqref{eq:mu' ineq riscritta}  and in the application of Jensen's inequality
    \begin{align*}
         & \int \nchi_{D^+_u}\frac{|Du|^p_1}{|\underline{D}u|_1}\,\d\Per( \{u>t\},\cdot)  = \left( \frac{\Per( \{u>t\},\Omega)}{\int \nchi_{D_u^+}|\underline Du|_1^{-1}\, \d\Per( \{u>t\},\cdot)}\right)^p \int \frac{\nchi_{D^+_u}}{|\underline{D}u|_1} \, \d\Per( \{u>t\},\cdot),
    \end{align*}
for a.e.\ $t \in (t_1,t_2)$. By the characterization of the equality case in Jensen's inequality, we obtain that for a.e.\ $t\in(t_1,t_2)$ it holds
   \begin{equation}
       |Du|_1\equiv c_t \qquad \frac{\nchi_{D^+_u}}{|\underline{D}u|_1}\Per( \{u>t\},\cdot)\text{-a.e.},
    \label{eq:ct jensen}
   \end{equation}
   for some constant $c_t> 0$ possibly depending on $t$. Actually, since $\Per( \{u>t\},\Omega \setminus D_u^+)=0$ for a.e.\ $t$ (recall \eqref{eq:togli Du+}), then \eqref{eq:ct jensen} upgrades to
   \begin{equation}
       |Du|_1\equiv c_t,  \qquad \Per( \{u>t\},\cdot)\text{-a.e.}.
    \label{eq:ct}
   \end{equation}
   Hence, equality in \eqref{eq:mu' ineq riscritta} combined with \eqref{eq:ct} gives
   \begin{equation}\label{eq:ct quasi trovato}
       -\mu'(t)=\frac{\Per( \{u>t\},\Omega)}{c_t}\overset{\eqref{eq:isop polya}}{=} {\sf C} \cdot \frac{\Per( \{u^*>t\})}{c_t}, \qquad \text{a.e.\ $t\in(t_1,t_2)$.}
   \end{equation}
   Let us now assume that $(u^*)' \neq 0$ $\omega$-a.e.\ on $(r_{t_2},r_{t_1})$.  Then $u^*$, being also strictly decreasing and locally absolutely continuous,  is invertible in $(r_{t_2},r_{t_1})$  with locally absolutely continuous inverse $H \colon(t_1,t_2)\to  (r_{t_2},r_{t_1}) $. Moreover $(t_1,t_2)\setminus u^*(D_{u^*}^+)$ is negligible as previously observed.
   Therefore \eqref{eq:mu' swag} gives
   \[
   -\mu'(t)=\frac{\Per(\{u^*>t\})}{|(u^*)'(H(t))|}, \qquad \text{a.e.\ $t \in (t_1,t_2).$}
   \]
   Hence, combining the above with \eqref{eq:ct quasi trovato} entails
   \[
   c_t={\sf C}\cdot |(u^*)'(H(t))|, \qquad \text{a.e.\ $t \in (t_1,t_2).$}
   \]
   Finally, since $u^*$ and $H$ are locally absolutely continuous, respectively where they are defined, they preserve $\mathcal L^1$-null sets. Hence preimages of $\mathcal L^1$-null subsets via $H$ are also $\mathcal L^1$-null. Therefore for a.e.\ $t \in (t_1,t_2)$ the function $u^*$ is differentiable at $H(t)$, the function $H$ is differentiable at $t$  and it holds
    \begin{equation}\label{eq:inverse derivative}
     (u^*)'(H(t))H'(t)=(u^*(H(t))'=1.
    \end{equation}
Thus  we deduce
 \[
   c_t=\frac{{\sf C}}{|H'(t)|}, \qquad \text{a.e.\ $t \in (t_1,t_2)$},
   \]
   and so
\begin{equation*}
       |H'(t)||Du|_1\equiv {\sf C},  \qquad \text{a.e.\ }t\in(t_1,t_2)\text{ and }\Per( \{u>t\},\cdot)\text{-a.e.}.
   \end{equation*}
This combined with \eqref{eq:for future use} shows \eqref{eq:new mess fixed}. Finally to show \eqref{eq:du positivo} we recall that we have equality in \eqref{eq:mu' ineq riscritta}, hence the last part of Proposition \ref{prop:non vanishing} gives   $\mm(D^0_{u} \cap \{ t_1<  u <t_2\}) =0$ (recall that we are assuming $\omega(\{(u^*)'=0\}\cap (r_{t_2},r_{t_1}))=0$), which shows \eqref{eq:du positivo} and ends the proof.
\end{proof}

The proof of our first main result then follows from Theorem \ref{thm:Polya W1p} and an approximation argument.
\begin{proof}[Proof of Theorem \ref{thm:main Sobolev PZ metric}]
Assume that $\int_\Omega |D u|_p^p\d \mm <+\infty$, otherwise there is nothing to prove.

From \eqref{eq:MeyersSerrin} there exists a sequence $(u_n)\subset C(\Omega)\cap W^{1,p}_{loc}(\Omega)$ so that  $u-u_n\to 0$ in $L^p(\Omega)$
        and $\int_\Omega |D u_n|_p^p \, \d \mm \to \int_\Omega |Du|_p^p\,\d\mm$.  Up to further cut-off and truncation  we can assume that $\supp(u_n)\subset (\supp (u))^\eps$
and  $\essinf u \le \essinf u_n$ for all $n\in\N$. In particular, for any $t>\essinf u_n$
we have
\begin{equation}\label{eq:mild level bound}
\begin{split}
      \mm(\{u_n>t\})&\le \mm(\{u>t' \})+\mm(\{|u_n-u|>t-t'\})\\
      &\le\mm(\{u>t' \})+|t-t'|^{-p}\|u_n-u\|_{L^p(\Omega)}^p <+\infty,
\end{split}
\end{equation}
for all $t'\in (\essinf u_n, t).$
Moreover $(u_n) \subset  C(\Omega)$ and so each $u_n$ satisfies Assumption \eqref{eq:ass2} (recall Section \ref{sec:key}). Therefore, by \eqref{eq:pz t1t2} in Theorem \ref{thm:Polya W1p} we have  for all $n\in \N$
\begin{equation}
    \int_\Omega |D u_n|_p^p \,\d \mm \ge {\sf C}^p \int_{\Omega^*} |(u_n^*)'|^p\,\d\omega,\label{eq:polya all n}
\end{equation}
with $u_n^*\in {\sf AC}_{loc}(\Omega^*).$ In particular $\sup_n\int_{\Omega^*} |(u_n^*)'|^p\,\d\omega<\infty.$ However $\omega$ is a weight bounded above and away from zero in every compact subset of $\Omega^*$, and so we deduce using the Morrey's inequality that $u_n^*$ are locally uniformly H\"older continuous in $\Omega^*.$ 
Moreover, $u_n^*$ are also locally uniformly bounded, as can be easily checked using \eqref{eq:mild level bound} and the fact that each $u_n^*$ is non-increasing.
Therefore, up passing to a subsequence, the functions $u_n^*$ converge locally uniformly to some $\tilde u\in {\sf AC}_{loc}(\Omega^*).$ Moreover $\tilde u$ is non-increasing as so are $u_n^*.$ Let us now show  that $\tilde u=u^*$ in $\Omega^*$.
As both $u^*$ and $\tilde u$ are continuous and  non-increasing, it is enough to show that $\omega(\{u^*>t\})=\omega(\{\tilde u>t\})$ for all $t>\essinf u^*=\essinf u$ and that $\essinf \tilde u=\essinf u^*$ (see e.g.\ \cite[Lemma 1.1.1]{ciao}). We first observe that $u-u_n\to 0$ in $L^p(\Omega)$ implies that 
\begin{equation}\label{eq:conv of distr}
    \mm(\{u_n>t\})\to \mm(\{u>t\}), \quad \text{for all $t>\essinf u$  satisfying $\mm(\{u=t\})=0$}
\end{equation}
and in particular for a.e.\ $t>\essinf u.$ This also gives $\essinf u_n\to \essinf u$ and so $\essinf u^*=\essinf \tilde u$. Then by \eqref{eq:conv of distr} and the properties of the rearrangement $\omega(\{u_n^*>t\})\to \omega(\{u^*>t\})$ for a.e.\ $t>\essinf u.$
Additionally, since $u_n^*$ converge locally uniformly in $\Omega^*$ to $\tilde u$ and all are non-increasing functions, it is easy to see that $\omega(\{u_n^*>t\})\to \omega(\{\tilde u>t\})$ for all $t>\essinf u^*$ such that $\omega(\{\tilde u=t\})=0.$ This shows that $\omega(\{u^*>t\})=\omega(\{\tilde u>t\})$ for a.e.\ $t>\essinf u$ and so actually for all $t,$ that is what we wanted.

Taking the limit in \eqref{eq:polya all n} by lower semicontinuity recalling the uniform convergence of $u_n^*$ to $u^*$ (e.g.\ arguing as in \cite[Proposition 1.2.7]{GP20}) we obtain \eqref{eq:PZ abstract}.
\end{proof}
\begin{remark}\label{rem: Dirichlet BC 1D}
In Theorem \ref{thm:main Sobolev PZ metric}, if we also assume that $u\ge 0$ and that $u \in W^{1,2}_0(U)$ for some open set $U\subset \Omega$ with $\mm(U)<\mm(\Omega)$, then we additionally have
$$\lim_{t \to b^-} u^*(t)=0,$$
where $b \in I$ is so that $U^*=(a,b)\subset \Omega^*$. Indeed, take $u_n \in \Lip_{bs}(U)$ with $u_n\ge 0$ and such that $u_n\to u$ in $W^{1,p}(\X)$. Since $\mm(U)<\mm(\Omega)\le \omega(\R)$, we deduce that $g(b)>0.$ Then \eqref{eq:PZ abstract} implies that $u_n^*$ are  uniformly H\"older continuous in a neighborhood of $b$ (because $w\ge g(b)/2\d t$ in a neighborhood of $b$). Moreover $\mm(\supp(u_n^*))<\mm(U^*)$ and so trivially $\lim_{t \to b^-} u_n^*(t)=0$ for all $n.$ Up to passing to a subsequence, $u_n^*$ converge uniformly in a neighborhood of $b$ to some function which must be $u^*$ (e.g.\ arguing as in the proof of Theorem \ref{thm:main Sobolev PZ metric} or using the $L^p$-continuity of the rearrangement). Therefore we must have also $\lim_{t \to b^-} u^*(t)=0$. \fr
\end{remark}
We conclude the Sobolev theory by studying further the equality case under general geometric assumptions on the source space that can be roughly summarized as follows: there is an underlying isoperimetric principle whose isoperimetric sets are metric balls with the property of being covered by minimizing geodesics generating from their centre. Additionally, we require the Sobolev-to-Lipschitz property to hold, namely we require that if $u \in W^{1,p}(\X)$ and $|Du|_p \in L^\infty(\X)$, then there is a representative $\tilde u\in \Lip(\X)$ so that $\Lip(\tilde u) = \|Du\|_{L^\infty}$.
\begin{theorem}\label{thm:radiality astratta}
    Under the assumptions of Theorem \ref{thm:main Sobolev PZ metric} with $\Omega=\X$ suppose furthermore that $\Xdm$ { is a PI-space} having the Sobolev-to-Lipschitz property and that equality  in \eqref{eq:isop general} for some {Borel} set $E$ has the following consequences:
    \begin{itemize}
        \item[{\rm i)}]$E$ is equal  up to $\mm$-null sets to a  ball $B_\rho(p)$ for some $p$ and with $\rho\coloneqq {\sf C}^{-1}F_g^{-1}(\mm(E))$;
        \item[{\rm ii)}] 
        If $\rho<\diam (\X)$ then every point in $\overline{B_\rho(p)}$ belongs to the interior of a geodesic starting from $p$ and of length strictly greater than $\rho.$ 
    \end{itemize}
    If $u$ satisfies equality in \eqref{eq:PZ abstract}  and  $|(u^*)'|\neq 0$ a.e.\ on $\{\essinf u< u^* < \esssup u\}$, then $u$ is radial, i.e.\ for some $x_0 \in\X$
    \[
    u=u^*\circ ({\sf  C} \sfd_{x_0}), \quad \mm\text{-a.e.}.
    \]
\end{theorem}
\begin{proof}
 Set  $m\coloneqq \essinf u$ and $M\coloneqq \esssup u.$ 
For brevity, we will write $E\cong F$ to say that $\mm(E \triangle F)=0.$ Recall that assumption \eqref{eq:ass2} for $u$ is automatically satisfied because $\Xdm$ is assumed to be a PI-space.
 We apply Theorem \ref{thm:Polya W1p} with $t_1\coloneqq m$ and $t_2\coloneqq M.$
Then under the current assumptions the rigidity part of Theorem \ref{thm:Polya W1p} gives us the following:
\begin{itemize}
\item[a)]  the set $E=\{u>t\}$ satisfies equality in \eqref{eq:isop general} for a.e.\ $t \in(m,M)$;
\item[b)] the inverse $H\coloneqq (u^*)^{-1} \colon (m,M) \to (r_{t_2},r_{t_1})$ is locally absolutely continuous and satisfies
\begin{equation}\label{eq:new mess fixed in proof}
       |H'(t)||Du|_p\equiv {\sf C},  \qquad \text{a.e.\ }t\in(m,M)\text{ and }\Per( \{u>t\},\cdot)\text{-a.e.}
   \end{equation}
   and
   \begin{equation}\label{eq:du positivo in proof}
   \mm(\{|D u|_1=0\}\cap \{m<   u< M\})=0.
\end{equation}
\end{itemize}
We claim that
\begin{equation}\label{eq:gradient identity avr}
    |H'|(u)|Du|_p={\sf C},\quad \mm\text{-a.e. on } \{m< u<M\}.
\end{equation}
Applying the coarea formula in the version \eqref{eq:coarea fine}, we get
\[
\begin{split}
    \int_{\{t_1< u< t_2\}} \left|  |H'(u)||Du|_1 -{\sf C}\right| |D u|_p \d \mm 
    &= \int_{\{t_1< \bar u< t_2\}\cap C_u} \left|  |H'(\bar u)||Du|_p -{\sf C}\right| |D u|_1 \d \mm \\
    &=\int_{t_1}^{t_2} \int \nchi_{C_u} \left|  |H'(t)||Du|_p -{\sf C}\right| \d \Per( \{u>t\},\cdot) \d t=0,
\end{split}
\]
where in the second step we used that  $\Per(\{u>t\},\cdot)$ is concentrated on the essential boundary $\partial^*\{u>t\}\cap \Omega$ for a.e.\ $t\in(t_1,t_2)$ (by the PI-assumption) and that  at any $x \in  C_u\cap \partial^*\{u>t\}$ we have $\bar u(x) =t$. Thanks to \eqref{eq:du positivo in proof} we conclude \eqref{eq:gradient identity avr}.

We can in fact improve \eqref{eq:gradient identity avr} to every $t\in (m,M)$. Indeed $t\mapsto \mm(\{u>t\})$ is continuous in $(m,M)$ (recall Proposition \ref{prop:non vanishing}) and so $\mm(\{u=t\})=0$ for all $t\in(m,M)$. Moreover for any $t \in(m,M)$ we can take $t_n\to t$ so that $\{u>t_n\}$ satisfies equality in \eqref{eq:isop general}. Hence $\mm(\{u>t_n\}\triangle \{u>t\})\to 0$ as $t_n \to t$, which by lower semicontinuity gives that $\{u>t\}$ also satisfies equality in \eqref{eq:isop general}.

From a) and by assumption i) we deduce that $\{u>t\}\cong B_t\coloneqq B_{\rho_t}(y_t)$  for some $y_t\in \X$ with $\rho_t={\sf C}^{-1}F_g^{-1}(\mu(t))$. In particular $\mu(t)=F_g({\sf C} \rho_t). $ For every $n\in \N$ with $1/n<(M-m)$ set $m_n\coloneqq m+1/n$ and $M_n\coloneqq (M-1/n)\wedge n$ and consider the sequence of functions 
$$u_n\coloneqq m_n\vee u\wedge M_n,$$
with the convention that $(\infty-1/n)\wedge n=n.$
Set $f\coloneqq {\sf C}^{-1}H\circ u_n$.
Next, we claim that 
\begin{quote}
    for all $t\in ({\sf C}^{-1}H(M_n),{\sf C}^{-1} H(m_n))$ we have $\{f<t\}\cong B_{t}(x_t)$ for some $x_t\in \X$ such that, if $t<\diam (\X)$, then every point in $\overline{B_{t}(x_t)}$ belongs to the interior of a geodesic starting from $x_t$ and of length strictly greater than $t.$ 
\end{quote}
 Indeed we have $\{f<t\}\cong \{{\sf C}^{-1}H\circ u_n<t\}\cong \{ u_n>H^{-1}({\sf C} t)\}\cong \{ u>H^{-1}({\sf C} t)\}$. We  observed above that $\{ u>H^{-1}({\sf C} t)\}\cong B_{\rho}(p)$ where $\rho$ satisfies $\mm(\{ u>H^{-1}({\sf C} t)\})=F_g(C \rho )$ and, by assumption ii),  if $\rho<\diam (\X)$, then every point in $\overline{ B_{\rho}(p)}$ belongs to the interior of a geodesic starting from $p$ and of length strictly greater than $\rho.$  Moreover, by  equimeasurability
 \[
   F_g(C \rho )=\mm(\{ u>H^{-1}({\sf C} t)\})=\omega(\{u^*>H^{-1}({\sf C} t)\})=\omega((-\infty,{\sf C} t))=F_g({\sf C}t).
 \]
As $g$ is strictly positive and $F_g$ is strictly increasing in $I$,  we must have $\rho=t$, which proves the above claim. Note that in particular $B_{s}(x_s)\subset B_{t}(x_t) $ up to $\mm$-null sets for all $s<t$. As a consequence  $B_{s}(x_s)\subset \overline{ B_{t}(x_t)}$ (as sets) for all $s<t$, indeed using that balls have positive measures we must have that for all $y \in B_s(x_s)$ and all $r>0$ the ball $B_r(y)$ intersects $B_t(x_t)$.

Next we claim that 
\begin{center}
    $f$ has a 1-Lipschitz representative.
\end{center}
We apply the chain rule in Lemma \ref{lem:local chain rule} with $\Omega=\X$, $I=(m,M)$ and $\phi=H$. Indeed $u$ takes value in $[m_n,M_n]\subset I$ up to $\mm$-null sets and from \eqref{eq:gradient identity avr}  we have $|H'|(u)|Du|_p\in L^p_{loc}(\X).$ Hence we obtain that $H(u) \in  W^{1,p}_{loc}(\X)$ and $|D (H(u))|_p=|H'|(u)|Du|_p={\sf C}$ $\mm$-a.e.\ in $\{m_n< u<M_n \}$. In particular by the locality  it holds $  |Df|_p= \nchi_{\{m_n<u<M_n\}}$ $\mm$-a.e.. The claim now follows from the Sobolev-to-Lipschitz property.

We now go on proving that $x_t$ is in fact independent of $t.$
To this goal let us first observe that 
\begin{equation}\label{eq:salvi ma astratti}
    \partial B_{t}(x_t)\subset \partial \{f<t\}\subset \{f=t\}, \qquad \forall\, t\in ({\sf C}^{-1}H(M_n),{\sf C}^{-1}H(m_n)).
\end{equation}
To see this note that for any such $t$ it holds $t<\diam(\X)$ indeed $\mm(B_t(x_t))=\mm(\{f<t\})\le \mm(\{m_n<u<M_n\})<\mm(\X)$. Hence $t<\diam(\X)$. In particular every point $y \in \partial B_{t}(x_t)$ belongs to the interior of a geodesic starting from $x_t$. From this we deduce that $\mm(B_r(y)\cap B_{t}(x_t))>0$ and $\mm(B_r(y)\setminus B_{t}(x_t))>0$ for all $r>0$. However $B_{t}(x_t)\cong \{f<t\}$ and so $\mm(B_r(y)\cap \{f<t\})>0$ and $\mm(B_r(y)\setminus  \{f\le t\})>0$ for all $r>0$ as well. This proves \eqref{eq:salvi ma astratti}. 

Assume now by contradiction that $x_t\neq x_{\bar t}$ for some $t,\bar t \in ({\sf C}^{-1}H(M_n),{\sf C}^{-1}H(m_n))$ with $\bar t<t.$ Set $\delta\coloneqq \sfd(x_t,x_{\bar t})>0.$ There is a geodesic $\gamma:[0,a)\to \X$, $a \in [0,\infty]$, with $\gamma_0=x_t$, $\gamma_{\delta}=x_{\bar t}$ and  $x\coloneqq \gamma_{t} \in \partial B_{t}(x_t)$. Hence $y\coloneqq \gamma_{\delta +\bar t}\in B_{{\bar t}}(x_{\bar t})$. Moreover, since $B_{\bar t}(x_{\bar t})\subset \overline{B_{t}(x_t)}$ we have  $t\ge \bar t+\delta .$
Then by \eqref{eq:salvi ma astratti} we have $f(x)=t$ and $f(y)=\bar t$ and since $f$ is 1-Lipschitz
\[
t-\bar t=f(x)-f(y)\le \sfd(x,y)=|t-(\bar t+\delta )|=t-(\bar t+\delta )=t-\bar t-\delta< t-\bar t,
\]
which is a contradiction. Hence $x_t=x_0$ for all $t \in ({\sf C}^{-1}H(M_n),{\sf C}^{-1}H(m_n))$ and for some $x_0 \in \X.$
Combining everything we obtain $\{f<t\}\cong B_{t}(x_0)$ for all $t \in ({\sf C}^{-1}H(M_n),{\sf C}^{-1}H(m_n))$. This shows that   $f=\sfd_{x_0}$  $\mm$-a.e.\ in $ \{{\sf C}^{-1}H(M_n)<f<{\sf C}^{-1}H(m_n)\}$ (actually everywhere by continuity) and by the definition of $f$
\[
    u= u^*\circ \big({\sf C}\sfd_{x_0}\big),\qquad \text{$\mm$-a.e. in }  \{m_n<u<M_n\},
\]
having used that $H^{-1}=u^*$ in $(H(M),H(m))$ and the fact that $H$ sends null sets to null sets.
By the arbitrariness of $n\in\N$ we deduce that the above also holds $\mm$-a.e.\ in $ \{m<u<M\}$.

To conclude the proof, observe that if $M<\infty$ and  $v\coloneqq \mm(\{u=M\})>0$ we also need to check that $u=u^*\circ \big({\sf C} \sfd_{x_0}\big)$ $\mm$-a.e.\ in $\{u=M\}.$ By definition of rearrangement we have $u^*=M$ in $(-\infty,F_g^{-1}(v))$. Moreover by the above discussion  it holds $\{u=M\}\cong \cap_{t<M} B_{\rho_t}(x_0)\cong B_r(x_0)$ with $r={\sf C}^{-1}F_g^{-1}(v)$. Combining these facts we have $u^*\circ\big({\sf C} \sfd_{x_0}\big)=M$  in $B_r(x_0)$, which is what we wanted. A similar argument shows that  $u^*\circ\big({\sf C} \sfd_{x_0}\big)=m$ at $\mm$-a.e.\ point in $\{ u=m\}.$
\end{proof}
In the previous proof, we used the following version of chain rule for Sobolev functions. We include a proof as we could not find it in the literature.\begin{lemma}[Chain rule]\label{lem:local chain rule}
    Let $\Xdm$ be a metric measure space, let $ \varnothing\neq \Omega\subset\X$ be open and $I\subset \R$ be an open interval (possibly unbounded). Let $\varphi \colon I\to \R$ be absolutely continuous (in particular $|\phi'|\in L^1(I)$)  and fix $p\in (1,\infty)$.  Assume that $u \in W^{1,p}_{loc}(\Omega)$ is so that $u$ takes values in $I$ up to $\mm$-negligible sets and $|\varphi'|(u)|Du|_p \in L^p_{loc}(\Omega)$.  Then, we have  $\varphi \circ u \in W^{1,p}_{loc}(\Omega)$ and
    \begin{align*}
        |D(\varphi \circ u)|_p = |\varphi '|(u)|Du|_p,\qquad\mm\text{-a.e.\ in }\Omega,
    \end{align*}
    where $|\varphi '|(u)$ is arbitrarily defined on $u^{-1}(\{t \in I \colon \nexists \varphi'(t)\}$. 
\end{lemma}
\begin{proof}
First notice that the conclusion makes sense because $|Du|_p=0$ $\mm$-a.e.\ on  $u^{-1}(\{t \in I \colon \nexists \varphi'(t)\}$, thanks to the locality of the $p$-minimal weak upper gradient.

 Up to adding a constant to both $u$ and to $\phi$, we can assume that $0 \in I$ and $\varphi(0)=0$.
 Consider the function $\phi_n:\R\to \R$
\[
\varphi_n(t) \coloneqq \int_0^t(-n) \vee \varphi'(s)\wedge n \,\d s,
\]
where $\phi'(s)$ is taken to be 0 when $s\in \R\setminus I.$
Then $\varphi_n \in \Lip(\R)$ with $\varphi_n(0)=0$ and $\|\phi_n\|_{L^\infty(\R)} \le \|\phi' \|_{L^1(I)}.$ Moreover by dominated convergence $\phi_n \to \phi$ pointwise in $I$.
 
 Let $x \in \Omega$ be arbitrary. Since $u \in W^{1,p}_{loc}(\Omega)$ there exists $\bar u \in W^{1,p}(\X)$ and a bounded open  neighborhood $U_x\subset \Omega$  of $x$ such that $u\restr {U_x}=\bar u$. Moreover, since by assumption $|\phi'|(u)|D u|_p\in L^p_{loc}(\Omega),$ up to intersecting $U_x$ with another neighborhood of $x,$ we can assume that $|\phi'|(u)|D u|_p\in L^p(U_x)$. 
 By the usual chain rule with Lipschitz functions, we have
 $\varphi_n \circ \bar  u \in W^{1,p}(\X)$ and 
 \begin{equation}\label{eq:wug bound phi'}
     |D(\varphi_n\circ \bar u)|_p= |\varphi_n'|(\bar u)|D  \bar u|_p = |\varphi_n'|(u)|D   u|_p\le  |\varphi'|(u)|D   u|_p,\qquad \mm\text{-a.e.\ in $U_x$},
 \end{equation}
where we used that $|\phi'_n|\le |\phi'|$ pointwise a.e.\ in $I.$
Fix now $\eta \in \Lip(\X)$ with $0\le \eta \le 1$, $\supp(\eta)\subset U_x$ and such that $\eta=1$ in $B_{r}(x)$ for some $r>0$. Then $\eta \phi_n(\bar u)\in W^{1,p}(\X)$ and by the above identity
\[
|D(\eta \phi_n(\bar u))|_p\le \nchi_{U_x} \left(\Lip(\eta) \|\phi_n\|_{L^\infty(\R)}+ |\varphi'|(u)|D   u|_p \right)\qquad \mm\text{-a.e.\ in $U_x$}.
\]
Note that the right hand side is uniformly bounded in $L^p(\X)$ by the choice of $U_x$ and since $ \|\phi_n\|_{L^\infty(\R)}\le \|\phi' \|_{L^1(I)}$.
 Moreover $\eta \phi_n(\bar u)=\eta \phi_n( u)$ $\mm$-a.e.\ and $\eta \phi_n( u)\to \eta \phi( u)$ $\mm$-a.e.. Hence by lower semicontinuity we get that $\eta \phi( u)\in W^{1,p}(\X)$ and  by \eqref{eq:wug bound phi'} that 
 \[
 |D(\eta \phi (u))|_p\le |\phi'|(u)|D u|_p\qquad \mm\text{-a.e.\ in $B_r(x)$}.
 \]
 Hence by the arbitrariness of $x$ we obtain that $\phi(u)\in W^{1,p}_{loc}(\X)$ with $|D(\phi (u))|_p\le |\phi'|(u)|D u|_p$ $\mm$-a.e.\ in $\Omega.$ To show the actual equality we adapt the argument in  \cite[Theorem 2.1.28]{GP20}. For any $k \in \N$ consider the function $\psi(t)\coloneqq kt-\varphi(t)$. Then, the previous part still applies giving that $\psi(u)\in W^{1,p}_{loc}(\Omega)$ and thus 
\begin{align*}
    k|D u|_p=|D (ku)|_p&\le  |D (\varphi \circ u)|_p+|D (\psi \circ u)|_p\le (|\varphi'|(u)+|\psi '|(u))|D u|_p=k|D u|_p,
\end{align*}
holds $\mm$-a.e.\ in $\{0\le \varphi'(u)\le k\}$. This shows that 
\[
|D (\varphi \circ u)|_p=|\varphi '|(u)|D u|_p, \qquad \text{$\mm$-a.e.\ in $\{0\le \varphi'(u)\le k\}$}.
\]
Taking $\psi(t)\coloneqq -kt-\varphi(t)$ the same argument shows the same identity in $\{-k\le \varphi'(u)\le 0\}$. From the arbitrariness of $k$ we get the desired conclusion.
\end{proof}
\subsection{Theory for functions of bounded variation}
Here we extend our rearrangement inequalities to BV functions and also prove that the rearrangement decreases the energy carried by the singular part and the energy carried by the jump set.
\begin{theorem}\label{thm:polya BV}
Let $\Xdm$ be a metric measure space, let $\varnothing \neq \Omega\subset\X$ be open and let $u \in BV_{loc}(\Omega)$ be as in Assumption \eqref{eq:ass2} with $\mm(\{u>t\})<\infty$ for all $t>\essinf u$. Let  $(I,\omega)$ be a weighted interval as in Assumption \eqref{eq:ass1} and $u^*$ be the rearrangement of $u$ with respect to  $\omega= g\d t\mres I$. Suppose  
\begin{equation}\label{eq:isop in u bv}
    \Per(\{u>t\},\Omega)\ge {\sf C}\cdot \Ig(\mu(t)), \qquad \text{a.e.\ }t \in (\essinf u,\esssup u),
\end{equation}
holds for some ${\sf C}>0$. Then, it holds
\[
    |\dD u|(\Omega) \ge  {\sf C}\int_{\Omega^*}g\,\d TV(u^*), \qquad |\dD^s u|(\Omega)\ge  {\sf C} \int_{\Omega^*}g\,\d TV^s(u^*), 
\]
and, if equality holds in the first of the above (with both sides non-zero and finite), then
\begin{equation}
    \Per(\{u>t\},\Omega) = {\sf C}\cdot \Ig(\mu(t)),\qquad \text{a.e.\ }t \in (\essinf u,\esssup u).\label{eq:isop BV polya}
\end{equation}
Finally, if $\X$ is a PI-space, then the above conclusions hold for all $u \in BV_{loc}(\Omega)$ and we also have
\[
|\dD u|(J_u)\ge {\sf C} \int_{\Omega^*}g\,\d TV(u^*)\mres{J_{u^*}}.
\]
\end{theorem}
\begin{proof}
Set $m \coloneqq \essinf u,M\coloneqq  \esssup u$. Recall that $u^* \in BV_{loc}(\Omega^*;\omega)$ and, thanks to the identification in Lemma \ref{lem:LEMMONE} and the Radon-Nikodym decomposition of $TV(u^*)$, we can prove equivalently the three claimed inequalities with the metric total variation $|\dD u^*|$ on the right-hand side.  

Fix  $u \in BV_{loc}(\Omega)$ satisfying Assumption \ref{eq:ass2}. Using the standard coarea formula and \eqref{eq:isop in u bv}, we get
\[
|\dD u|(\Omega) = \int_m^M\Per(\{u>t\},\Omega)\, \d t \ge  {\sf C} \int_m^M \Ig(\mu(t)) \,\d t \overset{\eqref{eq:Per u* g rt}}{=} {\sf C} \int_m^M \Per(\{ u^*>t\})\,\d t,
\]
The first claimed inequality then follows by application of the coarea formula on the right-hand side. In particular, conclusion \eqref{eq:isop BV polya} is immediate if equality occurs in the above.

We pass to the proof of the second claimed inequality. Recalling  that $|\dD^s u|$ is concentrated on $\Omega \setminus D^+_u$ (by definition of $D^+_u,D^0_u$) and that $u^*(D^+_{u^*})$ is Borel (see Lemma \ref{lem:monotone functions}), we estimate 
\begin{align*}
|\dD^s u|(\Omega) & = |\dD u|(\Omega\setminus D^+_u) = \int_m^M \Per(\{u>t\},\Omega\setminus  D^+_u)\, \d t \\
&= \int_{(m,M)\setminus u^*(D^+_{u^*})}\Per(\{u>t\},\Omega\setminus D^+_u)\, \d t + \int_{u^*(D^+_{u^*})}\Per(\{u>t\},\Omega\setminus  D^+_u )\, \d t \\
&\ge \int_{(m,M)\setminus u^*(D^+_{u^*})}\Per(\{u>t\},\Omega\setminus  D^+_u )\, \d t.
\end{align*}
We notice now the following important fact
\[
\int_{(m,M)\setminus u^*(D^+_{u^*})}\Per(\{u>t\},\Omega\setminus  D^+_u )\, \d t = 
\int_{(m,M)\setminus u^*(D^+_{u^*})}\Per(\{u>t\},\Omega)\,\d t.
\]
Indeed, notice 
\[
0 \overset{\eqref{eq: mu' mu'}}{\ge} \int \frac{\nchi_{ D_u^+}}{|\underline D u|_1}\, \d \Per(\{u>t\},\cdot),\qquad \text{a.e.\ } t \in (m,M)\setminus u^*(D^+_{u^*})\text{ so that }\exists \mu'(t)=0.
\]
Since by Lemma \ref{lem:derivative mu} we know that $\mu'(t) =0 $ for a.e.\ $t \in (m,M)\setminus u^*(D^+_{u^*})$, it follows that $\Per(\{u>t\}, D^+_u)=0$ a.e.\ on $(m,M)\setminus u^*(D^+_{u^*})$ as desired.  Therefore, by  \eqref{eq:isop in u bv} we reach
\[
|\dD^s u|(\Omega) \ge  {\sf C} \int_{(m,M)\setminus u^*(D^+_{u^*})}\Per(\{u^*>t\})\, \d t.
\]
Using coarea this time for $u^*$ (recall $ |\dD^s u^*|(\Omega^*) =  |\dD u^*|(\Omega^*\setminus D^+_{u^*})$ from Remark \ref{rmk:nobady cares}) yields 
\begin{align*}
    |\dD^s u^*|(\Omegastar) &=  \int_{(m,M)\setminus u^*(D^+_{u^*})}\Per(\{u^*>t\},\Omegastar\setminus D^+_{u^*})\, \d t + \int_{u^*(D^+_{u^*})}\Per(\{u^*>t\},\Omegastar\setminus D^+_{u^*})\, \d t \\
    &\le \int_{(m,M)\setminus u^*(D^+_{u^*})}\Per(\{u^*>t\})\, \d t + \int_{u^*(D^+_{u^*})}\Per(\{u^*>t\}, \Omegastar \setminus D^+_{u^*})\, \d t.
\end{align*}
We see now that the second claimed inequality in the statement would follow by the combination of what it has been achieved so far, provided 
\[
\int_{u^*(D^+_{u^*})}\Per(\{u^*>t\},\Omegastar\setminus D^+_{u^*})\, \d t=0.
\]
However, for a.e.\ $t \in u^*(D^+_{u^*})$, we know by Lemma \ref{lem:derivative mu} that $(u^*)^{-1}(\{t\})$ is a singleton  belonging to $D^+_{u^*}$. Hence, by \eqref{eq:perimeter of halfline}, we get $\Per(\{u^*>t\},\Omegastar\setminus D^+_{u^*}) = g\delta_{(u^*)^{-1}(\{t\})}(\Omegastar\setminus D^+_{u^*}) =0$ at any such $t$.

From here, we shall assume that $\X$ is a PI-space and prove the claimed inequality for the jump set. As discussed in Section \ref{sec:key}, in this case any $u \in BV_{loc}(\Omega)$ satisfies Assumption \eqref{eq:ass2}, implied by the fact that $\Per(E,\cdot)$ is concentrated on $\partial^*E \cap \Omega$, if $\nchi_E \in BV_{loc}(\Omega)$. In particular, all the previous steps hold for any $u \in BV_{loc}(\Omega)$. Let us start by arguing as previously done to get
\begin{align*}
 |\dD u|(J_u) &= \int_{(m,M)\setminus u^*(\Omegastar)} \Per(\{u>t\},J_u)\,\d t  + \int_{u^*(\Omegastar)} \Per(\{u>t\} ,J_u)\,\d t \\
&\ge \int_{(m,M)\setminus u^*(\Omegastar)} \Per(\{u>t\},\Omega)\,\d t,
\end{align*}
having used (thanks to $\X$ being PI) that $\Per(\{u>t\},\cdot)$ is concentrated on $\partial^*\{u>t\}\cap \Omega$, that $\partial^*\{u>t\}\cap J_u = \partial^*\{u>t\}\cap \Omega \setminus C_u$ and finally \eqref{eq:regularization ii} to get $\Per(\{u>t\} , \Omega \setminus C_u) = \Per(\{u>t\},\Omega)$ for a.e.\ $t \in (m,M)\setminus u^*(\Omegastar)$. By coarea formula, we also have
\begin{align*}
    |\dD u^*|(J_{u^*}) \le   \int_{(m,M)\setminus u^*(\Omegastar)} \Per(\{u^*>t\})\,\d t + \int_{ u^*(\Omegastar)}\Per(\{u^*>t\},J_{u^*})\,\d t.
\end{align*}
However, the second integral vanishes since $u^*(C_{u^*})\cup u^*(J_{u^*}) = u^*(\Omegastar),$ $u^*(J_{u^*})$ is at most countable and if $t \in u^*(C_{u^*})$ then $\partial^*\{u^*>t\} \subset C_{u^*}$. Using as before \eqref{eq:isop in u bv}, we get the third inequality. 
\end{proof}
\begin{remark}
    Notice that the first rearrangement inequality as well \eqref{eq:isop BV polya} hold for any $u \in BV_{loc}(\Omega)$ even without Assumption \eqref{eq:ass2}, having only relied on the standard coarea formula. \fr
\end{remark}
\section{P\'olya-Szeg\H{o} inequalities: Ricci lower bounds}\label{sec:Polya Ricci}
In this part, we specialize our rearrangement theory on spaces with Ricci curvature lower bounds and prove our main results for this part. We face each setting separately.
\subsection{Positive Ricci lower bound}
Let us start by recalling the following sharp and rigid isoperimetric inequality in this setting. For any $N \in [1,\infty)$, recall the isoperimetric profile
\[
\mathcal I_{N-1,N}(v) \coloneqq c_N \sin^{N-1}(b_v),\qquad \forall \, v \in (0,1),
\]
where $c_N =  \big(\int_0^\pi \sin^{N-1}(t)\,\d t \big)^{-1}$ and $b_v \in [0,\pi]$ is so that $c_N \int_0^{b_v}\sin^{N-1}(t)\,\d t = v.$ Note  that $\mathcal I_{N-1,N}=\Ig$ for $g=c_N \sin^{N-1}$ as defined in \eqref{eq:profile}.
\begin{theorem}[\cite{CavallettiMondino17-Inv,CavallettiMaggiMondino19}]\label{thm:levy gromov cavamondino}
    Let $\Xdm$ be an essentially non-branching ${\sf CD}(N-1,N)$ space for some $N\in (1,\infty)$ with $\mm(\X)=1$. Then, it holds
    \begin{equation}\label{eq:Levy gromov positive curv}
          \Per(E)\ge  \mathcal I_{N-1,N}(\mm(E)),\qquad \forall \, E\subset\X \text{ Borel}.
    \end{equation}
Moreover, if equality holds in \eqref{eq:Levy gromov positive curv} then:
\begin{itemize}
    \item[{\rm i)}]  $E$ is  a metric ball $B_\rho(p)$ up to $\mm$-negligible sets, where  $\mm(B_\rho(p))= \int_0^\rho c_N \sin^{N-1}(t)\d t$ ;
    \item[{\rm ii)}] $\X$ is covered by geodesics of length $\pi$ starting from $p$;
    \item[{\rm iii)}] if $\X$ is an ${\sf RCD}(N-1,N)$ space, then it is isomorphic to a spherical suspension and $p$ is a tip.
\end{itemize}
\end{theorem}
\begin{proof}
Inequality  \eqref{eq:Levy gromov positive curv} and item iii) are proved in \cite{CavallettiMondino17-Inv}, while i) is a consequence of \cite{CavallettiMaggiMondino19}.     Item ii) is also essentially contained in the (much more general) results in \cite{CavallettiMaggiMondino19}. For the convenience of the reader, we add further details. The argument in \cite{CavallettiMaggiMondino19} starts considering a suitable disintegration  of the measure $\mm$ as
\begin{equation}\label{eq:disint}
    \mm=\int_Q \int_{X_q} \mm_q\, \d\mathfrak{q},
\end{equation}
where $\mathfrak{q}$ is a probability measure on some set $Q$ and $\mm_q$ are probability measures supported on (possibly degenerate) closed geodesics $X_q\subset \X.$ From statement \cite[(5.5)]{CavallettiMaggiMondino19} it follows that, if equality holds in \eqref{eq:Levy gromov positive curv} then there exists $Q_l^g\subset Q$ with $\mathfrak{q}(Q\setminus Q_l^g)=0$ and such that the geodesics $\{X_q\}_{q \in Q_l^g}$ are all of length $\pi$ and have common endpoints $\bar q,\bar p\in \X.$ From formula \eqref{eq:disint} it then follows  that $\{X_q\}_{q \in Q_l^g}$ cover $\X$ up to $\mm$-null sets. Moreover from the very beginning of the proof of \cite[Theorem 1.1]{CavallettiMaggiMondino19} we see that we can take $p=\bar q$ in item i). This proves that  $\X$ is covered up to $\mm$-null sets by geodesics of length $\pi$ starting from $p$, from which item ii) follows by a compactness argument.
\end{proof}
Actually, by combining the full analysis with diameter dependence in \cite{CavallettiMondino17-Inv} with \cite[Lemma 3.2]{CavallettiMondinoSemola23}, the following improvement holds
\begin{equation}
 \Per(E)\ge  {\sf BBG}_N(\diam(\X)) \cdot \mathcal I_{N-1,N}(\mm(E)),
\label{eq:LevyGromov BBG}
\end{equation}
where
\begin{equation}\label{eq:BBG constant}
    {\sf BBG}_N(D) \coloneqq \left(\frac{\int_0^{\frac \pi 2} \cos^{N-1}(t) \,\d t }{\int_0^{D/2} \cos^{N-1}(t) \,\d t  }\right)^{\frac 1N}\ge 1,\qquad \forall \, D \in (0,\pi],
\end{equation}
is the sharp Berard-Besson-Gallot lower bound \cite{BerardBessonGallot85}. Recall also e.g.\ from \cite[Lemma 3.3]{CavallettiMondinoSemola23} that there exist constants $C_N >0$ so that 
\begin{equation}
    {\sf BBG}_N(D)^2 -1\ge C_N\left( \pi -D\right)^N,\qquad \forall \, D \in (0,\pi].\label{eq:BBG quantitativo}
\end{equation}
Actually, in \cite[Lemma 3.3]{CavallettiMondinoSemola23} the existence of some $C_N>0$ for the above to hold is shown in the range $D \in (D_N,\pi]$ for some $D_N \in (0,\pi)$. However, as ${\sf BBG}_N(D) \to +\infty$ as $D\downarrow 0$ and $\cos^{N-1}$ is strictly positive in $(0,\pi/2)$, a simple compactness argument shows that, for a possibly smaller constant $C_N>0$, this holds for all $D \in (0,\pi]$.

In the above, for any $N \in (1,\infty)$ the $N$-spherical suspension over a metric measure space $(Y,\mm_Y,\sfd_Y)$ is defined  to be
$([0,\pi]\times_{\sin}^N Y)\coloneqq Y\times \left[0,\pi \right]/(Y\times \left\{0,\pi\right\})$ endowed with the distance $\sfd((t,y),(s,y'))\coloneqq \cos^{-1}\big(\cos(s)\cos(t)+\sin(s)\sin(t)\cos\left({\sfd}_Y(y,y')\wedge\pi\right)\big)$ and measure $\mm\coloneqq \sin^{N-1}(t) \d t\,\otimes\, \mm_Y.$ The points $\{0\}\times Y,\{\pi\}\times Y$ are called tips of the suspension.
\begin{proof}[Proof of Theorem \ref{thm:main polya compact}]
    Recall that the isoperimetry of Theorem \ref{thm:levy gromov cavamondino}, in the form \eqref{eq:LevyGromov BBG}, reads
    \[
    \Per(E) \ge {\sf BBG}_N(\diam(\X)) \cdot \mathcal I^\flat_g(\mm(E)) ,\qquad \forall \, E \subset \X\text{ Borel},
    \]
    where $g = c_N \sin^{N-1}$.  Hence, by considering $ \Omega=\X, I = (0,\pi)$, then $\omega =\mm_{N-1,N} = g\Leb 1\mres I$ and ${\sf C} \coloneqq {\sf BBG}_N(\diam(\X)) $ we see that Theorem \ref{thm:main Sobolev PZ metric} as Theorem \ref{thm:polya BV} here apply (recall that $\X$ a PI-space by \cite{Sturm06I,Sturm06II} and \cite{Rajala12-2}). This concludes the proof of i) and ii).

    We are therefore left to prove the rigidity parts of the statement and we start considering $u \in W^{1,p}(\X)$ satisfying \eqref{eq: equality PZ compact}. Since ${\sf BBG}_N(D)\ge 1$, we directly get
    \[
        {\sf BBG}_N(D) =1, \quad \text{ and thus by \eqref{eq:BBG quantitativo} also }  \diam(\X) =\pi.
    \]  
    In particular, if $\X$ is also assumed an ${\sf RCD}$ space, then the maximal diameter theorem  \cite{Ketterer13} gives directly that $\X$ is isomorphic to a spherical suspension. 
    
    Assume now (that $\X$ is an essentially non-branching ${\sf CD}(N-1,N)$  space, see \cite{RajalaSturm12}, and) that  $(u^*)' \neq 0$ a.e.\ on $\{\essinf u<u^*<\esssup u\}$. Then, $u$ is deduced radial by application of Theorem \ref{thm:radiality astratta} recalling also that $\X$ has the Soboleb-to-Lipschitz property (cf.\ \cite[Section 4.1.3]{Gigli13_splitting}) and Theorem \ref{thm:levy gromov cavamondino}.
\end{proof}
\subsection{Non-negative Ricci curvature and Euclidean volume growth}
We start recalling the available isoperimetry in this setting. On a ${\sf CD}(0,N)$ space $\Xdm$ the \emph{asymptotic volume ratio}
\begin{equation}
    {\sf AVR}(\X) \coloneqq \lim_{r\to\infty} \frac{\mm(B_r(x))}{\omega_Nr^N} \in [0,\infty),
\end{equation}
is well defined and independent on $x \in \X$ by Bishop-Gromov monotonicity. A Euclidean sharp isoperimetric inequality  holds \cite{BaloghKristaly21} if ${\sf AVR}(\X) >0$ (in smooth-setting previously derived by \cite{Brendle20,AgostinianiFogagnoloMazzieri20,FogagnoloMazzieri22, Johne}, see also \cite{CavallettiManini22-isoMCP}).
\begin{theorem}\label{thm:isoperimetric avr}
Let $\Xdm$ be an ${\sf CD}(0,N)$ space with $N\in(1,\infty),{\sf AVR}(\X)>0$. Then
\begin{equation}
     \Per(E) \ge N({\sf AVR}(\X)\omega_N)^{\frac 1N}\mm(E)^{\frac{N-1}N}, \qquad \forall \, E\subset \X \text{ Borel, $\mm(E)<+\infty$}.\label{eq:isoperimetry AVR}
\end{equation}
\end{theorem}
Here $ \omega_N\coloneqq \pi^{N/2} \Gamma^{-1}\left(N/2+1\right)$ and we notice that $\Ig(v) =  N({\sf AVR}(\X)\omega_N)^{\frac 1N}v^{\frac{N-1}{N}} $ for $v>0$ and $g(t) = N\omega_N t^{N-1}$, as defined in \eqref{eq:profile}. Let us discuss the rigidity in \eqref{eq:isoperimetry AVR}. This has been proved in \cite{AntonelliPasqualettoPozzettaSemola22} under the noncollapsed assumption extended to the current setting by \cite[Theorem 1.4]{CavallettiManini22}.
\begin{theorem}\label{thm:rigidity ISOAVR}
Let $\Xdm$ be an essentially non-branching ${\sf CD}(0,N)$ space  with $N\in(1,\infty)$ and ${\sf AVR}(\X)>0$. 
If equality holds in \eqref{eq:isoperimetry AVR} for some bounded Borel set $E\subset \X$ with $\mm(E)<\infty$ then:
\begin{itemize}
    \item[{\rm i)}]  $E$ is  a metric ball $B_{\rho}(p)$ up to $\mm$-negligible sets, with $\mm(E)=\mm(B_{\rho}(p))=\omega_N  {\sf AVR}(\X)\rho^N;$
    \item[{\rm ii)}] $\X$ is covered by geodesics rays starting from $p$;
    \item[{\rm iii)}] if $\X$ is an $\RCD(0,N)$ space, then it is isomorphic to an $N$-Euclidean cone and $p$ is a tip.
\end{itemize}
\end{theorem}
Recall the an $N$-Euclidean  cone  over a metric measure space $(Y,\mm_Y,\sfd_Y)$ is defined  to be the space $Y\times [0,\infty)/(Y\times \{0\})$ endowed with the distance $\sfd((t,y),(s,y'))\coloneqq \sqrt{t^2+s^2-2st\cos(\sfd_Y(y,y')\wedge \pi)}$ and the measure $\mm\coloneqq t^{N-1} \d t\,\otimes\, \mm_Y$, for $N>1$. The point $Y \times \{0\}$ is called tip of the cone. If only $\mm\coloneqq t^{N-1} \d t\,\otimes\, \mm_Y$  is assumed we speak of an $N$-volume cone.
\begin{remark}\label{topological reg avr}
    \rm  
    As an outcome of the topological regularity for isoperimetric sets, {the boundedness assumption on $E$ in Theorem \ref{thm:rigidity ISOAVR} can be removed, see \cite[Theorem 1.3]{PasqualettoRajala25} as well as the earlier \cite{APPV23}.} \fr
\end{remark}
\begin{proof}[Proof of Theorem \ref{thm:main polya noncompact}]
Recall that the isoperimetry in Theorem \ref{thm:isoperimetric avr} reads as
\begin{equation}\label{eq:isop avr in proof}
     \Per(E) \ge {\sf AVR}(\X)^{\frac 1N} \cdot \mathcal I^\flat_g(\mm(E)) ,\qquad \forall \, E \subset \X\text{ Borel with }\mm(E)<\infty,
\end{equation}
where $g(t) = N\omega_N t^{N-1}$. Hence, by considering $ \Omega=\X, I = (0,\infty)$, then $\omega =  \mm_{0,N} = g\Leb 1\mres I$ and ${\sf C} \coloneqq  {\sf AVR}(\X)^{\frac 1N} $ we see that Theorem \ref{thm:main Sobolev PZ metric} as well as Theorem \ref{thm:polya BV} here apply (recall that $\X$ is a  PI-space). This concludes the proof of i) and ii).

We are therefore left to prove the rigidity statements. Hence we assume that equality holds in \eqref{eq:polya Sobolev AVR} with both sides non-zero and finite. Then the rigidity part of Theorem \ref{thm:Polya W1p} ({taking into account Remark \ref{topological reg avr}}) gives the existence of a set $E=\{u>t\}$ that satisfies equality in \eqref{eq:isop avr in proof} for a.e.\ $t \in(m,M)$ and that $\X$ is an $N$-volume cone.

Assume now that  $(u^*)' \neq 0$ a.e.\ on $\{\essinf u<u^*<\esssup u\}$. Then, $u$ is deduced radial by application of Theorem \ref{thm:radiality astratta} recalling also that $\X$ has the Sobolev-to-Lipschitz property (cf.\ \cite[Section 4.1.3]{Gigli13_splitting}) and also Theorem \ref{thm:rigidity ISOAVR}. {Finally, if $\X$ is also an ${\sf RCD}(0,N)$ space, the last conclusion follows by \cite{DePhilippisGigli15}}.
\end{proof}
\subsection{A boosted inequality in Convex Cones}\label{sec:boosted cones}
As anticipated, we can specialize the previous analysis to the setting of weighted convex cones. For a weight $w$, recall that we denote by $\mm_w(E) \coloneqq \int_E w\Leb d$ and by  $H^{1,p}_0(\Omega;w)$ the closure of $C^\infty_c(\Omega)$ with respect to the weighted Sobolev norm. 
\begin{theorem}\label{thm:boosted covnex cones}
    Let $d \in \N,d\ge 2$ be an integer and let $\Sigma \subset \R^d$ be an open convex cone with vertex $0$ so that $ \Sigma = \tilde \Sigma \times \R^{d-k}$ with $k \in\N, 0\le k<d$ with $\tilde \Sigma$ being an open convex cone containing no lines. Let $w \colon \overline{\Sigma} \to [0,\infty)$ be a non-negative (non identically zero)  continuous $\alpha$-homogeneous function so that $w^{1/\alpha}$ is concave on $\Sigma$. Set $N\coloneqq d+\alpha$. Then, for every $  \varnothing \neq\Omega\subset \Sigma$ open with $\mm_w(\Omega)<\infty$ and $u \in H^{1,2}_0(\Omega;w)$ with $\essinf u>0$, the decreasing rearrangement $u^*$ with respect to  $ \mm_{0,N}$ satisfies $ u^* \in {\sf AC}_{loc}(\Omega^*)$ and it holds
    \begin{equation}\label{eq:boosted pz}
         \int_\Omega |\nabla u|^2w\,\d\Leb d \ge  \big( \mm_w(B_1(0)\cap \Sigma) /\omega_N \big)^{\frac{2}{N}} \int_{\Omega^*} |(u^*)'|^2 \,\d \mm_{0,N} + c \mathcal A(\Omega) \mm_w(\Omega)^{1-\frac 2N}s^2,
    \end{equation}
    for some constant $c>0$ depending on $d,\alpha,\Sigma$ and $w$, where  
    \[
          \mathcal A(\Omega) = \inf\left\{ \frac{\mm_w\big(\Omega\triangle (B_r(x) \cap \Sigma)\big)}{\mm_w(\Omega)}  \colon x \in \R^k\times {0^{d-k}} \text{ and } \mm_w(\Omega)= \mm_w(B_r(x_0)) \right\},
    \]
    and $s \coloneqq \sup \left\{ t>0 \colon \mm_w(\{u>t\})  \ge \mm_w(\Omega)\big(1- \mathcal A(\Omega)/4 \big)\right\}$.
\end{theorem}
\begin{proof}    In this proof,  $c>0$ is a constant depending only on  $d,\alpha,\Sigma$ and $w$ whose value might change from line to line.
    
    Recall that, it is known under the standing assumptions that the convex cone can be regarded implicitly as a metric measure space by setting
    \[
        (\X,\sfd,\mm_w) \coloneqq (\overline{\Sigma}, |\cdot |, w\Leb d),
    \]
    and, thanks to the assumptions on the weight $w$  \cite{CabreRosOtonSerra16}, it holds (see \cite{BaloghKristaly21,CavallettiManini22}) that $\X$ is an ${\sf RCD}(0,N)$ space for $N= d+\alpha$ satisfying
    \[  
        {\sf AVR}(\X) =   \frac{\mm_w(B_1(0)\cap \Sigma)}{\omega_N}.
    \]

   Note, by lower semicontinuity, that $u \in H^{1,2}_0(\Omega;w) \subset W^{1,2}_{loc}(\X)$ with $|Du|\le |\nabla u|$ a.e.\ on $\Omega$. Therefore \eqref{eq:boosted pz} in the case $\mathcal A (\Omega)=0$ follows directly from Theorem \ref{thm:main noncompact AVR}. 

   We now proceed with the proof of \eqref{eq:boosted pz} when $\mathcal A (\Omega)>0$. It was already noticed in \cite{BaloghKristaly21,CavallettiManini22} that $ \Per(E) = \int_{\partial^* E\cap\Sigma} w\,\d\HH^{d-1}$  for any $E\subset \X$ Borel, that is that the perimeter measure coincides with the classical weighted perimeter relative to $\Sigma$. Thus, we can invoke Theorem \ref{thm:Polya W1p} with $I=(0,\infty)$, specifically \eqref{eq:improved polya rs} with $r=0,t=\esssup u$ to deduce (recall also \eqref{eq:dependence gradient}) that  the decreasing rearrangement $u^*$  of $u$ with respect to $\mm_{0,N}$ satisfies $u^* \in {\sf AC}_{loc}(\Omega^*)$ and 
    \begin{equation}
     \int_\Omega |\nabla u|^2 w\,\d\Leb d \ge \int |D u|^2 \,\d \mm_w\ge {\sf AVR}(\X)^\frac{2}{N}\int_0^\infty  \left(\frac{\Per( \{u>t\})}{\Per(\{u^*>t\})}\right)^2 \int |(u^*)'|\,\d\Per(\{ u^*>t\},\cdot) \d t.
    \label{eq:improved covnex cone}
    \end{equation}
    Recall here $\Per(\{ u^*>t\},\cdot)$ is the perimeter measure on $\Omega^*$ with respect to the metric measure space $(\R,|\cdot |,\mm_{0,N})$. Recall $\mu(t) =\mm_w(\{u>t\})$  for $t>0$ is the distribution function of $u$.  Fix any $t>0$. By the quantitative isoperimetric inequality \cite[Theorem 1.3]{CintiGlaudoPratelliRosOtonSerra20}, we infer
    \begin{equation}
       \Per(\{u>t\}) - \Per(\{u^*>t\}) \ge c \mu(t)^{\frac{N-1}{N}} \mathcal A(\{u>t\})^2,  \label{eq:quantitative level set}
    \end{equation}
    for some constant $c>0$ depending only on  $d,\alpha,\Sigma$ and $w$ (here we are using that $\Per(\{u^*>t\}) = N({\sf AVR}(\X)\omega_N)^{1/N}\mu(t)^{\frac{N-1}{N}})$). Notice that $\mathcal A(\{u>t\})$ is also well defined, as the asymmetry does not depend on the equivalence class up to $\Leb d$-negligible set of $u$. By convexity, we then have
    \begin{align*}
         \Per(\{u>t\})^2&\ge   \Per(\{u^*>t\})^2  +2 \Per(\{u^*>t\})\left( \Per(\{u>t\})  - \Per(\{u^*>t\}) \right) \\
         &\ge   \Per(\{u^*>t\})^2 + 2 (N {\sf AVR}(\X))^{\frac 1N}\omega_N \mu(t)^{\frac{N-1}{N}})  \left( \Per(\{u>t\})  - \Per(\{u^*>t\}) \right).
    \end{align*}
    Combining the above with \eqref{eq:quantitative level set} and then plugging in \eqref{eq:improved covnex cone} we get, after using coarea formula, that
    \[
      \int_\Omega |\nabla u|^2 w\,\d\Leb d \ge {\sf AVR}(\X)^\frac{2}{N} \int_0^\infty |(u^*)'|^2\, \d \mm_{0,N} + c\int_{(0,s)\setminus u^*(D^+_{u^*})} \frac{\mathcal A(\{u>t\})^2\mu(t)^{2\frac{N-1}{N}}}{-\mu'(t)}\,\d t,
    \]
    having also used Lemma \ref{lem:LEMMONE} to get the first term and Lemma \ref{lem:derivative mu} for the second term. Actually, arguing as in the proof of Theorem \ref{thm:Polya W1p}, we get also that $[0, s)\setminus u^*(D^+_{u^*})$ is negligible. 
    
    Since $\mu(t)$ is decreasing and $\essinf u>0$ strictly by assumption, we have $\mu(t)\ge \mu(s) > \mm_w(\Omega)(1-\mathcal A(\Omega)/4)$ for all $t \in (0,s)$. Arguing as in \cite[Lemma 2.8]{BrascoDephilippis17} in this setting (this principle goes under the name of  ``propagation of asymmetry'', see also  \cite[Section 5]{HansenNadirashvili94}) using that  $\mm_w(\Omega\setminus \{u>t\}) /\mm_g(\Omega) \le \mathcal A(\Omega)/4$), we deduce that $\mathcal A(\{u>t\}) \ge \mathcal A(\Omega)/2$ for all $t \in (0,s)$. Thus, we get
    \[
     \int_0^s \frac{\mathcal A(\{u>t\})^2\mu(t)^{2\frac{N-1}{N}}}{-\mu'(t)}\,\d t\ge   \frac{\mathcal A(\Omega)^2}{8} \mm_g(\Omega)^{2\frac{N-1}{N}}\int_0^s\frac{1}{-\mu'(t)}\,\d t,
    \]
    having used that $\mm_w(\Omega)(1-\mathcal A(\Omega)/4)\ge \mm_w(\Omega)/2$ and by convention we set $1/\mu'(t)=0$ if $t \notin u^*(D^+_{u^*})$.
    Finally, by convexity and Jensen inequality, we observe
    \[
    \int_0^s\frac{1}{-\mu'(t)}\,\d t \ge \frac{s^2}{\int_0^s-\mu'(t)\,\d t} \ge \frac{s^2}{\mm_w(\Omega)-\mu(s)} \ge \frac{4s^2}{\mathcal A(\Omega)\mm_w(\Omega)},
    \]
    where the central inequality is due to the fact $\mu(0) - \mu(s) \ge \int_0^s -\mu'(t)\,\d t$ as, in general, $t\mapsto \mu(t)$ might be not locally absolutely continuous (cf.\ with Lemma \ref{prop:non vanishing}) and the last inequality is due to the choice of $s$. This concludes the proof.
\end{proof}

\section{Proofs: geometric and functional inequalities under Ricci lower bounds}

\subsection{Quantitative stability under positive Ricci lower bounds}
In this part, we prove each conclusion of  Theorem \ref{thm:main quant compact spaces} separately.

\begin{proof}[Proof of i) in Theorem \ref{thm:main quant compact spaces}]
Up to scaling, it is enough to work with $\mm(\X)=1$. Let us fix $u \in W^{1,p}(\X)$ non-zero with $\int u|u|^{p-2}\,\d\mm =0$ and set for brevity
\[
\delta(u) :=   \|u\|_{L^p(\X)}^{-p}\|Du\|_{L^p(\X)}^p- \lambda_{p,N-1,N}.
\]
Consider $u^*$ the decreasing rearrangement of $ u$ with respect to $\mm_{N-1,N}$. By Theorem \ref{thm:main polya compact}, we can estimate using \eqref{eq:BBG quantitativo} for some $C_N>0$ 
\begin{align*}
\delta(u) \ge  \big(1+C_N( \pi - \diam(\X))^N\big)^{\frac p2}\frac{\|(u^*)'\|^p_{L^p(\mm_{N-1,N})}}{\|u^*\|_{L^p(\mm_{N-1,N})}^{p}}- \lambda_{p,N-1,N}.
\end{align*}
We now notice that $u^*$ is a competitor to optimize $\lambda_{p,N-1,N}$ since $ \int u^*|u^*|^{p-2}\,\d\mm_{N-1,N}=0$ by vi) in Lemma \ref{lem:basic prop u*} choosing $G(t)=t|t|^{p-2}$. In particular,  we obtain
\[
\delta(u) \ge C (\pi - \diam(\X))^{{N}}  \lambda_{p,N-1,N},
\]
for a possibly smaller constant $C>0$ depending on $p,N$ (having used the global lower bound $(1+x)^q-1>qx$ for $x>0$ if $q\ge 1$, while for $q\in (0,1)$, a lower bound $(1+x)^q-1>c_{N,q}x $ valid for $x \in (0,C_N\pi^d)$ for a suitable $c_{N,q}>0$ depending on $q,d$). The conclusion, in this case is then given by absorbing in $C$ the  positive number $ \lambda_{p,N-1,N}$ (if $p\ge 2$ and $N\in\N$ then we can also use the explicit bound $\lambda_{p,N-1,N} \ge \frac{N^{p/2}}{(p-1)^{p-1}}$  \cite{LiWang2016}).
\end{proof}

\begin{proof}[Proof of ii) in Theorem \ref{thm:main quant compact spaces}]
Up to scaling, we can suppose $\mm(\X)=1$.  Fix $u\in W^{1,2}(\X)$ non-constant and set
\[
\delta(u) :=   \frac{q-2}{N} - \frac{\| u\|^2_{L^{q}(\X)} - \|u\|_{L^2(\X)}^2 }{ \| Du\|_{L^2(\X)}^2 }.
\]
Consider $u^*$ the decreasing rearrangement of $u$ with respect to $\mm_{N-1,N}$. By Theorem \ref{thm:main polya compact}, we can estimate using \eqref{eq:BBG quantitativo} for some $C_N >0$
\[
     \| D u\|_{L^2(\X)}^2\Big(\frac{q-2}{N} - \delta(u)\Big) \ge \Big(\frac{q-2}{N} - \delta(u)\Big)(1+C_N(\pi - \diam(\X))^N)\|(u^*)'\|_{L^2(\mm_{N-1,N})}^2.
\]
On the other hand, by equimeasurability, we have
\[
\|D u\|_{L^2(\X)}^2\Big(\frac{q-2}{N} - \delta(u)\Big)  =  \| u^*\|^2_{L^{q}(\mm_{N-1,N})} - \|u^*\|_{L^2(\mm_{N-1,N})}^2 \le  \frac{q-2}{N} \| (u^*)'\|_{L^2(\mm_{N-1,N})}^2,
\]
having used, in the last inequality, the validity of the Sobolev inequality for $u^*$ in the model interval (see, e.g., \cite{CavallettiMondino17}). Since $u$ is assumed not constant, then $u^*$ is not constant (for instance, by the second part of i) in Lemma \ref{lem:basic prop u*}). Hence, we can simplify the $\| (u^*)'\|_{L^2}$ term and rearrange to achieve
\[
C_N\frac{q-2}{N}(\pi - \diam(\X))^N \le \delta(u)(1+C_N(\pi-\diam(\X))^N). 
\]
Finally, using that $\pi- \diam(\X) \le \pi$, we find $C>0$ depending only on $N$ and $q$ for the conclusion \eqref{intro:q-sob quantitative} to hold.
\end{proof}

\begin{proof}[Proof of iii) in Theorem \ref{thm:main quant compact spaces}]
Up to scaling, it is enough to work $\mm(\X)=1$. Fix $u\in W^{1,2}(\X)$ non-constant with $\int |u|\,\d\mm =1$ and set
\[
\delta(u)\coloneqq \Big(\int |u|\log|u|\,\d\mm\Big)^{-1}\int_{\{|u| >0\}}\frac{|Du|^2}{|u|}\,\d\mm - 2N.
\]
Consider now the decreasing rearrangement $|u|^*$ of $|u|$ with respect to  $\mm_{N-1,N}$. Note that $\sqrt{|u|^*} = \big(\sqrt{|u|}\big)^*$ $\mm_{N-1,N}$-a.e.\ on $(0,\pi)$ point recalling \cite[Lemma 2.20]{mondinovedovato21} since $t\mapsto\sqrt{t}$ is monotone. By Theorem \ref{thm:main polya compact}, using $2|D\sqrt {|u|}| = |Du|/\sqrt{|u|}$ on $\{|u| >0\}$ by the  chain rule, as well by \eqref{eq:rearrangeEntropy}, we have
\begin{align*}
    \delta(u) &= \frac{2 \int|D\sqrt{ |u|}|^2\,\d\mm}{\int |u|\log|u|\,\d\mm}- 2N \ge 2\frac{ {\sf BBG}_N(\diam(\X))^2 \int|(\sqrt{ |u|^*})'|^2\,\d\mm_{N-1,N}}{\int |u|^*\log |u|^*\,\d\mm_{N-1,N}}- 2N \\
    &\overset{\eqref{eq:BBG quantitativo}}{\ge} C_N(\pi-\diam(\X))^N\frac{\int_{\{|u|^* >0\}}\frac{|(|u|^*)'|^2}{||u|^*|}\,\d\mm_{N-1,N}}{\int |u|^*\log |u|^*\,\d\mm_{N_1,N}} \\
    &\ge 2NC_N(\pi-\diam(\X))^N,
\end{align*}
having used in the last inequality  that $\int |u|^* \d\mm_{N-1,N}=1$ and { that $|u|^*$ is not constant (e.g.\ by i) in Lemma \ref{lem:basic prop u*})}, so that  $|u|^*$ satisfies the log-Sobolev inequality in the model interval (see, e.g., \cite{CavallettiMondino17}). For a suitable $C$ depending only on $N$, the above concludes the proof of \eqref{intro:log Sob quantitative}.
\end{proof}
In the  above, we used the following lemma
\begin{lemma} 
    Let  $\Xdm$ be a metric measure space, let $\varnothing \neq \Omega\subset\X$ be open and let $u \in L^p(\Omega)$ for $p \in [1,\infty)$. Consider  $\omega \ll \Leb 1$  satisfying \eqref{eq:mass compatibility},\eqref{eq:zero a sx} and consider the decreasing rearrangement $|u|^*$ of {$|u|$ with respect to $\omega$}. Then, it holds
    \begin{equation}
       \int_\Omega |u|^p\log(|u|^p)\,\d\mm = \int_\Omegastar (|u|^*)^p\log( (|u|^*)^p)\,\d\omega, \label{eq:rearrangeEntropy}
    \end{equation}
    meaning that one integral makes sense if and only if the other does, in which case they agree.
\end{lemma}
\begin{proof}
Follows from \eqref{eq:composition rearr G} taking $G(t)\coloneqq t^p\log(t^p)$ for $t\ge 0$ and $G(t)\coloneqq 0$ for $t< 0.$
\end{proof}
\subsection{Rigidity under non-negative Ricci curvature and Euclidean volume growth}
In this section, we prove each conclusion of  Theorem \ref{thm:main noncompact AVR} separately. 
\begin{proof}[Proof of i) in Theorem \ref{thm:main noncompact AVR}]
     Here we suppose for some $u \in W^{1,p}_0(\Omega)$ with $\mm(\Omega)<\infty$ that \eqref{intro: p-Eig AVR} holds with both sides non-zero and finite. Set $\rho>0$ so that $\mm_{0,N}(0,\rho)=\mm(\Omega)$ and notice, by scaling, that
    \begin{equation}    \lambda^\mathcal{D}_{p,N,\rho} =  F_{N,p}\mm(\Omega)^{-\frac pN}.
     \label{eq:scaling eigenvalue}
    \end{equation}
    Since $u \in W^{1,p}_0(\Omega)\subset W^{1,p}_{loc}(\X)$ (by zero extension outside of $\Omega$), we can  consider $|u|^*$ the decreasing rearrangement of $|u|$ with respect to  $\mm_{0,N}$ (notice $\mm(\{|u|>t\})<\infty$ for $t>0$, as $|u| \in L^p(\Omega)$). Note also that $\Omega^*=(0,\rho)$  and that by  Remark \ref{rem: Dirichlet BC 1D} we have that $\lim_{t\to \rho^-}|u|^*(t)=0$.  Hence by Lemma \ref{lem:w1p0 1-d} we have $|u|^*\in W^{1,p}_0([0,\rho))$. Therefore we can estimate
    \begin{align*}
       \|Du\|_{L^p(\Omega)}^p&\overset{\eqref{eq:polya Sobolev AVR}}{\ge}   {\sf AVR}(\X)^{\frac pN}\|(|u|^*)'\|^p_{L^{p}(\Omega^*)} \ge \lambda^\mathcal{D}_{p,N,\rho} {\sf AVR}(\X)^{\frac pN}\||u|^*\|^p_{L^{p}(\Omega^*)}\\
       &\overset{\eqref{eq:scaling eigenvalue}}{=} F_{N,p}\big(\mm(\Omega)^{-1}{\sf AVR}(\X)\big)^{\frac pN}\|u\|^p_{L^{p}(\Omega)}.
    \end{align*}
    Therefore, all inequalities must be equalities  giving in turn that $ \lambda^\mathcal{D}_{p,N,\rho} = \frac{\int_0^{\rho} |(|u|^*)'|^p\,\d\mm_{0,N}}{\int_0^{\rho}||u|^*|^p\,\d\mm_{0,N}}.$ In particular $|u|^*$ satisfies all the assumptions of Lemma \ref{lem:1D regularity} with $\lambda=\lambda^\mathcal{D}_{p,N,\rho}$, hence it satisfies \eqref{intro:p-ODE AVR} and $(|u|^*)'\neq 0$ on $(0,\rho)$. We can thus invoke the rigidity case in Theorem \ref{thm:main polya noncompact} to deduce that $|u|= |u|^*\circ \big({\sf AVR}(\X)^{1/N} \sfd_{x_0}\big)$ $\mm$-a.e.\ on $\{|u|>0\}$ for some $x_0\in\Omega$ {and that $\X$ is an $N$-volume cone in the $\CD(0,N)$ (in fact an $N$-Euclidean cone in the $\RCD(0,N)$ case) with tip $x_0$.} In particular, $\{|u|>0\}$ is a ball up to $\mm$-negligible sets. 

    Now we show that $\Omega$ is equivalent to $\{|u|>0\}$ up to negligible set. To this aim, notice that $\mm_{0,N}(\{|u|^*>0\}) = \mm_{0,N}(0,\rho)$ since $|u|^*$ is strictly positive in the interior of $(0,\rho)$ (as $|u|^*$ is non-negative and monotone decreasing with $(|u|^*)'\neq 0$). Therefore, by equimeasurability and definition of $\rho$, we have $\mm(\{|u|>0\}) = \mm_{0,N}(0,\rho) = \mm(\Omega)$. Since, clearly, $\{|u|>0\}\subset \Omega$ up to $\mm$-negligible sets, we deduce that $\Omega$ and  $\{|u|>0\}$ are equivalent. Hence, $\Omega$ is equivalent to a ball centred at $x_0$ of radius $\rho / {\sf AVR}(\X)^{1/N} $. Finally, since $|u|$ is radial and $|u|^*$ is strictly positive, then the claim for $u$ follows up to possibly inverting the sign.
\end{proof}
\begin{proof}[Proof of ii) in Theorem \ref{thm:main noncompact AVR}]
    Here we suppose for some non-zero $u \in W^{1,p}_{loc}(\X)\cap L^{p^*}(\X)$ and 
     $p > N$ that \eqref{intro:p-Sob AVR} holds with both sides nonzero and finite where
    \begin{equation}
       S_{N,p} = \frac 1N\Big(\frac{N(p-1)}{N-p}\Big)^{\frac{p-1}p}\Big(\frac{\Gamma(N+1)}{N\omega_N\Gamma(N/p)\Gamma (N+1-N/p)}\Big)^{\frac 1N}, \label{eq:Sobolev constant}
    \end{equation} 
    is the sharp Euclidean Sobolev constant (generalized, if $N$ is not a natural number). Consider $|u|^*$ the decreasing rearrangement of $|u|$ with respect to  $\mm_{0,N}$ (notice $\mm(|u|>t)<\infty$ for $t>0$, as $u \in L^{p^*}(\X)$) and estimate
    \[
         S_{p,N}{\sf AVR}(\X)^{-\frac 1N}\|Du\|_{L^p(\X)}\overset{\eqref{eq:polya Sobolev AVR}}{\ge}  S_{p,N} \|(|u|^*)'\|_{L^{p^*}(\mm_{0,N})} \ge    \||u|^*\|_{L^{p^*}(\mm_{0,N})} =\|u\|_{L^{p^*}(\mm_{0,N})} ,
    \]
    having used, since $|u|^* \in {\sf AC}_{loc}(0,\infty)$, the Bliss inequality \cite{Bliss30} in the second inequality. Thus, all inequalities must be equalities and by characterization of extremal for the Bliss inequalities we get
    \[
     |u|^*(t) = a\big(1+b t^{\frac p{p-1}}\big)^{\frac{N-p}{p}},\qquad\forall \, t>0,
    \]
    for some $a\in \R,b>0$. In particular, it holds $(|u|^*)'\neq 0$ a.e.\ on $(0,\infty)$ and therefore the rigidity case in Theorem \ref{thm:main polya noncompact} gives that $|u|= |u|^*\circ \big({\sf AVR}(\X)^{1/N} \sfd_{x_0}\big)$ for some $x_0\in\X$ {and the conical rigidity of the ambient space as above.} Since $|u|^*$ is strictly positive,  the claimed identity for $u$ follows up to possibly inverting sign.
\end{proof}
\begin{proof}[Proof of iii) in Theorem \ref{thm:main noncompact AVR}]
    Fix $p>1$  and a function $u \in W^{1,p}(\X)$ satisfying $ \int|u|^p\,\d\mm=1$ (in particular, $u$ is not constant) and such  that \eqref{intro:p-LogSov AVR} holds with both sides non-zero and finite where
    \begin{equation}
       L_{p,N} \coloneqq \frac{p}{N}\Big(\frac{p-1}{e}\Big)^{p-1} \Big(\sigma_N \Gamma\big(N/p'+1\big)\Big)^{-\frac{p}{N}}\label{eq:LogSobolev constant}
    \end{equation}
    is the sharp log-Sobolev Euclidean constant \cite{DelPinoDolbeault03} (extended to  non-integer $N$). Consider $|u|^*$ the decreasing rearrangement of $|u|$ with respect to  $\mm_{0,N}$ and estimate
    \begin{align*}
          \frac{N}{p}\log\left( L_{p,N}{\sf AVR}(\X)^{-\frac pN} \|Du\|_{L^p(\X)}^p\right) &\overset{\eqref{eq:polya Sobolev AVR}}{\ge}  \frac{N}{p}\log\left( L_{p,N} \|(|u|^*)'\|_{L^p(\mm_{0,N})}^p\right)\\
          &\ge \int (|u|^*)^p\log (|u|^*)^p\,\d \mm_{0,N}  \overset{\eqref{eq:rearrangeEntropy}}{=}   \int |u|^p\log |u|^p\,\d \mm ,
    \end{align*}
    having used, since $|u|^* \in {\sf AC}_{loc}(0,\infty)$, the Euclidean log-Sobolev inequality \cite{BaloghDonKristaly24}. Therefore, all inequalities must be equalities, and by the characterization of equality cases in the Euclidean log-Sobolev inequality (by \cite{DelPinoDolbeault03} for $p\le N$ and $N$ integer, while for general $p \in (1,\infty), N>1$ by \cite{BaloghDonKristaly24}) it must hold for some $ \lambda>0$ that
    \[
    |u|^*(t) = \lambda^{\frac N{pp'}} \left(\Gamma( \tfrac n{p'} +1) \omega_N\right)^{-\frac 1p}e^{-\lambda\frac{ t^{p'}}{p}},\qquad \forall \, t>0.
    \]
    In particular, it holds $(|u|^*)'\neq 0$ a.e.\ on $(0,\infty)$ and therefore the rigidity case in Theorem \ref{thm:main polya noncompact} gives that $|u|= |u|^*\circ \big({\sf AVR}(\X)^{1/N} \sfd_{x_0}\big)$ for some $x_0\in\X$ {as well as the conical rigidity of the ambient space as above.} Since $|u|^*$ is strictly positive, the claim for $u$ follows up to possibly inverting sign.
\end{proof}
\section{Proofs: geometric and functional inequalities in the Euclidean space and
beyond}
\subsection{The Faber-Krahn inequality in the Euclidean
space with radial log-convex density}\label{sec:log-convex conj}
In this part, we prove Theorem \ref{thm:faber Krahn log convex}. 

Let $f\colon [0,\infty) \to \R$ be smooth and convex and consider the radial log-convex density
\[
    g(x) = e^{f(|x|)},\qquad \text{on }\R^d.
\]
Kenneth Brakke conjectured that balls around the origin are isoperimetric regions for every volume on $\R^d$ with radially log-convex density, see \cite[Conjecture 3.2]{RosalesCaneteBayleMorgan08}. The validity of this conjecture and the full characterization of isoperimetric regions was affirmatively solved by Chambers in \cite{Chambers19} (see also \cite{FuscoLaManna23} for a quantitative improvement). Let us first recall the isoperimetric problem in this weighted scenario. For $g(x)=e^{f(|x|)}$ as discussed, define the weighted volume and weighted perimeter respectively as
\[
\mm_g(E) \coloneqq \int_E g\,\d \Leb d, \qquad P_g(E) \coloneqq \int_{\partial^* E} g\,\d\HH^{d-1},  
\]
for every  set of finite perimeter $E\subset \R^d$, in the classical sense, where $\HH^{d-1}$ is the $(d-1)$-dimensional Hausdorff measure. We refer to \cite{AmbrosioFuscoPallara,Maggi12_Book} for this classical theory. In particular, the isoperimetric problem in this weighted scenario reads as
\begin{equation}
\mathcal I_g(v)\coloneqq \inf\left\{P_g(E) \colon \mm_g(E)=v\right\},\qquad\forall \, v>0,\label{eq:isop problem log convex}
\end{equation}
where the infimum runs over all sets with finite perimeter. A set is called isoperimetric for the volume $v>0$, if it is a minimum of the above. We can now state the resolution of the conjecture.
\begin{theorem}[\cite{Chambers19}]\label{thm:chambers 1}
    Given a radially log-convex density $g(x)=e^{f(|x|)}$, balls around the origin minimize \eqref{eq:isop problem log convex} for any given volume $v>0$.
\end{theorem}
In particular, we have the following expression for the isoperimetric profile
\[
\mathcal I_g(v) = P_g( B_{r_v}(0)) = g(r_v)\HH^{d-1}(B_{r_v}(0)) = g(r_v)d\omega_d^{\frac 1d} r_v^{\frac{d-1}{d}},\qquad \forall \, v>0,
\]
where $ r_v>0$ so that $v = \mm_g(B_{r_v}(0))$.

Moreover, we also have the full characterization of isoperimetric sets. Define the numbers
\[
R(g)\coloneqq \sup\{ |x|\colon f(x)=f(0),\, x \in \R^d\}, \qquad V_g\coloneqq \mm_g(B_{R(g)}(0)).
\]
\begin{theorem}[\cite{Chambers19}]\label{thm:chambers 2}
The only isoperimetric regions are either balls centred at the origin, or balls entirely contained in $B_{R(g)}(0)$.
\end{theorem}
In particular, we have that:
\begin{itemize}
    \item[{\rm i)}] If $\mm_g(E)\ge V_g$ and $P_g(E)=\mathcal I_g(E)$, then $\mm_g(E\triangle B_{r_{\mm(E)_g}}(0))=0 $;
    \item[{\rm ii)}] If $\mm_g(E)< V_g$ and $P_g(E)=\mathcal I_g(E)$, then $\mm_g(E\triangle B_{r_{\mm_g(E)}}(x))=0 $ for some $x \in B_{R(g)}(0)$ so that $B_{r_{\mm_g(E)}}(x)\subset B_{R(g)}(0) $.
\end{itemize}
We are now ready to prove our main result for this part.
\begin{proof}[Proof of Theorem \ref{thm:faber Krahn log convex}]
    We subdivide the proof into two steps.

    \noindent\textsc{Preliminary step}. We regard $\R^d$ with radial log-convex density $g$ as a metric measure space 
    \[
   (\X,\sfd,\mm_g)\coloneqq (\R^d,|\cdot |, g\Leb d).
    \]
    We notice that $ H^{1,p}_0(\Omega;g) \subset  W^{1,p}_0(\Omega ) \subset W^{1,p}_{loc}(\X)$ (by zero extension outside of $\Omega$) and $|Du|_p\le |\nabla u|$ a.e.\ on $\Omega$  holds for all $u \in H^{1,p}_0(\Omega;g)$. This simply follows by definition and lower semicontinuity of the $W^{1,p}$-space using the identity $\lip \, u  = |\nabla u|$ on $\Omega$ for $u$ smooth (see, e.g., \cite{Bjorn-Bjorn11}). Next, note that
    \begin{equation}
    \Per(E,\cdot ) = g\HH^{d-1}\mres{\partial^* E},\qquad \forall \, E\subset \R^d \text{ Borel so that }\nchi_E \in BV(\X), \label{eq:Per is g H}
    \end{equation}
    as measures on $\R^d$. The simple verification relies on the fact that $g$ is locally positive and locally bounded and on standard lower semicontinuity arguments using the definition of $\Per(E,\cdot)$  as well as classical density arguments (see e.g.\ the proof of Lemma \ref{lem:LEMMONE}). In particular, \eqref{eq:Per is g H} yields that Assumption \eqref{eq:ass2} is true for every $u  \in  W_{loc}^{1,p}(\X)$ by the same argument employed in ii) in Section \ref{sec:key}.
    
    For future use, we point out that $W^{1,p}(\X)$ coincides with the space $H^{1,p}(\R^d;g)$ of classical Sobolev functions with respect to the weighted Lebesgue measure (see \cite{LucicPasqualettoRajala21} for a detailed study). In particular, the Sobolev-to-Lipschitz property holds in $(\X,\sfd,\mm_g)$ by classical reasons.

    \noindent\textsc{Proof of  the results}. We start by proving \eqref{eq:faber density conjecture}. Let $u \in H^{1,p}_0(\Omega;g)$. We can thus invoke Theorem \ref{thm:main Sobolev PZ metric} with $ \Omega=\X, I = (0,\infty)$, then $\omega \coloneqq  \nchi_{(0,\infty)}e^{f(t)} d\omega_d t^{d-1}\d t $ and ${\sf C} \coloneqq  1 $ to deduce that the decreasing rearrangement $|u|^*$ of $|u|$ with respect to  $\omega$ is in $ {\sf AC}_{loc}(\Omega^*)$ where $\Omega^*=(0,\rho)$ with $\mm_g(\Omega)=\omega((-\infty,\rho))$ and it satisfies
   \[
    \int_\Omega|\nabla u|^pg\,\d \Leb d \ge \int |D|u||_p^p\,\d\mm_g \overset{(\ast)}{\ge} \int_0^\rho |(|u|^*)'|^p \,\d\omega.
    \]
    Moreover by Remark \ref{rem: Dirichlet BC 1D} we have $\lim_{t \to \rho^-} |u^*|(t)=0$.
    Now, the function $\tilde u(x) \coloneqq |u|^* (|x|)$ is clearly $H^{1,p}_0(B_\rho (0);g)$, and therefore by definition of $\lambda_{p}^\cD(B_\rho (0);g)$ it holds
    \[
    \int_0^\rho |(|u|^*)'|^p\,\d\omega = \int_{B_{\rho}(0)} |\nabla \tilde u|^pg\,\d \Leb d \ge \lambda_{p}^\cD(B_\rho (0);g) \int_{B_{\rho}(0)} |\tilde u|^pg\,\d\Leb d.
    \]
    By arbitrariness of $u$ and since $|u|,|u|^*$ and $\tilde u$ are equidistributed, the proof of \eqref{eq:faber density conjecture} is concluded.

    Next, we characterize the equality case and further assume that $u \in H^{1,p}_0(\Omega;g)$ satisfies equality in \eqref{eq: equality faber}. Tracing equalities in the previous proof, we get the identity
    \[
     \lambda_{p}^\cD(\Omega;g)  = \lambda_{p}^\cD(B_\rho(0);g) =    \frac{\int_0^\rho |(|u|^*)'|^p \,\d\omega}{\int_0^\rho |u^*|^p \,\d\omega}
    \]
    and also that equality holds in the starred $(\ast)$ inequality.
    On the other hand for all $v \in W^{1,p}((0,\rho);\omega)$ satisfying $\lim_{t \to \rho^-} v(t)=0$, taking $\tilde v\coloneqq v(|x|)$ shows that
    \[
    \lambda_{p}^\cD(B_\rho(0);g) \le  \inf_{ v\in W^{1,p}((0,\rho);\omega), \,\, \lim_{t \to \rho^-} v(t)=0}   \frac{\int_0^\rho |v'|^p \,\d\omega}{\int_0^\rho |v|^p \,\d\omega}.
    \]
    Moreover  $|u^*|$ satisfies equality in the above inequality. Therefore, Lemma \ref{lem:1D regularity} (recall also Lemma \ref{lem:w1p0 1-d}) gives the validity of the differential equation \eqref{eq:ODE log convex} and also that that $(|u|^*)'\neq 0$ a.e.\ on $(0,\rho)$. The latter property,  the first step in this proof and the characterization result for isoperimetric sets given by Theorem \ref{thm:chambers 2} make it then possible to apply Theorem \ref{thm:radiality astratta}. Hence, we get that $|u|=|u|^*(|x-x_0|)$ a.e.\ on $\Omega$ for some $x_0 \in B_{R(g)}(0)$ which, is forced to be $x_0=0$ if $\mm_g(\Omega)\ge V_g$. In particular, since $|u|^*>0$, then $u$ is radial up to possibly inverting the sign.
    
    We conclude by showing that $\Omega$ is a.e.\ equivalent to a ball. Notice $\omega(\{|u|^*>0\}) = \omega(0,\rho)=\mm_g(\Omega) $. Indeed $|u|^*$ is strictly positive on $(0,\rho)$ (as $|u|^*$ is non-negative and monotone decreasing with $(|u|^*)'\neq 0$) and thus, by equimeasurability, also $\mm_g(\{|u|>0\})=\mm_g(\Omega)$. Therefore, since $\{|u|>0\}\subset \Omega$ up to negligible sets and  $\{|u|>0\}$ is equivalent to a ball centred at $x_0$ with radius $\rho$, it follows that $\Omega$ is equivalent to a ball.
\end{proof}
\subsection{The Sobolev inequality in the Euclidean space outside a convex set}\label{sec:outside convex}
Let us recall the relative isoperimetric inequality outside a convex set \cite{ChoeGhomiRitore07} in the version due to \cite{FuscoMorini23}. We denote by $P(E,\Omega)$ the classical perimeter of a Borel set $E$ relative to an open set $\Omega$, see \cite{AmbrosioFuscoPallara,Maggi12_Book}.
\begin{theorem}\label{thm:isoperimetric outsideconvec}
    Let $d\ge 2$ and $C\subset \R^d$ be  closed, convex and with non-empty interior. Then, for every set of finite perimeter $E \subset \R^d\setminus C$ we have
    \begin{equation}
         P(E,\R^d\setminus C) \ge 2^{-1/d} d\omega_d^{1/d}|E|^{\frac{d-1}{d}}.
    \label{eq:isop outside C}
    \end{equation}       
    Moreover, equality occurs for some $E \subset \R^d\setminus C$ if $E$ is a half-ball. 
\end{theorem}
Notice that $v\mapsto 2^{-1/d} d \omega_d^{1/d}v^{\frac{d-1}{d}}$ is the isoperimetric profile relative to an upper half hyperplane in $\R^d$ whose isoperimetric sets for every volume are exactly half balls. Hence the above is a rigid isoperimetric comparison for sets outside a convex set. We are now ready to prove our main result.
\begin{proof}[Proof of Theorem \ref{thm:main outsideconvex}]
     Recall that  $W^{1,p}(\R^d)$ coincide with the classical Sobolev space $H^{1,p}(\R^d)$ and, in this case, $|Du|_p = |\nabla u|$ a.e.\ for every $u \in W^{1,p}(\R^d)$ (cf.\ \cite{Bjorn-Bjorn11}). The discussion applies for the space $BV(\X)$ and more generally by locality, the same discussion again applies to the related notions of calculus on an open set. In particular, we have the representation formula
     \begin{equation}
         \Per(E,\cdot)= \HH^{d-1}\mres{\partial^* E\cap (\R^d\setminus C)}, \qquad \text{if }\nchi_{E} \in BV(\R^d\setminus C), \label{eq:Per is Prelative}
     \end{equation}
     as measures on $\R^d\setminus C$.  As discussed in ii) in Section \ref{sec:key}, the above formula ensures that if  $u \in W^{1,p}_{loc}(\R^d\setminus C)$, then $u$ satisfies Assumption \ref{eq:ass2}.

     We now prove \eqref{eq:Sobolev outside convex}. Fix arbitrary $u \in \dot H^{1,p}(\R^d\setminus C)\cap L^{p^*}(\R^d\setminus C)$. In particular, it holds $\Leb d(\{|u|>t\})<\infty$ for $t >0$. Then, setting $\Omega\coloneqq \R^d\setminus C, I\coloneqq (0,\infty),\omega\coloneqq \nchi_{(0,\infty)} d\omega_d^{1/d} t^{d-1}\d t, {\sf C}\coloneqq 2^{-1/d}$, we can invoke Theorem \ref{thm:main Sobolev PZ metric} (recalling \eqref{eq:Per is Prelative} and Theorem \ref{thm:isoperimetric outsideconvec}) to deduce that, the monotone rarrangement $|u|^*$ of $|u|$ with respect to $\omega$ satisfies
    \[
      \int_{\R^d\setminus C} |\nabla u|^p \,\d\Leb d = \int|D|u||^p\,\d\mm \overset{(\ast)}{\ge} 2^{-p/d}\int_0^\infty|(|u|^*)'|^p\,\d\omega.
    \]
    Now, the Bliss inequality \cite{Bliss30} combined with v) in Lemma \ref{lem:basic prop u*} gives  \eqref{eq:Sobolev outside convex}. More generally, if $u  \in \dot H^{1,p}(\R^d\setminus C)$ is arbitrary, then by definition of the space $\dot H^{1,p}(\R^d\setminus C)$ and a cut-off argument using the weak lower semicontinuity of the $L^{p^*}(\R^d\setminus C)$ norm, we deduce that $u \in L^{p^*}(\R^d\setminus C)$ from what just proved, and thus the proof of  \eqref{eq:Sobolev outside convex} is completed.

    We now analyze the equality case and suppose that $u \in \dot H^{1,p}(\R^d\setminus C)$ non-zero satisfies equality in \eqref{eq:Sobolev outside convex}. The argument is similar to Theorem \ref{thm:radiality astratta}. Since $u \in L^{p^*}(\R^d\setminus C)$, we can trace the equalities in the previous part to deduce that equality occurs in the application of the Bliss inequality. The latter fact directly implies for some numbers $a,b>0$ that 
    \[
    |u|^*(t) = a(1+bt^{\frac{p-1}{p}})^{\frac{p-d}{p}},\qquad \forall \, t>0.
    \]
    In particular, we have $(|u|^*)'\neq 0$ on $(0,\infty)$. Note also that $\esssup |u|^*=\esssup |u|=a$. 
    Moreover, equality also holds in the starred inequality and, by Theorem \ref{thm:Polya W1p}, we deduce that equality holds in \eqref{eq:isop outside C} with $E=\{|u|>t\}$ for a.e.\ $t>0$. However, since $|u|,|u|^*$ are equidistributed and $\omega(\{|u|^*=t\}) =0$, this upgrades to all $t>0$. Therefore, Theorem \ref{thm:isoperimetric outsideconvec} gives that the sets $\{|u|>t\}$ coincide up to negligible sets with (open) half balls that we denote by $B^+_{\rho_t}(y_t)\subset \R^d\setminus C$, centred at some point $y_t \in \partial C$ for some radius $\rho_t\ge 0$ for all $t>0$. Moreover the flat part of $\partial B^+_{\rho_t}(y_t)$  belongs to $\partial C$ (see \cite{FuscoMorini23}). Since $(|u|^*)'\neq 0$, then the last part of Theorem \ref{thm:Polya W1p} gives that $|u|^*$ is invertible with locally absolutely continuous inverse $H \colon (0,a)\to (0,\infty)$ satisfying $|H'(|u|)||D u|=2^{-1/d}$ a.e.\ on $\{|u|>0\}$.

   Fix now any $n \in \N$ with $1/n<a$ and set  $f\coloneqq  2^{1/d} H\circ u_n$, 
   where $u_n \coloneqq  1/n\vee| u| \wedge (a-1/n)$. As in Theorem \ref{thm:radiality astratta}, we readily see that $\{f<t\} \cong B_t^+(x_t)$ for all $t \in ( 2^{1/d}H(a-1/n), 2^{1/2}H(1/n))$, where $B_t^+(x_t)$ is some open half ball centered at some $x_t \in \partial C$, of radius $t>0$ and such that the flat part of $\partial B_{ t}^+(x_{0})$ lies in $\partial C$.

   Secondly, we note that $f$ has a $1$-Lipschitz representative on $B_t^+(x_t)$.  Indeed, by the chain rule in Lemma \ref{lem:local chain rule} we deduce that $f \in W^{1,p}_{loc}(\R^d\setminus C)$ and  $  |D f|= \nchi_{\{m_n<u<M_n\}}$ a.e.\ in $\R^d\setminus C$ (cf.\ with Theorem \ref{thm:radiality astratta}). In particular $f$ has a representative in $ \Lip_{loc}(\R^d\setminus C)$, still denoted by $f.$ 
   Moreover  we get that $f$ is  $1$-Lipschitz in $B_t^+(x_t)$ for all $t>0$ (as $B_t^+(x_t)$ is convex).

   Next we consider the (closed) half sphere $S(t)\coloneqq \partial B^+_{t}(x_t)\cap \partial B_{t}(x_t)$ and claim that 
   \begin{equation}\label{eq:f=t on sphere}
       S(t)\setminus C\subset \{f=t\}
   \end{equation}
    To see this note that for all $z \in S(t)$ and $r>0$, the ball $B_r(z)$ intersects both $B_{t}^+(x_t)$ and  $(\R^d \setminus C)\setminus B_{t}^+(x_t)$ in a set of positive Lebesgue measure. The first one is obvious. On the other hand if the second was false we would be able to find a point in $B_r(z)\cap C$ in the open half-plane containing $B_{t}^+(x_t)$. However, since $x_t \in C$ and $C$ is convex, this would contradict $B_{t}^+(x_t)\cap C=\emptyset .$ Moreover $\{f<t\} \cong B_t^+(x_t)$, hence $B_r(z)$ intersects both $\{f<t\}$ and $\{f\ge t\}$ in a set of positive Lebesgue measure. Since $f$ is continuous in $\R^d\setminus C$, if $z \notin C$ this implies that $f(z)=t$, which proves the claim.

We are now ready to prove that $x_t$ is in fact independent of $t.$
Suppose by contradiction  that $x_t \neq x_{\bar t}$ for some $t,\bar t\in ( 2^{1/d}H(a-1/n), 2^{1/2}H(1/n))$ with $\bar t < t$ and set $\delta\coloneqq |x_t-x_{\bar t}|>0.$ Note that $B_{\bar t}^+(x_{\bar t})\subset \overline{B_t^+(x_t)}$. Moreover, since $x_{\bar t} \in C$, we must have that $x_{\bar t}\in \partial B_t^+(x_t)$ and so $x_t$ lies on the flat part of $\partial B_t^+(x_t)$.  We can then find a segment $l$ with endpoints $x_t$, $x\in S(t)$ and containing $x_{\bar t}$. Denote also by $y$ the point in $l\cap \partial B_{\bar t}(x_{\bar t})$ closest to $x.$ In particular by construction $y \in S(\bar t)$ and $|y-x_{ t}|=\delta+\bar t\le t.$
Recall that $f$ is 1-Lipschitz in $B_t^+(x_t)$ and so can be extended to a 1-Lipschitz function in $\overline{B_t^+(x_t)}$, still denoted by $f$. From \eqref{eq:f=t on sphere} we then get $f(x)=t$, $f(y)=\bar t.$ Therefore
   \[
t-\bar t=f(x)-f(y)\le |x-y|=|t-(\bar t+\delta )|=t-(\bar t+\delta )=t-\bar t-\delta< t-\bar t,
\]
which is a contradiction. Hence $x_t\equiv x_0$ for all $t \in ( 2^{1/d}H(a-1/n), 2^{1/2}H(1/n))$ and some $x_0 \in \partial C$. Additionally, because $B_{ t}^+(x_{ 0})$, $t>0$, form an increasing sequence of sets, we must have that the flat part of $\partial B_{ t}^+(x_{0})$ belongs to a common hyperplane $H$. Up to rigid motions we can thus assume that $\cup_{t>0} B_{ t}^+(x_{ 0})=\{x_d>0\}$ and that $x_0$ is the origin. Finally, since the flat part of $\partial B_{ t}^+(x_{0})$ lies in $\partial C$ we obtain that $\{x_d=0\}\subset \partial C$ and that $\{x_d>0\}$ is a connected component of $\R^d\setminus C$. Therefore $f(x)=|x|$ in  $\{1/n<|u|\le a-1/n\}$  and $\{x_d>0\}\cong \{|u|>0\}$.  From this, by the arbitrariness of $n\in \N$ and arguing as in Theorem \ref{thm:radiality astratta}, we deduce that
\[
|u|(x)=|u|^*\big(2^{-1/d}|x|\big), \quad \text{a.e.\ in $\{|u|>0\}=\{x_d>0\}$}
\]
and that $u=0$ a.e.\ in the other connected components of $\R^d\setminus C.$
 It remains to show that $C$ is a strip. Consider a point $(\bar x,t)\in C$, with $\bar x \in \R^{d-1}$ and $t\le 0$. Then $C$ contains the convex hull of $(\bar x,t)$ and $\{x_d=0\}$, which is the strip $\R^{d-1}\times [t,0]$. Hence, since $C$ is closed, denoting by $-L$ the infimum of all $t\le 0$ such that there is $(\bar x,t)\in C$ we deduce $\R^{d-1}\times [-L,0]\subset C$. On the other hand if we had $y \in C\setminus  (\R^{d-1}\times [-L,0])$ we would  have $y=(\bar x,t)\in C$ for some $t<-L$ (recall $\{x_d>0\}\subset \R^d\setminus C$), which is a contradiction. Hence $\R^{d-1}\times [-L,0]= C$. The proof is concluded.
\end{proof}
\subsection{Neumann eigenvalue lower bound in metric spaces}
In  \cite{BCT15} it was proved the following lower bound for the Neumann eigenvalue for open sets in the Euclidean space.
\begin{theorem}[{\cite[Theorem 1.1]{BCT15}}]
    Let $\Omega \subset \rr^d$ be open and bounded with $\Leb d(\Omega)>0$. Then:
    \begin{equation}\label{eq:BCT bound}
          \lambda_{p}^\cN(\Omega) \ge 2^\frac{p}{d} K_d(\Omega)^p\lambda_{p}^\cD(B_R(0)),
    \end{equation}
    for all $p\in(1,\infty)$ where $|B_R(0)|=|\Omega|$ and $K_d(\Omega)\coloneqq \underset{E\subset \Omega}{\inf} \Per(E,\Omega) \Big(d \omega_d^{1/d} \min(|E|,|\Omega \setminus E|)^\frac{d-1}{d}\Big)^{-1}.$
\end{theorem}
In \cite{BCT15} is also shown that \eqref{eq:BCT bound} is sharp for $p=n=2.$  In this part, as an application of Theorem \ref{thm:main Sobolev PZ metric}, we extend the above result to arbitrary metric measure spaces. It is worth noting that the proof of \eqref{eq:BCT bound} in \cite{BCT15} is considerably different from the one presented here since it is based on a comparison result \`a la Chiti and on a reverse H\"older inequality. Instead, our argument is a direct application of the P\'olya-Szeg\H{o} principle, provided we rearrange with respect to the appropriate weight $g$.

Fix $r>0$ and $Q>1$. Consider the one-dimensional weight $\omega=g(t)\nchi_{[0,2r]}(t)\d t $, where $g:[0,2r]\to \rr$ is given by
$g(t)\coloneqq 
    \omega_Q \frac1Q Q \min(t^{Q-1},(2r-t)^{Q-1}).$
In particular $\omega_{Q,r}(\R)=2\omega_Qr^Q$ and it can be easily checked that
\begin{equation}
\Ig(v)=Q \omega_Q^\frac1Q \min(v,\omega_{Q,r}(\R)-v)^\frac{Q-1}{Q}, \qquad \forall \, v \in (0,\omega_{Q,r}(\R)).
\label{eq:profile min Q}
\end{equation}
Also we note that the m.m.s.\ $I_{Q,r}\coloneqq ([0,2r],|\cdot |,\omega_{Q,r})$ is an $\RCD(0,Q)$ space. Indeed the function $g^\frac1{Q-1}$ is concave. For $Q\in \mathbb N$, the space $I_{Q,r}$ can be seen as the mGH-limit as $k\uparrow \infty$ of the m.m.\ spaces $(C_k,|. |,\frac1{a_k}\mathcal H^Q)$, where $C_k\subset \rr^Q$ is the closed convex set obtained by gluing by their bases two $Q$-dimensional cones of height $r$  with base an $(Q-1)$-dimensional disk of radius $\frac1k,$ and $a_k\coloneqq \frac{\mathcal H^Q(C_k)}{2\omega_Qr^Q}.$  

\begin{lemma}\label{lem:half interval}
    For $r>0$, consider the one dimensional m.m.s.\ $I_{Q,r}\coloneqq ([0,2r],|\cdot |,\omega_{Q,r})$. Then
    \[
    \lambda_p(I_{Q,r})=  \lambda_p^\cN(I_{Q,r})=\lambda_p^\cD([0,r)),
    \]
   where all the eigenvalues  are computed in the  m.m.s.\ $I_{Q,r}.$
\end{lemma}
\begin{proof}
The first identity is by definition, hence we only need to show the second identity.

    Since $I_{Q,r}$ is a compact $\RCD(0,Q)$ space it admits a Sobolev embedding (analogously, one can argue that it is a PI-space as a suitable weighted uniform domain \cite{Keith03,Bjorn-Bjorn11}), hence via the direct method of calculus of variations we can find a minimizer $u \in W^{1,p}([0,2r];g)$ for $ \lambda_p(I_{Q,r})$, i.e.\ in the Sobolev space in the metric measure space $([0,2r],|\cdot |,\omega_{Q,r})$ (see e.g.\ \cite[Theorem 4.3]{MondinoSemola20}). By Lemma \ref{lem:w1p0 1-d}, we have that $W^{1,p}([0,2r];g)=\{ f \in{\sf AC}_{loc}(0,r) \ : \  f',f\in L^p((0,r);g)\}$. Hence $u$ satisfies the assumption  of Lemma \ref{lem:1D regularity} with $\lambda=\lambda_p(I_{Q,r})$, from which we deduce
    \begin{equation}\label{eq:ode for u}
         \left(|u'|^{p-2}u'g \right)'=-\lambda_p(I_{Q,r}) g |u|^{p-2}u.
    \end{equation}
    Since $u$ has a continuous representative and $\int_0^{2r} |u|^{p-2}u g=0$, we must have that $u(x)=0$ for some $x\in (0,2r).$ Without loss of generality we can assume that $x\in (0,r].$ Consider the function $\tilde u$ which coincides with $u$ in $(0,x)$ and identically zero in $[x,r]$. Then by approximating $\tilde u$ with functions in $\Lip_c([0,x))$ we have that 
    \[
     \frac{\int_0^x |u'(t)|^p g(t)\d t}{\int_0^x |u(t)|^p g(t)\d t}\ge \lambda_p^\cD([0,r)).
    \]
    On the other hand multiplying \eqref{eq:ode for u} by $u$ and integrating by parts (using that $g(0)=0$) we obtain
    \[
    -\lambda_p(I_{Q,r})\int_0^x |u|^p g =\int_0^x \left(|u'|^{p-2}u'g \right)' u \d t=-\int_0^x |u'|^{p}g  \d t.
    \]
    Combining with the above we obtain $\lambda_p(I_{Q,r})\ge \lambda_p^\cD([0,r))$. On the other hand consider $v$ a competitor for $\lambda_p^\cD([0,r))$. Then the function $\tilde u$ that equals $v$ in $(0,r)$ and $-v(r-t)$ in $(r,2r)$ is clearly a competitor  for  $\lambda_p(I_{Q,r})$. This immediately shows that $\lambda_p(I_{Q,r})\le \lambda_p^\cD([0,r))$ concluding the proof.
\end{proof}
We are now ready to prove our last main result.
\begin{proof}[Proof of Theorem \ref{thm:main neumann}]
    If $\lambda_p^\cN(\Omega)=+\infty$ there is nothing to prove. Assume then $\lambda_p^\cN(\Omega)<\infty$. In particular ,for all $\eps>0$ we can find $u \in W^{1,p}(\Omega)$ with $\int |u|^{p-2}u\, \d \mm=0$, $\|u\|_{L^p(\X)}=1$ satisfying
    \[
    \int |D u|_p^p \d \mm\le  \lambda_p^\cN(\Omega)+\eps.
    \]
    Clearly we have $\mm(\{u>t\})<\infty$ for all $t \in \R.$  Moreover taking $r=\left(\frac{\mm(\Omega)}{2\omega_Q}\right)^\frac{1}{Q},$ it holds $\omega_{Q,r}(\R)=2\omega_Qr^Q=\mm(\Omega).$ Hence it makes sense to consider the rearrangement $u^*$ of $u$ with respect to  $\omega_{N,r}.$ Thanks to \eqref{eq:isop relative} and recalling \eqref{eq:profile min Q}, we can apply Theorem \ref{thm:main Sobolev PZ metric} and obtain
    \[
    C^p\int_0^R |(u^*)'|^p\d \omega_{N,r}\le  \int |D u|_p^p\d \mm \le \lambda_p^\cN(\Omega)+\eps.
    \]
    Finally, by Proposition \ref{lem:basic prop u*} v)-vi), we have $\int |u^*|^{p-2}u^*\d \omega_{Q,r}=0$ and $\|u^*\|_{L^p(\omega_{Q,r})}=1$. Therefore $u^*$ is a competitor for $\lambda_p^\cN(I_{Q,r})$, that is $\int_0^R |(u^*)'(t)|^p\d \omega_{Q,r}\ge \lambda_p^\cN(I_{Q,r}).$ From the arbitrariness of $\eps$ we obtain that 
    $\lambda_p^\cN(\Omega)\ge C^p\lambda_p^\cN(I_{Q,r}).$ Then by Lemma \ref{lem:half interval}  we obtain that
    \[
     \lambda_p^\cN(\Omega)\ge C^p\lambda_p^\cD([0,r)).
    \]
    From this \eqref{eq:neumann lower bound} follows noting that $\lambda_p^\cD([0,r)$ (computed in $I_{Q,r}$) coincides with the Dirichlet $p$-eigenvalue of $[0,r)$, but computed in the m.m.s.\ $I_Q\coloneqq ([0,\infty),|\cdot |, Q\omega_Q t^{Q-1}\d t)$, and by scaling.
\end{proof}

\appendix

\section{Calculus on weighted interval}\label{Appendix:Calculus1D}
Due to the need to consider everywhere in this note monotone functions on weighted intervals, we prove here some useful results around the nonsmooth calculus.
\begin{lemma}\label{lem:LEMMONE}
        Let $J\subset \R$ be a possibly unbounded open interval. Let $g \colon J \to (0,\infty)$ be continuous. Define $\omega \coloneqq   g\Leb 1\mres J$ and fix $v\in L^1_{loc}(\omega)$.

        Then $v$ is of locally bounded variation on $J$ in the classical sense if and only if $v \in BV_{loc}(J;\omega)$, in which case $|\dD v|=g TV(v)$ and $|Dv|_1 = |v'|$ $\omega$-a.e. In particular, it holds as measures on $J$ that
        \begin{equation}\label{eq:perimeter of halfline}
        \Per((-\infty,r)\cap J,\cdot )=g(r)\delta_{r}, \qquad \forall \, r\in J.
        \end{equation}       
        Finally, if  $p>1$ we have that  $v\in {\sf AC}_{loc}(J)$  and $v,|v'|\in L^p_{loc}(\omega)$  if and only if $v \in W^{1,p}_{loc}(J;\omega)$ and, in this case,  it holds $|Du|_p = |v'|$ $\omega$-a.e.. 
\end{lemma}
\begin{proof}
Denote $J=(-a,b)$ with $a,b \in [0,\infty]$.
Thanks to the assumptions on the weight $g$, it holds $C_K\le  g \le C_K^{-1}$ on $K$  for $K\subset J$ compact and for some $C_K>0$. In particular, the measure $\omega$ and $\Leb 1\mres{J}$ are mutually absolutely continuous and the space $L^1_{loc}$ as well as the $L^1_{loc}$ convergence with respect to either $\omega$ or to $\Leb 1\mres{J}$  coincide. We shall tacitly use multiple times the above facts.

Let us assume first that $v$ is of locally bounded variation.  Fix $U$ open relatively compact in $J$, in particular $TV(v)(U)$ is finite. By the Anzellotti-Giaquinta approximation result (see, e.g., \cite[Theorem 3.9]{AmbrosioFuscoPallara}) we obtain functions $v_n \in C^\infty(U)$ such that $v_n \to v$ in $L^1(U)$  and $TV(v_n)(U)\to TV(v)(U)$. In particular, $TV(v_n)$ converge as measures to $TV(v)$ in duality with bounded continuous functions in $U$ (see e.g.\ before \cite[Proposition 3.15]{AmbrosioFuscoPallara})).  Hence
\[
|\dD v|(U) \le \liminf_n| \dD v_n|(U) \le \liminf_n \int_U |v_n'|g\,\d \Leb 1 = \int_U g\,\d TV(v),
\]
where in the first equality we used the lower semicontinuity of total variation measures (in the metric sense) with respect to $L^1_{loc}$ convergence, in the second inequality the fact that $v_n \in \Lip_{loc}(U)$ and the definition of $| \dD v_n|$ and, in last equality, we used that $g$ is continuous and bounded on $U$.

For the converse implication, assume $v \in BV_{loc}(J;\omega)$. Consider an optimal sequence $(v_n)\subset \Lip_{loc}(U)$ so that $v_n \to v$ in $L^1(U)$ and  $\int \lip \, v_n\, \d\omega \to |\dD v|(U)$. This already shows that $v$ is classically of bounded variation in $U$, since $\|v_n'\|_{L^1(U)}$ is uniformly bounded.
Moreover, thanks to the lower semicontinuity of $|\dD v|$ with respect to the $L^1$-convergence, $(\lip \, v_n ) \omega$ converges to $|\dD v|$ weakly in duality with continuous functions on $U$. Using Cavalieri's formula, we can estimate
\begin{align*}
    \int_U g\,\d TV(v) &= \int_{\R} TV(v)(\{g> t\}\cap U)\, \d t \le \liminf_{n\to\infty}\int_{\R} \int_{\{g>t\}\cap U} |v_n'|\, \d\Leb 1\d t \\
    &=\liminf_{n\to\infty} \int_U g |v_n'|\,\d \Leb 1 = |\dD v|(U),
\end{align*}
where in the first inequality we used that $TV(v)(A)\le \liminf_n \int_A |v_n'| \d \Leb 1 $ for every open set $A\subset U$ combined with Fatou's Lemma. We obtained that
\begin{equation}
|\dD v|(U) =\int_U g\,\d TV(v),\qquad \text{for all open sets }U\subset J.
\label{eq:DV = TV}
\end{equation}
From this, we deduce that $|\dD v|$ and $g TV(v)$ coincide as locally finite measures on $J$.  In particular, their absolutely continuous parts with respect to $\omega=g \Leb 1$ must coincide, that is $|D v|_1=|v'|$ $\omega$-a.e.. 

From the above, as a particular case we have  $\nchi_{(-a,r)}\in BV_{loc}(J;\omega)$ and consequently, from what we have just proved, it holds
\[
\Per((-\infty,r)\cap J,\cdot ) =|\dD \nchi_{(-\infty,r)\cap J}|=gTV(\nchi_{(-\infty,r)})\mres J,
\]
as measures on $J$, the first being the perimeter measure of $(-a,r)\subset J$ (according to \eqref{eq:def du}). This immediately yields \eqref{eq:perimeter of halfline}.

Let us conclude by proving the last statement. If $v \in W^{1,p}_{loc}(J;\omega)$, then the fact \eqref{eq:dependence gradient} and the identity \eqref{eq:DV = TV} easily imply that $v \in {\sf AC}_{loc}(J)$ and $|v'| \le |Du|_p$ $\omega$-a.e..

Conversely, if $v \in {\sf AC}_{loc}(J)$, the identity \eqref{eq:DV = TV} forces $|\dD^s v|=0$, by uniqueness of the Radon Nikodym decomposition of $|\dD v|$. Thus, if $v,|v'|\in L^p_{loc}(\omega)$, a cut-off and a convolution argument, using the estimate $|((\eta v)\ast \rho_\eps)'|\le |(\eta v)'|\ast \rho_\eps$, for $(\rho_\eps)_{\eps>0}$ convolution kernel and $\eta$ Lipschitz boundedly supported cut-off function, gives $v \in W^{1,p}_{loc}(J;\omega)$ and $|Du|_p \le |v'|$.
\end{proof}
The following gives the characterization of Sobolev functions with zero boundary conditions.
\begin{lemma}\label{lem:w1p0 1-d}
    Let $g:[0,\infty)\to (0,\infty)$ be a continuous function and set $\omega\coloneqq g \d t$. Then for all $r>0$ and $p>1$ it holds
    \[
    W^{1,p}_0([0,r);\omega) = \left\{f \in {\sf AC}_{loc}(0,r) \colon \int_0^r (|f|^p+|f'|^p)g\,\d t <\infty,\,\,  \lim_{t \to r^-} f(t)=0 \right\},
    \]
    where the equality is intended up to a.e.\ identification.
\end{lemma}
\begin{proof}
    Suppose first that $f \in W^{1,p}_0([0,r);\omega)  $. Note first that also $f \in W^{1,p}_0((0,r);\omega) $ and so by Lemma \ref{lem:LEMMONE} we  have $f \in {\sf AC}_{loc}(0,r)$ and  $\int_0^r (|f|^p+|f'|^p)g\,\d t <\infty$. Consider  $f_n \to f $ in $W^{1,p}_0([0,r);\omega)  $ with $f_n \in \Lip_{bs}([0,r).$  In particular trivially  $\lim_{t \to r^-} f_n(t)=0$. On the other hand, $\int_0^r |f_n'(t)|^pg(t)$ is uniformly bounded and since $g$ is bounded away from zero near $r$, we deduce that $f_n$ are uniformly H\"older continuous in a left-neighborhood of $r.$ This is sufficient to deduce that   $\lim_{t \to r^-} f(t)=0$.

    Conversely suppose that have $f \in {\sf AC}_{loc}(0,r)$,  $\int_0^r (|f|^p+|f'|^p)g\,\d t <\infty$ and  $\lim_{t \to r^-} f(t)=0$. Then by Lemma \ref{lem:LEMMONE} we have $f\in W^{1,p}((0,r);\omega)$ and so trivially also $f\in W^{1,p}([0,r);\omega)$. On the other hand the function $f_n\coloneqq (f-1/n)^+\in W^{1,p}([0,r);\omega)$ converge in $W^{1,p}([0,r);\omega)$ to $f$ as $n\to +\infty.$ Since clearly $f_n \in W^{1,p}_0([0,r);\omega)$ we conclude the proof.
\end{proof}

Next, we collect some useful results on eigenvalue problems in weighted intervals.
\begin{lemma}\label{lem:1D regularity}
    Let $g:(0,r)\to (0,\infty)$ be a continuous function with $g\in L^1(0,r)$ and fix $p\in(1,\infty).$ Let $\mathcal W$ be one of the following two function spaces: {
    \begin{align*}
            &   \left \{f \in {\sf AC}_{loc}(0,r) \colon \int_0^r |f|^pg\,\d t <\infty,\,\,  \lim_{t \to r^-} f(t)=0 \right\};\\
            &\left \{f \in {\sf AC}_{loc}(0,r) \colon \int_0^r |f|^pg\d t<\infty,\,\, \int_0^r |f|^{p-2}f g\d t =0  \right\}.
    \end{align*} }
    Suppose that $u \in \mathcal W\setminus\{0\}$,  satisfies
    \begin{equation}\label{eq:RQ in 1d}
        \frac{\int_0^r |u'(t)|^p g(t)\d t}{\int_0^r |u(t)|^p g(t)\d t}=\lambda\coloneqq \inf_{ f \in \mathcal W\setminus\{0\}}   \frac{\int_0^r |f'(t)|^p g(t)\d t}{\int_0^r |f(t)|^p g(t)\d t}<\infty.
    \end{equation}
    Then $|u'|^{p-2}u'g \in C^1(0,r)$, $|u'|^{p-2}u'\in C^0(0,r)$ and
    \begin{equation}\label{eq:implicit ode}
       \left(|u'|^{p-2}u'g \right)'=-\lambda g |u|^{p-2}u,
    \end{equation}
   and if $g$ is a.e.\ differentiable it also holds
    \begin{equation}\label{eq:ode}
         \left(|u'|^{p-2}u'\right)'+|u'|^{p-2}u' (\log g)'=-\lambda  |u|^{p-2}u, \qquad \text{a.e.\ in $(0,r).$}
    \end{equation}
    Moreover if $u$ is non-negative and monotone non-increasing, then  $u'<0$ in $(0,r)$ (and in particular $u>0$ in $(0,r)$).
\end{lemma}
\begin{proof}
Suppose first that {$\mathcal W=\{f \in {\sf AC}_{loc}(0,r)\colon\int_0^r |f|^pg\,\d t <\infty,\,\,  \lim_{t \to r^-} f(t)=0 \}$.}
Fix $\phi\in C^{\infty}_c(0,r)$. Then $u+ s\phi \in \mathcal W$ for all $s\in \R$ and is easily seen that the function
 $$s\mapsto \int_0^r |u'(t)+s \phi'(t)|^p g(t)-\lambda |u(t)+s \phi(t) |^p g(t)\d t$$
 is differentiable at $s=0$ and the derivative must vanish from \eqref{eq:RQ in 1d}. Computing the derivative explicitly  we deduce
    \begin{equation}\label{eq:weak 1d plap}
        \int_0^r |u'(t)|^{p-2}u'(t) \phi'(t) g(t)\d t=-  \lambda \int_0^r |u(t)|^{p-2}u(t) \phi(t) g(t)\d t, \qquad \forall\, \phi \in  C^{\infty}_c(0,r).
    \end{equation}
   Because $w\coloneqq |u'|^{p-2}u'g \in L^1_{loc}(0,r)$, from \eqref{eq:weak 1d plap} we get that $w\in W^{1,1}_{loc}(0,r)$ with $w'=-\lambda u^{p-1} g$ a.e.\ in $(0,r)$. This shows \eqref{eq:implicit ode} and since both $u$ and $g$ are continuous we obtain that $w\in C^1(0,r).$ Noting that $1/g$ is also continuous in $(0,r)$ we obtain that $|u'|^{p-2}u'\in C^0(0,r).$
   A similar standard, but slightly more involved, argument can be used to obtain \eqref{eq:implicit ode}  in the case {$\mathcal W=\{f \in {\sf AC}_{loc}(0,r) \colon \int_0^r |f|^pg\d t<\infty,\,\, \int_0^r |f|^{p-2}f g(t)\d t =0  \}$} (see e.g.\ \cite[Lemma 2.4]{DGB92}).
From \eqref{eq:implicit ode}, since $g$ is strictly positive in $(0,r)$,  assuming it is also a.e.\ differentiable we obtain identity \eqref{eq:ode}.

 Suppose now that $u$ is non-negative and monotone non-increasing. Then $u'\le 0$ a.e.\ in $(0,r)$ and so $w\le 0$. However by \eqref{eq:implicit ode} we have $w'=-\lambda u^{p-1} g\le 0$. Therefore $w$ is monotone non-increasing and non-positive in $(0,r)$. Hence either $w\equiv 0$ or $w<0$ in $(0,r).$ If $w=0$ it means that $u'=0$ a.e.\ in $(0,r)$  and so $u$ is constant. However the only constant function in $\mathcal W$ (in both cases) is zero, and $u\neq 0$ by assumption.
 Hence $w<0$ and, since $g>0$, we deduce that $u'<0$ in $(0,r).$ 
\end{proof}

\medskip

\noindent\textbf{Acknowledgments}. Both authors are members of INDAM-GNAMPA. F.N. was partially supported by the Academy of Finland project \emph{Approximate incidence geometry}, Grant No.\ 355453 and acknowledges the MIUR Excellence Department Project awarded to the Department of Mathematics, University of Pisa, CUP I57G22000700001. 
I.Y.V.\ was partially supported by the European Union (ERC, ConFine, 101078057).
We thank the referee for the useful suggestions.

\def\cprime{$'$} \def\cprime{$'$}


\begin{thebibliography}{100}

\bibitem{AgostinianiFogagnoloMazzieri20}
{\sc V.~Agostiniani, M.~Fogagnolo, and L.~Mazzieri}, {\em Sharp geometric
  inequalities for closed hypersurfaces in manifolds with nonnegative {R}icci
  curvature}, Invent. Math., 222 (2020), pp.~1033--1101.

\bibitem{AguehGhoussoubKang04}
{\sc M.~Agueh, N.~Ghoussoub, and X.~Kang}, {\em Geometric inequalities via a
  general comparison principle for interacting gases}, Geom. Funct. Anal., 14
  (2004), pp.~215--244.

\bibitem{Ambrosio2001}
{\sc L.~Ambrosio}, {\em Some fine properties of sets of finite perimeter in
  {A}hlfors regular metric measure spaces}, Adv. Math., 159 (2001), pp.~51--67.

\bibitem{Ambrosio2002}
\leavevmode\vrule height 2pt depth -1.6pt width 23pt, {\em Fine properties of
  sets of finite perimeter in doubling metric measure spaces}, Set-Valued
  Anal., 10 (2002), pp.~111--128.

\bibitem{AmbICM}
\leavevmode\vrule height 2pt depth -1.6pt width 23pt, {\em Calculus, heat flow
  and curvature-dimension bounds in metric measure spaces}, in Proceedings of
  the {I}nternational {C}ongress of {M}athematicians---{R}io de {J}aneiro 2018.
  {V}ol. {I}. {P}lenary lectures, World Sci. Publ., Hackensack, NJ, 2018,
  pp.~301--340.

\bibitem{AmbrosioDiMarino14}
{\sc L.~Ambrosio and S.~Di~Marino}, {\em Equivalent definitions of {$BV$} space
  and of total variation on metric measure spaces}, J. Funct. Anal., 266
  (2014), pp.~4150--4188.

\bibitem{AmbrosioFuscoPallara}
{\sc L.~Ambrosio, N.~Fusco, and D.~Pallara}, {\em Functions of Bounded
  Variation and Free Discontinuity Problems}, Oxford Science Publications,
  Clarendon Press, 2000.

\bibitem{AmbrosioGigliSavare11-3}
{\sc L.~Ambrosio, N.~Gigli, and G.~Savar{\'e}}, {\em Density of {L}ipschitz
  functions and equivalence of weak gradients in metric measure spaces}, Rev.
  Mat. Iberoam., 29 (2013), pp.~969--996.

\bibitem{AmbrosioIkonenLucicPasqualetto24}
{\sc L.~Ambrosio, T.~Ikonen, D.~Lu{\v{c}}i{\'c}, and E.~Pasqualetto}, {\em
  Metric {Sobolev} spaces. {I}: {Equivalence} of definitions}, Milan J. Math.,
  92 (2024), pp.~255--347.

\bibitem{AmbrosioMirandaPallara04}
{\sc L.~Ambrosio, M.~Miranda, Jr., and D.~Pallara}, {\em Special functions of
  bounded variation in doubling metric measure spaces}, in Calculus of
  variations: topics from the mathematical heritage of {E}. {D}e {G}iorgi,
  vol.~14 of Quad. Mat., Dept. Math., Seconda Univ. Napoli, Caserta, 2004,
  pp.~1--45.

\bibitem{Ambrosio-Pinamonti-Speight15}
{\sc L.~Ambrosio, A.~Pinamonti, and G.~Speight}, {\em Tensorization of
  {C}heeger energies, the space {$H^{1,1}$} and the area formula for graphs},
  Adv. Math., 281 (2015), pp.~1145--1177.

\bibitem{AntonelliPasqualettoPozzettaSemola22}
{\sc G.~Antonelli, E.~Pasqualetto, M.~Pozzetta, and D.~Semola}, {\em Asymptotic
  isoperimetry on non collapsed spaces with lower {R}icci bounds}, Math. Ann.,
  389 (2024), pp.~1677--1730.

\bibitem{APPV23}
{\sc G.~Antonelli, E.~Pasqualetto, M.~Pozzetta, and I.~Y. Violo}, {\em
  Topological regularity of isoperimetric sets in pi spaces having a
  deformation property}, Accepted Proc. R. Soc. Edinb. Sect. A Math.,
  https://doi.org/10.1017/prm.2023.105,  (2023).

\bibitem{AshbaughBenguria1992}
{\sc M.~S. Ashbaugh and R.~D. Benguria}, {\em A sharp bound for the ratio of
  the first two eigenvalues of {D}irichlet {L}aplacians and extensions}, Ann.
  of Math. (2), 135 (1992), pp.~601--628.

\bibitem{Aubin76-2}
{\sc T.~Aubin}, {\em Probl\`emes isop\'{e}rim\'{e}triques et espaces de
  {S}obolev}, J. Differential Geometry, 11 (1976), pp.~573--598.

\bibitem{Aubry07}
{\sc E.~Aubry}, {\em Finiteness of {$\pi_1$} and geometric inequalities in
  almost positive {R}icci curvature}, Ann. Sci. \'{E}cole Norm. Sup. (4), 40
  (2007), pp.~675--695.

\bibitem{Baernstein19}
{\sc A.~Baernstein, II}, {\em Symmetrization in analysis}, vol.~36 of New
  Mathematical Monographs, Cambridge University Press, Cambridge, 2019.
\newblock With David Drasin and Richard S. Laugesen, With a foreword by Walter
  Hayman.

\bibitem{BakryLedoux96}
{\sc D.~Bakry and M.~Ledoux}, {\em {L\'evy-Gromov's isoperimetric inequality
  for an infinite dimensional diffusion generator}}, Inventiones mathematicae,
  123 (1996), pp.~259--281.

\bibitem{BaloghDonKristaly24}
{\sc Z.~M. Balogh, S.~Don, and A.~Krist{\'a}ly}, {\em Sharp weighted
  log-{Sobolev} inequalities: characterization of equality cases and
  applications}, Trans. Am. Math. Soc., 377 (2024), pp.~5129--5163.

\bibitem{BaloghKristaly21}
{\sc Z.~M. Balogh and A.~Krist\'aly}, {\em Sharp isoperimetric and {S}obolev
  inequalities in spaces with nonnegative {R}icci curvature}, Math. Ann., 385
  (2023), pp.~1747--1773.

\bibitem{BaloghKristalyTripaldi23}
{\sc Z.~M. Balogh, A.~Krist\'{a}ly, and F.~Tripaldi}, {\em {Sharp log-{S}obolev
  inequalities in ${\sf CD}(0,N)$ spaces with applications}}, J. Funct. Anal.,
  286 (2024), pp.~Paper No. 110217, 41.

\bibitem{BerardBessonGallot85}
{\sc P.~B\'{e}rard, G.~Besson, and S.~Gallot}, {\em Sur une in\'{e}galit\'{e}
  isop\'{e}rim\'{e}trique qui g\'{e}n\'{e}ralise celle de {P}aul
  {L}\'{e}vy-{G}romov}, Invent. Math., 80 (1985), pp.~295--308.

\bibitem{BerardMeier1982}
{\sc P.~B\'{e}rard and D.~Meyer}, {\em In\'{e}galit\'{e}s
  isop\'{e}rim\'{e}triques et applications}, Ann. Sci. \'{E}cole Norm. Sup.
  (4), 15 (1982), pp.~513--541.

\bibitem{Besson04}
{\sc G.~Besson}, {\em From isoperimetric inequalities to heat kernels via
  symmetrisation}, in Surveys in differential geometry. {V}ol. {IX}, vol.~9 of
  Surv. Differ. Geom., Int. Press, Somerville, MA, 2004, pp.~27--51.

\bibitem{Bjorn-Bjorn11}
{\sc A.~Bj{\"o}rn and J.~Bj{\"o}rn}, {\em Nonlinear potential theory on metric
  spaces}, vol.~17 of EMS Tracts in Mathematics, European Mathematical Society
  (EMS), Z\"urich, 2011.

\bibitem{BjornSh07}
{\sc J.~Bj\"orn and N.~Shanmugalingam}, {\em Poincar\'e inequalities, uniform
  domains and extension properties for newton-sobolev functions in metric
  spaces}, Journal of Mathematical Analysis and Applications, 332 (2007),
  pp.~190--208.

\bibitem{Bliss30}
{\sc G.~A. Bliss}, {\em An {I}ntegral {I}nequality}, J. London Math. Soc., 5
  (1930), pp.~40--46.

\bibitem{Bobkov96}
{\sc S.~Bobkov}, {\em Extremal properties of half-spaces for log-concave
  distributions}, Ann. Probab., 24 (1996), pp.~35--48.

\bibitem{Bobkov1997}
{\sc S.~G. Bobkov}, {\em An isoperimetric inequality on the discrete cube, and
  an elementary proof of the isoperimetric inequality in {G}auss space}, Ann.
  Probab., 25 (1997), pp.~206--214.

\bibitem{Borell1975}
{\sc C.~Borell}, {\em The {B}runn-{M}inkowski inequality in {G}auss space},
  Invent. Math., 30 (1975), pp.~207--216.

\bibitem{BrandoliniChiacchio23}
{\sc B.~Brandolini and F.~Chiacchio}, {\em Some applications of the {C}hambers
  isoperimetric inequality}, Discrete Contin. Dyn. Syst. Ser. S, 16 (2023),
  pp.~1242--1263.

\bibitem{BCT15}
{\sc B.~Brandolini, F.~Chiacchio, and C.~Trombetti}, {\em Optimal lower bounds
  for eigenvalues of linear and nonlinear neumann problems}, Proceedings of the
  Royal Society of Edinburgh: Section A Mathematics, 145 (2015), pp.~31--45.

\bibitem{BrascoDephilippis17}
{\sc L.~Brasco and G.~De~Philippis}, {\em Spectral inequalities in quantitative
  form}, in Shape optimization and spectral theory, De Gruyter Open, Warsaw,
  2017, pp.~201--281.

\bibitem{BrascoDePhilippisVelichkov2015}
{\sc L.~Brasco, G.~De~Philippis, and B.~Velichkov}, {\em Faber-{K}rahn
  inequalities in sharp quantitative form}, Duke Math. J., 164 (2015),
  pp.~1777--1831.

\bibitem{Brendle20}
{\sc S.~Brendle}, {\em Sobolev inequalities in manifolds with nonnegative
  curvature}, Comm. Pure Appl. Math., 76 (2023), pp.~2192--2218.

\bibitem{BrothersZiemer88}
{\sc J.~E. Brothers and W.~P. Ziemer}, {\em Minimal rearrangements of {S}obolev
  functions}, J. Reine Angew. Math., 384 (1988), pp.~153--179.

\bibitem{Burchard97}
{\sc A.~Burchard}, {\em Steiner symmetrization is continuous in {$W^{1,p}$}},
  Geom. Funct. Anal., 7 (1997), pp.~823--860.

\bibitem{CabreRosOtonSerra16}
{\sc X.~Cabr\'{e}, X.~Ros-Oton, and J.~Serra}, {\em Sharp isoperimetric
  inequalities via the {ABP} method}, J. Eur. Math. Soc. (JEMS), 18 (2016),
  pp.~2971--2998.

\bibitem{CaputoCavallucci24}
{\sc E.~Caputo and N.~Cavallucci}, {\em Poincar{\'e} inequality and energy of
  separating sets}, Adv. Calc. Var., 18 (2025), pp.~915--942.

\bibitem{CaputoKoivuRajala24}
{\sc E.~Caputo, J.~Koivu, and T.~Rajala}, {\em Sobolev, {$BV$} and perimeter
  extensions in metric measure spaces}, Ann. Fenn. Math., 49 (2024),
  pp.~135--165.

\bibitem{CarlenKerce01}
{\sc E.~A. Carlen and C.~Kerce}, {\em On the cases of equality in {B}obkov's
  inequality and {G}aussian rearrangement}, Calc. Var. Partial Differential
  Equations, 13 (2001), pp.~1--18.

\bibitem{CatinoMonticelli22}
{\sc G.~Catino and D.~D. Monticelli}, {\em {Semilinear elliptic equations on
  manifolds with nonnegative Ricci curvature}}.
\newblock Accepted {\rm J. Eur. Math. Soc}, arXiv:2203.03345, 2022,
  \url{https://doi.org/10.4171/jems/1484}.

\bibitem{CatinoMonticelliRoncoroni23}
{\sc G.~Catino, D.~D. Monticelli, and A.~Roncoroni}, {\em On the critical
  {$p$}-{L}aplace equation}, Adv. Math., 433 (2023), pp.~Paper No. 109331, 38.

\bibitem{CavallettiMaggiMondino19}
{\sc F.~Cavalletti, F.~Maggi, and A.~Mondino}, {\em Quantitative isoperimetry
  \`a la {L}evy-{G}romov}, Comm. Pure Appl. Math., 72 (2019), pp.~1631--1677.

\bibitem{CavallettiManini22}
{\sc F.~Cavalletti and D.~Manini}, {\em {Rigidities of Isoperimetric inequality
  under nonnegative Ricci curvature}}.
\newblock arXiv:2207.03423, 2022 \url{https://doi.org/10.4171/jems/1532}.

\bibitem{CavallettiManini22-isoMCP}
\leavevmode\vrule height 2pt depth -1.6pt width 23pt, {\em Isoperimetric
  inequality in noncompact {$\sf {MCP}$} spaces}, Proc. Amer. Math. Soc., 150
  (2022), pp.~3537--3548.

\bibitem{CavallettiMondino17-Inv}
{\sc F.~Cavalletti and A.~Mondino}, {\em Sharp and rigid isoperimetric
  inequalities in metric-measure spaces with lower ricci curvature bounds},
  Invent. Math., 208 (2017), pp.~803--849.

\bibitem{CavallettiMondino17}
\leavevmode\vrule height 2pt depth -1.6pt width 23pt, {\em Sharp geometric and
  functional inequalities in metric measure spaces with lower {R}icci curvature
  bounds}, Geom. Topol., 21 (2017), pp.~603--645.

\bibitem{CavallettiMondinoSemola23}
{\sc F.~Cavalletti, A.~Mondino, and D.~Semola}, {\em Quantitative {O}bata's
  theorem}, Anal. PDE, 16 (2023), pp.~1389--1431.

\bibitem{Chambers19}
{\sc G.~R. Chambers}, {\em Proof of the log-convex density conjecture}, J. Eur.
  Math. Soc. (JEMS), 21 (2019), pp.~2301--2332.

\bibitem{Cheeger00}
{\sc J.~Cheeger}, {\em Differentiability of {L}ipschitz functions on metric
  measure spaces}, Geom. Funct. Anal., 9 (1999), pp.~428--517.

\bibitem{ChenLi23}
{\sc D.~Chen and H.~Li}, {\em Talenti's comparison theorem for {P}oisson
  equation and applications on {R}iemannian manifold with nonnegative {R}icci
  curvature}, J. Geom. Anal., 33 (2023), pp.~Paper No. 123, 20.

\bibitem{ChenMao2024}
{\sc R.~Chen and J.~Mao}, {\em {Several isoperimetric inequalities of Dirichlet
  and Neumann eigenvalues of the Witten-Laplacian}}.
\newblock arXiv:2403.08075, 2024.

\bibitem{Choe03}
{\sc J.~Choe}, {\em Relative isoperimetric inequality for domains outside a
  convex set}, Arch. Inequal. Appl., 1 (2003), pp.~241--250.

\bibitem{ChoeGhomiRitore07}
{\sc J.~Choe, M.~Ghomi, and M.~Ritor\'e}, {\em The relative isoperimetric
  inequality outside convex domains in {$\mathbf{R}^n$}}, Calc. Var. Partial
  Differential Equations, 29 (2007), pp.~421--429.

\bibitem{ChoeGulliver92}
{\sc J.~Choe and R.~Gulliver}, {\em Isoperimetric inequalities on minimal
  submanifolds of space forms}, Manuscripta Math., 77 (1992), pp.~169--189.

\bibitem{CianchiFusco02}
{\sc A.~Cianchi and N.~Fusco}, {\em Functions of bounded variation and
  rearrangements}, Arch. Ration. Mech. Anal., 165 (2002), pp.~1--40.

\bibitem{CintiGlaudoPratelliRosOtonSerra20}
{\sc E.~Cinti, F.~Glaudo, A.~Pratelli, X.~Ros-Oton, and J.~Serra}, {\em Sharp
  quantitative stability for isoperimetric inequalities with homogeneous
  weights}, Trans. Amer. Math. Soc., 375 (2022), pp.~1509--1550.

\bibitem{CiraoloFigallIRoncoroni20}
{\sc G.~Ciraolo, A.~Figalli, and A.~Roncoroni}, {\em Symmetry results for
  critical anisotropic {$p$}-{L}aplacian equations in convex cones}, Geom.
  Funct. Anal., 30 (2020), pp.~770--803.

\bibitem{C-ENV04}
{\sc D.~Cordero-Erausquin, B.~Nazaret, and C.~Villani}, {\em A
  mass-transportation approach to sharp {S}obolev and {G}agliardo-{N}irenberg
  inequalities}, Adv. Math., 182 (2004), pp.~307--332.

\bibitem{DGB92}
{\sc B.~Dacorogna, W.~Gangbo, and N.~Sub\'{\i}a}, {\em Sur une
  g\'{e}n\'{e}ralisation de l'in\'{e}galit\'{e} de {W}irtinger}, Ann. Inst. H.
  Poincar\'{e} C Anal. Non Lin\'{e}aire, 9 (1992), pp.~29--50.

\bibitem{DePhilippisGigli15}
{\sc G.~De~Philippis and N.~Gigli}, {\em From volume cone to metric cone in the
  nonsmooth setting}, Geom. Funct. Anal., 26 (2016), pp.~1526--1587.

\bibitem{DFV24}
{\sc N.~De~Ponti, S.~Farinelli, and I.~Y. Violo}, {\em Pleijel nodal domain
  theorem in non-smooth setting}, Transactions of the American Mathematical
  Society, Series B, 11 (2024), pp.~1138--1182.

\bibitem{DelPinoDolbeault03}
{\sc M.~Del~Pino and J.~Dolbeault}, {\em The optimal {E}uclidean
  {$L^p$}-{S}obolev logarithmic inequality}, J. Funct. Anal., 197 (2003),
  pp.~151--161.

\bibitem{DiMarinoPhD}
{\sc S.~Di~Marino}, {\em Recent advances on {B}{V} and {S}obolev spaces in
  metric measure spaces}.
\newblock \url{https://cvgmt.sns.it/paper/2568/}, PhD Thesis, 2024.

\bibitem{DiMarinoSpeight13}
{\sc S.~Di~Marino and G.~Speight}, {\em The {$p$}-weak gradient depends on
  {$p$}}, Proc. Amer. Math. Soc., 143 (2015), pp.~5239--5252.

\bibitem{Ehrhard1983}
{\sc A.~Ehrhard}, {\em Sym\'{e}trisation dans l'espace de {G}auss}, Math.
  Scand., 53 (1983), pp.~281--301.

\bibitem{eriksson2023density}
{\sc S.~Eriksson-Bique and P.~Poggi-Corradini}, {\em Density of continuous
  functions in {S}obolev spaces with applications to capacity}, Trans. Amer.
  Math. Soc. Ser. B, 11 (2024), pp.~901--944.

\bibitem{Faber23}
{\sc G.~Faber}, {\em {Beweiss dass unter allen homogenen Membranen von gleicher
  Fl\'ache und gleicher Spannung die kreisf\"ormgige den leifsten Grundton
  gibt}}.
\newblock Sitz. bayer Acad. Wiss., 169--172, 1923.

\bibitem{FeroneVolpicelli03}
{\sc A.~Ferone and R.~Volpicelli}, {\em Minimal rearrangements of {S}obolev
  functions: a new proof}, Ann. Inst. H. Poincar\'{e} C Anal. Non Lin\'{e}aire,
  20 (2003), pp.~333--339.

\bibitem{FigalliIndrei13}
{\sc A.~Figalli and E.~Indrei}, {\em A sharp stability result for the relative
  isoperimetric inequality inside convex cones}, J. Geom. Anal., 23 (2013),
  pp.~938--969.

\bibitem{FogagnoloMalchiodiMazzieri23}
{\sc M.~Fogagnolo, A.~Malchiodi, and L.~Mazzieri}, {\em A note on the critical
  {L}aplace equation and {R}icci curvature}, J. Geom. Anal., 33 (2023),
  pp.~Paper No. 178, 17.

\bibitem{FogagnoloMalchiodiMazzieri23_Correction}
\leavevmode\vrule height 2pt depth -1.6pt width 23pt, {\em Correction: {A} note
  on the critical {L}aplace equation and {R}icci curvature}, J. Geom. Anal., 34
  (2024), pp.~Paper No. 51, 1.

\bibitem{FogagnoloMazzieri22}
{\sc M.~Fogagnolo and L.~Mazzieri}, {\em Minimising hulls, p-capacity and
  isoperimetric inequality on complete {R}iemannian manifolds}, J. Funct.
  Anal., 283 (2022), p.~Paper No. 109638.

\bibitem{Frank2022}
{\sc R.~L. Frank}, {\em Rearrangement methods in the work of {E}lliott {L}ieb},
  in The physics and mathematics of {E}lliott {L}ieb---the 90th anniversary.
  {V}ol. {I}, EMS Press, Berlin, [2022] \copyright 2022, pp.~351--375.

\bibitem{FuscoCIME08}
{\sc N.~Fusco}, {\em Geometrical aspects of symmetrization}, in Calculus of
  variations and nonlinear partial differential equations, vol.~1927 of Lecture
  Notes in Math., Springer, Berlin, 2008, pp.~155--181.

\bibitem{FuscoLaManna23}
{\sc N.~Fusco and D.~A. La~Manna}, {\em Some weighted isoperimetric
  inequalities in quantitative form}, J. Funct. Anal., 285 (2023), pp.~Paper
  No. 109946, 24.

\bibitem{FuscoMorini23}
{\sc N.~Fusco and M.~Morini}, {\em Total positive curvature and the equality
  case in the relative isoperimetric inequality outside convex domains}, Calc.
  Var. Partial Differential Equations, 62 (2023), pp.~Paper No. 102, 32.

\bibitem{Gallot88}
{\sc S.~Gallot}, {\em Isoperimetric inequalities based on integral norms of
  {R}icci curvature}, Ast\'erisque,  (1988), pp.~191--216.

\bibitem{Gentil03}
{\sc I.~Gentil}, {\em The general optimal {$L^p$}-{E}uclidean logarithmic
  {S}obolev inequality by {H}amilton-{J}acobi equations}, J. Funct. Anal., 202
  (2003), pp.~591--599.

\bibitem{Gigli13_splitting}
{\sc N.~Gigli}, {\em {The splitting theorem in non-smooth context}}.
\newblock arXiv:1302.5555, 2013.

\bibitem{Gigli23_working}
\leavevmode\vrule height 2pt depth -1.6pt width 23pt, {\em {De Giorgi and
  Gromov working together}}.
\newblock arXiv:2306.14604, 2023.

\bibitem{GigliHan14}
{\sc N.~Gigli and B.~Han}, {\em Independence on $p$ of weak upper gradients on
  $\rm {R}{C}{D}$ spaces}, Journal of Functional Analysis, 271 (2014).

\bibitem{GigliNobili21}
{\sc N.~Gigli and F.~Nobili}, {\em A first-order condition for the independence
  on {$p$} of weak gradients}, J. Funct. Anal., 283 (2022), p.~Paper No.
  109686.

\bibitem{GP20}
{\sc N.~Gigli and E.~Pasqualetto}, {\em Lectures on Nonsmooth Differential
  Geometry}, SISSA Springer Series 2, 2020.

\bibitem{Gromov07}
{\sc M.~Gromov}, {\em Metric structures for {R}iemannian and non-{R}iemannian
  spaces}, Modern Birkh\"auser Classics, Birkh\"auser Boston Inc., Boston, MA,
  english~ed., 2007.
\newblock Based on the 1981 French original, With appendices by M. Katz, P.
  Pansu and S. Semmes, Translated from the French by Sean Michael Bates.

\bibitem{GunesMondino23}
{\sc M.~A. Gunes and A.~Mondino}, {\em A reverse {H}\"older inequality for
  first eigenfunctions of the {D}irichlet {L}aplacian on {${\rm RCD}(K,N)$}
  spaces}, Proc. Amer. Math. Soc., 151 (2023), pp.~295--311.

\bibitem{HamelNadirashviliRuss2011}
{\sc F.~Hamel, N.~Nadirashvili, and E.~Russ}, {\em Rearrangement inequalities
  and applications to isoperimetric problems for eigenvalues}, Ann. of Math.
  (2), 174 (2011), pp.~647--755.

\bibitem{HansenNadirashvili94}
{\sc W.~Hansen and N.~Nadirashvili}, {\em Isoperimetric inequalities in
  potential theory}, in Proceedings from the {I}nternational {C}onference on
  {P}otential {T}heory ({A}mersfoort, 1991), vol.~3, 1994, pp.~1--14.

\bibitem{hebey99}
{\sc E.~Hebey}, {\em Nonlinear analysis on manifolds: {S}obolev spaces and
  inequalities}, vol.~5 of Courant Lecture Notes in Mathematics, New York
  University, Courant Institute of Mathematical Sciences, New York; American
  Mathematical Society, Providence, RI, 1999.

\bibitem{HeinonenKoskelaShanmugalingam55}
{\sc J.~Heinonen, P.~Koskela, N.~Shanmugalingam, and J.~T. Tyson}, {\em Sobolev
  spaces on metric measure spaces}, vol.~27 of New Mathematical Monographs,
  Cambridge University Press, Cambridge, 2015.
\newblock An approach based on upper gradients.

\bibitem{Johne}
{\sc F.~Johne}, {\em Sobolev inequalities on manifolds with nonnegative
  {B}akry-\'{E}mery {R}icci curvature}.
\newblock arXiv:2103.08496, (2021).

\bibitem{Keith03}
{\sc S.~Keith}, {\em Modulus and the {P}oincar\'e{} inequality on metric
  measure spaces}, Math. Z., 245 (2003), pp.~255--292.

\bibitem{ciao}
{\sc S.~Kesavan}, {\em Symmetrization \& applications}, vol.~3 of Series in
  Analysis, World Scientific Publishing Co. Pte. Ltd., Hackensack, NJ, 2006.

\bibitem{Ketterer13}
{\sc C.~Ketterer}, {\em Cones over metric measure spaces and the maximal
  diameter theorem}, J. Math. Pures Appl. (9), 103 (2015), pp.~1228--1275.

\bibitem{Ketterer15}
\leavevmode\vrule height 2pt depth -1.6pt width 23pt, {\em Obata's rigidity
  theorem for metric measure spaces}, Anal. Geom. Metr. Spaces, 3 (2015),
  pp.~278--295.

\bibitem{Kim00}
{\sc I.~Kim}, {\em An optimal relative isoperimetric inequality in concave
  cylindrical domains in {$\bf R^n$}}, J. Inequal. Appl., 5 (2000),
  pp.~97--102.

\bibitem{KinnunenKorteShanmugalingamTuominen12}
{\sc J.~Kinnunen, R.~Korte, N.~Shanmugalingam, and H.~Tuominen}, {\em A
  characterization of {N}ewtonian functions with zero boundary values}, Calc.
  Var. Partial Differential Equations, 43 (2012), pp.~507--528.

\bibitem{Krahn26}
{\sc E.~Krahn}, {\em {\"Uber Minimaleigenschaften der Kugel in drei und mehr
  Dimensionen}}.
\newblock Acta Comm. Univ. Tartu (Dorpat), A9 (1926) pp. 1--44 (English
  transl.: \"U. Lumiste and J. Peetre (eds.), Edgar Krahn, 1894--1961, A
  Centenary Volume, IOS Press, 1994, Chap. 6, pp. 139-174.

\bibitem{Krahn25}
{\sc E.~Krahn}, {\em \"{U}ber eine von {R}ayleigh formulierte
  {M}inimaleigenschaft des {K}reises}, Math. Ann., 94 (1925), pp.~97--100.

\bibitem{Kristaly22-GAFA}
{\sc A.~Krist\'aly}, {\em Lord {R}ayleigh's conjecture for vibrating clamped
  plates in positively curved spaces}, Geom. Funct. Anal., 32 (2022),
  pp.~881--937.

\bibitem{Kristaly23}
\leavevmode\vrule height 2pt depth -1.6pt width 23pt, {\em Sharp {S}obolev
  inequalities on noncompact {R}iemannian manifolds with {${\rm Ric}\ge 0$} via
  optimal transport theory}, Calc. Var. Partial Differential Equations, 63
  (2024), p.~Paper No. 200.

\bibitem{KristalyMondino24}
{\sc A.~Krist{\'a}ly and A.~Mondino}, {\em Principal frequency of clamped
  plates on {{\(\mathsf{RCD}(0,N)\)}} spaces: sharpness, rigidity, and
  stability}, Proc. Lond. Math. Soc. (3), 131 (2025), p.~36.
\newblock Id/No e70079.

\bibitem{Krummel17}
{\sc B.~Krummel}, {\em {Higher codimension relative isoperimetric inequality
  outside a convex set}}.
\newblock arXiv:1710.04821, 2017.

\bibitem{Lieb83}
{\sc E.~H. Lieb}, {\em Sharp constants in the {H}ardy-{L}ittlewood-{S}obolev
  and related inequalities}, Ann. of Math. (2), 118 (1983), pp.~349--374.

\bibitem{LiebLoss1997}
{\sc E.~H. Lieb and M.~Loss}, {\em Analysis}, vol.~14 of Graduate Studies in
  Mathematics, American Mathematical Society, Providence, RI, 1997.

\bibitem{LionsPacella90}
{\sc P.-L. Lions and F.~Pacella}, {\em Isoperimetric inequalities for convex
  cones}, Proc. Amer. Math. Soc., 109 (1990), pp.~477--485.

\bibitem{LionsPacellaTricarico1988}
{\sc P.-L. Lions, F.~Pacella, and M.~Tricarico}, {\em Best constants in
  {S}obolev inequalities for functions vanishing on some part of the boundary
  and related questions}, Indiana Univ. Math. J., 37 (1988), pp.~301--324.

\bibitem{LiuWangWeng23}
{\sc L.~Liu, G.~Wang, and L.~Weng}, {\em The relative isoperimetric inequality
  for minimal submanifolds with free boundary in the {E}uclidean space}, J.
  Funct. Anal., 285 (2023), pp.~Paper No. 109945, 22.

\bibitem{LucicPasqualettoRajala21}
{\sc D.~Lu\v{c}i\'c, E.~Pasqualetto, and T.~Rajala}, {\em Characterisation of
  upper gradients on the weighted {E}uclidean space and applications}, Ann.
  Mat. Pura Appl. (4), 200 (2021), pp.~2473--2513.

\bibitem{Maggi12_Book}
{\sc F.~Maggi}, {\em Sets of finite perimeter and geometric variational
  problems}, vol.~135 of Cambridge Studies in Advanced Mathematics, Cambridge
  University Press, Cambridge, 2012.
\newblock An introduction to geometric measure theory.

\bibitem{OhtaMai21}
{\sc C.~H. Mai and S.-i. Ohta}, {\em Quantitative estimates for the
  {B}akry-{L}edoux isoperimetric inequality}, Comment. Math. Helv., 96 (2021),
  pp.~693--739.

\bibitem{Martio16-2}
{\sc O.~Martio}, {\em The space of functions of bounded variation on curves in
  metric measure spaces}, Conform. Geom. Dyn., 20 (2016), pp.~81--96.

\bibitem{Miranda03}
{\sc M.~Miranda}, {\em Functions of bounded variation on ``good'' metric
  spaces}, Journal de Math\'{e}matiques Pures et Appliqu\'{e}es, 82 (2003),
  pp.~975--1004.

\bibitem{MondinoSemola20}
{\sc A.~Mondino and D.~Semola}, {\em Polya-{S}zego inequality and {D}irichlet
  {$p$}-spectral gap for non-smooth spaces with {R}icci curvature bounded
  below}, J. Math. Pures Appl. (9), 137 (2020), pp.~238--274.

\bibitem{mondinovedovato21}
{\sc A.~Mondino and M.~Vedovato}, {\em A {T}alenti-type comparison theorem for
  {${\rm RCD}(K,N)$} spaces and applications}, Calc. Var. Partial Differential
  Equations, 60 (2021), pp.~Paper No. 157, 43.

\bibitem{NobiliPasqualettoSchultz21}
{\sc F.~Nobili, E.~Pasqualetto, and T.~Schultz}, {\em On master test plans for
  the space of {BV} functions}, Adv. Calc. Var., 16 (2023), pp.~1061--1092.

\bibitem{NobiliViolo21}
{\sc F.~Nobili and I.~Y. Violo}, {\em Rigidity and almost rigidity of {S}obolev
  inequalities on compact spaces with lower {R}icci curvature bounds}, Calc.
  Var. Partial Differential Equations, 61 (2022), p.~Paper No. 180.

\bibitem{NobiliViolo24}
\leavevmode\vrule height 2pt depth -1.6pt width 23pt, {\em Stability of
  {S}obolev inequalities on {R}iemannian manifolds with {R}icci curvature lower
  bounds}, Adv. Math., 440 (2024), p.~Paper No. 109521.

\bibitem{OhtaTakatsu20}
{\sc S.-i. Ohta and A.~Takatsu}, {\em Equality in the logarithmic {S}obolev
  inequality}, Manuscripta Math., 162 (2020), pp.~271--282.

\bibitem{PasqualettoRajala25}
{\sc E.~Pasqualetto and T.~Rajala}, {\em A note on {Laplacian} bounds,
  deformation properties and isoperimetric sets in metric measure spaces}.
\newblock arXiv:2503.14132, 2025.

\bibitem{PoliaSzego51}
{\sc G.~P\'olya and G.~Szeg\H{o}}, {\em Isoperimetric Inequalities in
  Mathematical Physics. (AM-27)}, Princeton University Press, 1951.

\bibitem{Rajala12-2}
{\sc T.~Rajala}, {\em Interpolated measures with bounded density in metric
  spaces satisfying the curvature-dimension conditions of {S}turm}, J. Funct.
  Anal., 263 (2012), pp.~896--924.

\bibitem{Raj20}
\leavevmode\vrule height 2pt depth -1.6pt width 23pt, {\em Approximation by
  uniform domains in doubling quasiconvex metric spaces}, Complex Anal.
  Synerg., 7 (2021), p.~5.
\newblock Paper No. 4.

\bibitem{RajalaSturm12}
{\sc T.~Rajala and K.-T. Sturm}, {\em Non-branching geodesics and optimal maps
  in strong {$CD(K,\infty)$}-spaces}, Calc. Var. Partial Differential
  Equations, 50 (2014), pp.~831--846.

\bibitem{Ros05}
{\sc A.~Ros}, {\em The isoperimetric problem}, in Global theory of minimal
  surfaces, vol.~2 of Clay Math. Proc., Amer. Math. Soc., Providence, RI, 2005,
  pp.~175--209.

\bibitem{RosalesCaneteBayleMorgan08}
{\sc C.~Rosales, A.~Ca\~{n}ete, V.~Bayle, and F.~Morgan}, {\em On the
  isoperimetric problem in {E}uclidean space with density}, Calc. Var. Partial
  Differential Equations, 31 (2008), pp.~27--46.

\bibitem{Keomkyo07}
{\sc K.~Seo}, {\em Geometric inequalities outside a convex set in a
  {R}iemannian manifold}, J. Math. Kyoto Univ., 47 (2007), pp.~657--664.

\bibitem{SetoWei2017}
{\sc S.~Seto and G.~Wei}, {\em First eigenvalue of the {$p$}-{L}aplacian under
  integral curvature condition}, Nonlinear Anal., 163 (2017), pp.~60--70.

\bibitem{Shan00}
{\sc N.~Shanmugalingam}, {\em Newtonian spaces: an extension of {S}obolev
  spaces to metric measure spaces}, Rev. Mat. Iberoamericana, 16 (2000),
  pp.~243--279.

\bibitem{Sturm06I}
{\sc K.-T. Sturm}, {\em On the geometry of metric measure spaces. {I}}, Acta
  Math., 196 (2006), pp.~65--131.

\bibitem{Sturm06II}
\leavevmode\vrule height 2pt depth -1.6pt width 23pt, {\em On the geometry of
  metric measure spaces. {II}}, Acta Math., 196 (2006), pp.~133--177.

\bibitem{Sturm24_Survey}
\leavevmode\vrule height 2pt depth -1.6pt width 23pt, {\em {Metric measure
  spaces and synthetic Ricci bounds: fundamental concepts and recent
  developments}}, in European Congress of Mathematics, 2023, pp.~125--159.

\bibitem{SudakovCirel1974}
{\sc V.~N. Sudakov and B.~S. Tsirelson}, {\em Extremal properties of
  half-spaces for spherically invariant measures}, Zap. Nau\v{c}n. Sem.
  Leningrad. Otdel. Mat. Inst. Steklov. (LOMI), 41 (1974), pp.~14--24, 165.
\newblock Problems in the theory of probability distributions, II.

\bibitem{Talenti76}
{\sc G.~Talenti}, {\em Best constant in {S}obolev inequality}, Ann. Mat. Pura
  Appl. (4), 110 (1976), pp.~353--372.

\bibitem{Talenti2016}
\leavevmode\vrule height 2pt depth -1.6pt width 23pt, {\em The art of
  rearranging}, Milan J. Math., 84 (2016), pp.~105--157.

\bibitem{Villani2016}
{\sc C.~Villani}, {\em Synthetic theory of ricci curvature bounds}, Japanese
  Journal of Mathematics, 11 (2016), pp.~219--263.

\bibitem{LiWang2016}
{\sc Y.-Z. Wang and H.-Q. Li}, {\em Lower bound estimates for the first
  eigenvalue of the weighted {$p$}-{L}aplacian on smooth metric measure
  spaces}, Differential Geom. Appl., 45 (2016), pp.~23--42.

\bibitem{Wu24}
{\sc W.~Wu}, {\em A {T}alenti-type comparison theorem for the {$p$}-{L}aplacian
  on {$RCD ( K, N )$} spaces and some applications}, Nonlinear Anal., 248
  (2024), p.~Paper No. 113631.

\end{thebibliography}
\end{document}